\documentclass[a4paper,10pt,twoside,english]{scrartcl}
\usepackage{marginnote}
\usepackage[utf8x]{inputenc}
\usepackage{amsmath,amssymb,amsthm}
\usepackage{mathrsfs}  
\usepackage[unicode=true]{hyperref}
\usepackage[all]{hypcap}
\usepackage[T1]{fontenc}
\usepackage{lmodern}
\usepackage{graphicx}
\usepackage{enumerate}
\usepackage{setspace}
\usepackage{xspace} 
\usepackage{blkarray}
\usepackage[subnum]{cases}

\newcommand{\titel}{Local law for random Gram matrices}
\title{\titel} 
\author{
Johannes Alt\footnote{Partially funded by ERC Advanced Grant RANMAT No. 338804.} \addtocounter{footnote}{-1}\addtocounter{Hfootnote}{-1}\\
{\small \begin{tabular}{c}{IST Austria}\\{jalt@ist.ac.at} \end{tabular}} 
\and László Erd\H{o}s\footnotemark \addtocounter{footnote}{-1}\addtocounter{Hfootnote}{-1} 
\\{\small \begin{tabular}{c} IST Austria\\ lerdos@ist.ac.at\end{tabular}} 
\and Torben Krüger\footnotemark\\
{\small \begin{tabular}{c} IST Austria\\ torben.krueger@ist.ac.at\end{tabular}}
}
\date{}

\hypersetup{
	pdftitle={\titel},
	pdfauthor={Johannes Alt, László Erd\H{o}s, Torben Krüger}%,
}

\numberwithin{equation}{section}

\newcommand{\R}{\mathbb{R}}  % The real numbers.
\ifdefined\C\renewcommand{\C}{\mathbb{C}}\else\newcommand{\C}{\mathbb{C}}\fi % Complex numbers
\renewcommand{\Im}{\mathrm{Im}\,} %imaginary part of a complex number
\renewcommand{\Re}{\mathrm{Re}\,} %real part of a complex number
\renewcommand{\i}{\mathrm{i}} %imaginary part of a complex number
\newcommand{\N}{\mathbb{N}}  % Positive integers.	
\newcommand{\E}{\mathbb{E}}  % expected value of random variable	
\renewcommand{\P}{\mathbb{P}}  % probability measure
\newcommand{\di}{\mathrm{d}} % differential
\newcommand{\eps}{\varepsilon} % "correct" epsilon
\newcommand*{\defeq}{\mathrel{\vcenter{\baselineskip0.5ex \lineskiplimit0pt\hbox{\scriptsize.}\hbox{\scriptsize.}}}=}
\newcommand*{\defqe}{=\mathrel{\vcenter{\baselineskip0.5ex \lineskiplimit0pt\hbox{\scriptsize.}\hbox{\scriptsize.}}}}

\DeclareMathOperator{\supp}{supp}

\newcommand{\gon}{\ensuremath{g_{1}}}
\newcommand{\gtw}{\ensuremath{g_{2}}}
\newcommand{\mon}{\ensuremath{m_{1}}}
\newcommand{\mtw}{\ensuremath{m_{2}}}

\newcommand{\gone}[1]{\ensuremath{g_{1,#1}}}
\newcommand{\gtwo}[1]{\ensuremath{g_{2,#1}}}

\newcommand{\done}[1]{\ensuremath{d_{1,#1}}}
\newcommand{\dtwo}[1]{\ensuremath{d_{2,#1}}}

\newcommand{\df}{\mathfrak d}
\newcommand{\vf}{\mathfrak v}
\newcommand{\ff}{\mathfrak f}
\newcommand{\gf}{\mathfrak g}
\newcommand{\mf}{\mathfrak m}
\newcommand{\Bf}{\mathfrak B}
\newcommand{\Ff}{\mathfrak F}
\newcommand{\Mf}{\mathfrak M}
\newcommand{\Sf}{\mathfrak S}
\renewcommand{\sf}{\sigma}
\newcommand{\Xf}{H}

\newcommand{\Hb}{\mathbb H}

\newcommand{\normtwo}[1]{\lVert #1 \rVert_{2}}
\newcommand{\normtwoa}[1]{\left\lVert #1 \right\rVert_{2}}
\newcommand{\normtwoinf}[1]{\lVert #1 \rVert_{2\to\infty}}
\newcommand{\norminf}[1]{\lVert #1 \rVert_{\infty}}
\newcommand{\norminfa}[1]{\left\lVert #1 \right\rVert_{\infty}}

\DeclareMathOperator{\Gap}{Gap}

\newtheoremstyle{test}% name
  {}%      Space above, empty = `usual value'
  {}%      Space below
  {\itshape}% Body font
  {}%         Indent amount (empty = no indent, \parindent = para indent)
  {\bfseries}% Thm head font
  {.}%        Punctuation after thm head 
  { }% Space after thm head: \newline = linebreak
  {}%         Thm head spec

\theoremstyle{test}
\newtheorem{defi}{Definition}[section]

\newtheorem{rem}[defi]{Remark}

\newtheorem{thm}[defi]{Theorem}
\newtheorem{lem}[defi]{Lemma}
\newtheorem{coro}[defi]{Corollary}
\newtheorem{pro}[defi]{Proposition}
\newtheorem*{rem*}{Remark}   %no numbering
\newtheorem*{ex*}{Example}   %no numbering
\newtheorem*{pro*}{Proposition} %no numbering
\newtheorem*{def*}{Definition}
\newtheorem*{coro*}{Corollary}
\newtheorem*{thm*}{Theorem}

\theoremstyle{test}
    \newtheorem{theorem}[defi]{Theorem}
    \newtheorem{proposition}[defi]{Proposition}
    \newtheorem{corollary}[defi]{Corollary}
    \newtheorem{lemma}[defi]{Lemma}
    \newtheorem{definition}[defi]{Definition}% howto make rm-style text inside definitions
    
    \newtheorem{convention}[defi]{Convention}
    \newtheorem{remark}[defi]{Remark}

\newcommand{\bels}[2] {
        \begin{equation} \label{#1} \begin{split} 
                #2 
        \end{split} \end{equation}
        }
\newcommand{\bes}[1]{
        \begin{equation*}  \begin{split} 
                #1 
        \end{split} \end{equation*}
        }

\newcommand{\rr}{\mathrm} %upright
\renewcommand{\cal}{\mathcal} 
 
\renewcommand{\frak}{\mathfrak} 
\newcommand{\wh}{\widehat}
\newcommand{\wt}{\widetilde}

\renewcommand{\P}{\mathbb{P}}

\newcommand{\ee}{\mathrm{e}} %\newcommand{\me}{\mathrm{e}}
\newcommand{\ii}{\mathrm{i}} %\newcommand{\mi}{\mathrm{i}}
\newcommand{\dd}{\mathrm{d}}

\newcommand{\pb}[1]{\bigl({#1}\bigr)}

\newcommand{\abs}[1]{\lvert #1 \rvert}

\newcommand{\absa}[1]{\left\lvert #1 \right\rvert}

\newcommand{\norm}[1]{\lVert #1 \rVert}

\newcommand{\normbb}[1]{\bigg\lVert #1 \bigg\rVert}

\newcommand{\norma}[1]{\left\lVert #1 \right\rVert}

\newcommand{\avg}[1]{\langle #1 \rangle}

\newcommand{\avga}[1]{\left\langle #1 \right\rangle}

\newcommand{\scalar}[2]{\langle{#1} \mspace{2mu}, {#2}\rangle}

\newcommand{\scalara}[2]{\left\langle{#1} \,\mspace{2mu},\, {#2}\right\rangle}

\DeclareMathOperator{\diag}{diag}
\DeclareMathOperator{\tr}{Tr}

\DeclareMathOperator{\im}{Im}

\DeclareMathOperator{\dist} {dist}                
\DeclareMathOperator*{\spec}{\sigma}						%Spectrum

\newcommand{\1} {\mspace{1 mu}}
\newcommand{\2} {\mspace{2 mu}}
\newcommand{\msp}[1] {\mspace{#1 mu}}

\newcommand{\dimhard}{(E1)\xspace}
\newcommand{\zerohard}{(F1)\xspace}
\newcommand{\dimrect}{(E2)\xspace}
\newcommand{\zerorect}{(F2)\xspace}

\begin{document}

\maketitle
\vspace*{-1.4cm}

\begin{abstract}
We prove a local law in the bulk of the spectrum for random Gram matrices $XX^*$, a generalization of sample covariance matrices, where $X$ is a large
matrix with independent, centered entries with arbitrary variances. The limiting eigenvalue density that generalizes the Marchenko-Pastur law is determined by solving a system of 
nonlinear equations.  Our entrywise and averaged local laws are on the optimal scale with the optimal error bounds. They hold both
in the square case (hard edge) and in the properly rectangular case (soft edge). In the latter case we also establish a macroscopic gap away from zero in the spectrum of $XX^*$.
\end{abstract}

\noindent \emph{Keywords:} Capacity of MIMO channels, Marchenko-Pastur law, hard edge, soft edge, general variance profile\\
\textbf{AMS Subject Classification:} 60B20, 15B52

\section{Introduction}

Random matrices were introduced in pioneering works by Wishart \cite{WISHART1928} and Wigner \cite{Wigner1955} for applications in mathematical statistics and nuclear physics, respectively. 
Wigner argued that the energy level statistics of large atomic nuclei could be described by the eigenvalues of a large 
\emph{Wigner matrix}, i.e., a hermitian matrix $H=(h_{ij})_{i,j=1}^N$
with centered, identically distributed and independent entries (up to the symmetry constraint $H = H^*$).
He proved 
that the empirical spectral measure (or density of states) converges to the \emph{semicircle law} as the dimension of the matrix $N$
goes to infinity. 
Moreover, he postulated that the statistics of the gaps between 
consecutive eigenvalues depend only on the symmetry type of the matrix and 
are independent of the distribution of the entries in the large $N$ limit.
 The precise formulation of this
 phenomenon is called the \emph{Wigner-Dyson-Mehta universality conjecture}, see  \cite{mehta2004random}.

Historically, the second main class of random matrices is the one of \emph{sample covariance matrices}. These are
of the form  $XX^*$ where  $X$ is a $p \times n$ matrix with  centered, identically distributed
 independent entries. In statistics context, its columns contain $n$ samples of a $p$-dimensional  data vector. 
In the regime of high dimensional
  data,   i.e., in the limit when $n, p\to \infty$ in such a way that
 the ratio $p/n$ converges to a constant,   the 
empirical spectral measure of $XX^*$ was explicitly  identified by
Marchenko and Pastur \cite{MarcenkoPastur1967}.
Random matrices of the form $XX^*$ also appear in 
the theory of wireless communication; 
the spectral density of these matrices is used to compute the transmission capacity 
of a  Multiple Input Multiple Output (MIMO)  channel. 
This fundamental connection between random matrix theory and wireless communication was established by Telatar \cite{ETT:ETT4460100604} and 
Foschini \cite{Foschini1998,BLTJ:BLTJ2015} (see also \cite{TulinoVerdu} for a review). 
In this model, the element $x_{ij}$ of the \emph{channel matrix} $X$  represents the transmission coefficient
from the $j$-th  transmitter to the $i$-th receiver antenna. The received signal is given by the linear
relation $y= Xs + w$, where $s$  is the input signal and $w$ is a Gaussian noise with variance $\sigma^2$.
 In case of i.i.d. Gaussian input signals, the  channel capacity is given by
\begin{equation}\label{capacity}
  \text{Cap}=   \frac{1}{p} \log\det \Big( I + \sigma^{-2} XX^*\Big).
\end{equation}

The assumption in these models that the matrix elements of $H$ or $X$ have identical distribution is a simplification
that  does not hold in many applications. In Wigner's model, the  matrix elements  $h_{ij}$ represent random
quantum transition rates  between physical states labelled by $i$ and $j$ and their distribution
may depend on these states. Analogously, the transmission coefficients in $X$ may have different distributions.  
This leads to the natural generalizations of both classes of random matrices by allowing
for  general variances, $s_{ij} \defeq \E\abs{h_{ij}}^2$ and $s_{ij} \defeq 
\E \abs{x_{ij}}^2$ , respectively.  We will  still assume the independence of the matrix elements
and their zero expectation. Under  mild conditions on the variance matrix $S= (s_{ij})$, the limiting spectral measure   
depends only on the second moments, i.e., on $S$, and otherwise it is independent of the
fine details of the distributions of the matrix elements. However, in general there is no explicit formula for the limiting spectral measure.
In fact, the only known way to find it in the general case is to solve 
a system of nonlinear deterministic equations, known as the Dyson (or Schwinger-Dyson) equation in this context,
see \cite{Berezin1973, Wegner1979, Girko2001, KhorunzhyPastur1994}.

For the generalization of Wigner's model, the Dyson equation is a system of equations of the form
\begin{equation}
\label{eq:QVE}
-\frac{1}{m_i(z)} = z + \sum_{j=1}^Ns_{ij}m_j(z), \quad \text{ for } i=1, \ldots, N, \qquad z \in \Hb,
\end{equation}
where $z$ is a complex parameter in the upper half plane $\Hb \defeq \{ z\in \C\; : \; \im z>0\}$.
The average $\avg{m(z)}=\frac{1}{N}\sum_i m_i(z)$ in the large $N$ limit gives
the Stieltjes transform of the limiting spectral density, which then can be computed
by inverting the Stieltjes transform. 
In fact, $m_i(z)$ approximates individual diagonal matrix elements $G_{ii}(z)$ of the resolvent
$G(z) = (H-z)^{-1}$, thus the solution of \eqref{eq:QVE} gives much more information on $H$
than merely the spectral density.
In the case when $S$ is a stochastic matrix, i.e., $\sum_j s_{ij}=1$ for every $i$, the solution  $m_i(z)$ to 
\eqref{eq:QVE} is independent of $i$ and the density is still the semicircle law. The corresponding
generalized Wigner matrix was introduced in \cite{EYYbulk} and the optimal local law was proven in  \cite{EYYBern, EYYrig}.
 For the general case, a detailed analysis of \eqref{eq:QVE} and the shapes of the possible density profiles was given
in \cite{AjankiQVE, AjankiCPAM} with the optimal local law  in   \cite{Ajankirandommatrix}.

Considering the $XX^*$
 model with a general variance matrix for $X$, we note that
in statistical applications the entries of  $X$ within the same row still have the same variance, i.e., 
$s_{ik} = s_{il}$ for all $i$ and all $k,l$. However, beyond statistics, for example modeling the capacity
of MIMO channels, applications require
to analyze the spectrum of $XX^*$ with a completely general variance profile for $X$ \cite{hachem2007,wirelesscommunication}.
These are called  \emph{random Gram matrices}, see e.g. \cite{Girko2001, Hachem2008IEEE}. The corresponding Dyson equation is (see \cite{Girko2001, wirelesscommunication,TulinoVerdu} and references therein)
\begin{equation} \label{eq:self_consistent_Stieltjes_Gram}
 - \frac{1}{m_i(\zeta)} = \zeta - \sum_{k=1}^n s_{ik} \frac{1}{1 + \sum_{j=1}^n s_{jk} m_j(\zeta)}, \quad \text{ for } i=1, \ldots, p, 
 \qquad \zeta \in \Hb.
\end{equation}
We have $m_i(\zeta) \approx (XX^*-\zeta)^{-1}_{ii}$ and the average of $m_i(\zeta)$ 
yields the Stieltjes transform of the spectral density exactly as in case of  the Wigner-type ensembles.
In fact, there is a direct link between these two models: 
Girko's symmetrization trick reduces \eqref{eq:self_consistent_Stieltjes_Gram} to studying \eqref{eq:QVE} 
on $\C^{N}$ with $N=n+p$, where $S$ and $H$ are replaced by 
\begin{equation}\label{Girko}
 \Sf = \begin{pmatrix} 0 & S \\ S^t & 0 \end{pmatrix}, \quad H = \begin{pmatrix} 0 & X \\ X^* & 0 \end{pmatrix},
 \end{equation}
respectively, and $z^2=\zeta$.

The limiting spectral density, also called the \emph{global law}, is typically the first question one asks about random matrix ensembles.
It can be  strengthened by considering its \emph{local} versions. In most cases, it is expected that
the deterministic density computed via the Dyson equation accurately  describes the eigenvalue
density down to the smallest possible scale which is slightly above the typical eigenvalue spacing
(we choose the  standard  normalization such that the spacing in the bulk spectrum is of order $1/N$).
This requires to understand the  trace of the resolvent $G(z)$ at a spectral parameter very close to the real axis, down to the scales
$\im z \gg 1/N$. 
Additionally, \emph{entry-wise local laws} and \emph{isotropic local laws}, i.e.,
controlling individual matrix elements $G_{ij}(z)$ and bilinear forms $\langle v, G(z) w\rangle$, carry important information on
eigenvectors and allow for perturbation theory. 
Moreover,  effective error bounds on the speed of convergence as $N$ goes to infinity are also of great interest.

Local laws have also played a crucial role in the recent proofs of the Wigner-Dyson-Mehta conjecture.
The three-step approach, developed in a series of works by Erd{\H o}s, Schlein, Yau and Yin
\cite{ErdosSchleinYau2011,erdoes_relaxation_flow_2012} (see \cite{ErdoesYau2012} for a review),
was based on establishing the local law as the first step. Similar input was necessary in the alternative
approach by Tao and Vu in \cite{TaoVu2011_Acta,tao2012}.

In this paper, we establish the optimal local law for random Gram matrices with  a general variance matrix $S$
in the bulk spectrum;  edge analysis and local spectral universality is deferred to a forthcoming work.
We show that the empirical spectral measure of $XX^*$ can be approximated by a deterministic measure $\nu$ 
on $\R$ with a continuous density away from zero and possibly a point mass at zero.  The convergence holds 
locally down to the smallest possible scale and with an optimal speed of order $1/N$.
In the special case when $X$ is a square matrix, $n=p$, the  measure $\nu$ does not have a point mass 
but the density has an inverse square-root singularity at zero (called the \emph{hard edge case}).
In the \emph{soft edge case}, $n\ne p$, the continuous part of $\nu$ is supported away from zero
and it has a point mass of size $1-n/p$ at zero if $p>n$. All these features are well-known for the 
classical Marchenko-Pastur setup, but in the general case we need to demonstrate them without 
any explicit formula.

We now summarize some  previous related results on  Gram matrices. If each entry of $X$ has the same variance,
local Marchenko-Pastur laws have  first been proved in 
 \cite{erdoes_relaxation_flow_2012,pillai2014} for the soft edge case; and in \cite{Cacciapuoti2013,Bourgarde2014} for the hard edge case. 
The isotropic local law was given  in \cite{EJP3054}.  
Relaxing the assumption of identical variances  to a doubly stochastic variance matrix of $X$,
the optimal  local Marchenko-Pastur law has been established in \cite{ECP3121} for the hard edge case.
Sample correlation matrices in the soft edge case were considered in \cite{bao2012}.

Motivated by the linear model in multivariate statistics  and to depart from the identical distribution,
random matrices of the form $TZ Z^*T^*$  have been extensively
 studied where $T$ is a deterministic  matrix and the entries of $Z$ 
 are independent, centered and have unit variance.  If $T$ is diagonal, then 
they are generalizations of sample covariance matrices as $TZ Z^*T^*=XX^*$ and the elements of
 $X = TZ$ are also independent.
 With this definition, all entries within one row of $X$ have the same variance since 
$s_{ij}=\E \abs{x_{ij}}^2 = (TT^*)_{ii}$, i.e., it is a special case of our random Gram matrix.
In this case the Dyson system of equations  \eqref{eq:self_consistent_Stieltjes_Gram} 
can be reduced to a single equation for the average $\langle m(z)\rangle$, i.e.,
 the limiting density can still be obtained from a \emph{scalar self-consistent equation}.
This is even true for  matrices of the form  $XX^*$ with $X = TZ\widetilde T$, where both $T$ and $\widetilde T$
 are deterministic, investigated for example in \cite{couillet2014}. 
For general $T$ the elements of $X=TZ$ are not independent, so general  sample covariance
matrices are typically not Gram matrices. 
The global law for  $TZ Z^*T^*$ has been proven by Silverstein and Bai in \cite{SILVERSTEIN1995175}. Knowles and Yin showed optimal local laws for a general deterministic  $T$ in \cite{Knowles2014}. 

Finally, we review some existing results on  random Gram matrices with general variance $S$,
when  \eqref{eq:self_consistent_Stieltjes_Gram}  cannot be reduced to a simpler scalar equation.  
The global law, even with nonzero expectation of $X$, has been determined by Girko \cite{Girko2001} via \eqref{eq:self_consistent_Stieltjes_Gram}
who also established  the existence and uniqueness of the solution to \eqref{eq:self_consistent_Stieltjes_Gram}.
More recently, motivated by the theory of wireless communication, Hachem, Loubaton and Najim  initiated a rigorous study of the 
asympotic behaviour of the channel capacity \eqref{capacity} with a general variance matrix $S$   \cite{Hachem2006649,hachem2007},
 This required to establish the global law under more general conditions than Girko;  see also \cite{Hachem2008IEEE} for a review from the point of view of applications. Hachem \emph{et. al.} have also 
established Gaussian fluctuations of the channel capacity \eqref{capacity} around a deterministic limit in \cite{hachem2008}
 for the centered case. For a nonzero expectation of $X$,  a similar result  was obtained in \cite{HachemCLT}, where $S$
  was restricted to a product form. 
Very recently in  \cite{GramMatrix}, a special  $k$-fold clustered  matrix $XX^*$ was considered,
where the samples came from $k$ different clusters with possibly different distributions. 
The Dyson equation  in this case reduces to a  system of $k$  equations.
In an information-plus-noise model of the form $(R + X)(R+ X)^*$,
the effect of adding a noise matrix to $X$ with identically distributed entries was studied knowing the limiting density of $RR^*$ \cite{DOZIER2007678}.

In all previous works concerning general Gram matrices, the spectral parameter $z$ was fixed,
in particular $\Im z$ had a positive lower bound independent of the dimension of the matrix.
Technically, this positive imaginary part provided the necessary contraction factor 
in the fixed point argument that led to the existence, uniqueness and stability of the solution
to the Dyson equation, \eqref{eq:self_consistent_Stieltjes_Gram}. For local laws down to the optimal scales $\Im z\gg 1/N$,
the regularizing effect of $\Im z$  is too weak. In the bulk spectrum $\Im z$ is effectively replaced with
the local density, i.e., with the average imaginary part $\Im \langle m(z) \rangle$. The main difficulty with this heuristics is its apparent
circularity: the yet unknown solution  itself is necessary for regularizing the equation. 
This problem is present in all existing proofs of any local law.  This circularity is broken 
by separating the analysis into three  parts. First, we analyze the behavior of the solution $m(z)$ as
$\im z \to 0$. Second, we show that the solution is stable under small perturbations of the equation
and the stability is provided by $\Im \langle m(E+i0) \rangle$ for any energy $E$ in the bulk spectrum. Finally, we show that the diagonal
elements of the resolvent of the random matrix satisfy a perturbed version of \eqref{eq:self_consistent_Stieltjes_Gram},
where the perturbation is controlled by large deviation estimates. Stability then provides the local law.

While this program could be completed directly for the Gram matrix and its Dyson equation \eqref{eq:self_consistent_Stieltjes_Gram},
the argument appears much shorter if we used Girko's linearization \eqref{Girko} to reduce the problem 
to a Wigner-type matrix and use the comprehensive analysis of \eqref{eq:QVE} from \cite{AjankiQVE, AjankiCPAM}
and the  local law from \cite{Ajankirandommatrix}.
There are two major obstacles to this naive approach. 

First,  the results of \cite{AjankiQVE, AjankiCPAM} 
are not applicable as $\Sf$ does not satisfy the uniform primitivity assumption imposed in these papers
(recall that
a matrix $A$ is primitive if there is a  positive integer $L$ such that all entries of  $A^L$ are strictly positive).
This property is crucial for many proofs in \cite{AjankiQVE, AjankiCPAM} but $\Sf$
in \eqref{Girko} is a typical example of a nonprimitive matrix. 
It is not a mere technical subtlety, in fact in the current paper, 
the stability estimates of \eqref{eq:QVE} require a  completely different treatment, culminating
in the key technical bound, the Rotation-Inversion lemma (see Lemma~\ref{lem:bulk_stability} later).

Second, Girko's transformation is singular around $z\approx 0$ since it involves a  $z^2=\zeta$ change in the spectral parameter.
This accounts for the singular behavior near zero
in the limiting density for Gram matrices,  while the corresponding Wigner-type matrix
has no singularity at zero. Thus, we need to perform a more accurate  analysis near zero.
  If $p\neq n$, the soft edge case, we derive and analyze two new equations for the 
first  coefficients in the expansion of $m$ around zero. Indeed, the solutions to these new
equations describe the point mass at zero and provide information about the gap above zero in the support of the approximating measure. 
In the hard edge case, $n=p$, an additional symmetry allows us to exclude a point mass at zero.

\emph{Acknowledgement:} The authors thank Zhigang Bao for  helpful discussions. 

\paragraph{Notation}
For vectors $v, w \in \C^l$, the operations product and absolute value are defined componentwise, i.e., $vw=(v_i w_i)_{i=1}^l$ and $\abs{v} = (\abs{v_i})_{i=1}^l$. 
Moreover, for $w \in (\C \setminus \{0\})^l$, we set $1/w\defeq (1/w_i)_{i=1}^l$.
For vectors $v, w \in \C^l$, we define $\avg{w} = l^{-1} \sum_{i=1}^l w_i$, $\scalar{v}{w} = l^{-1} \sum_{i=1}^l \overline{v_i} w_i$, $\norm{w}_2^2 = l^{-1} \sum_{i=1}^l \abs{w_i}^2$ and $\norm{w}_\infty = \max_{i=1, \ldots, l} \abs{w_i}$, $\norm{v}_1 \defeq \avg{\abs{v}}$. Note that $\avg{w} = \scalar{1}{w}$ where we used the convention that $1$ also denotes the vector $(1,\ldots, 1) \in \C^l$. 
For a matrix $A \in \C^{l \times l}$, we use the short notations $\norminf{A} \defeq \norm{A}_{\infty \to \infty}$ and $\normtwo{A} \defeq \norm{A}_{2 \to 2}$ if the domain and the target are equipped with the same norm
whereas we use $\normtwoinf{A}$ to denote the matrix norm of $A$ when it is understood as a map $(\C^l, \norm{\cdot}_2) \to (\C^l, \norm{\cdot}_\infty)$. 

\section{Main results}

Let $X=(x_{ik})_{i,k}$ be a $p \times n$ matrix with independent, centered entries and variance matrix $S = (s_{ik})_{i,k}$, i.e., 
\[\E x_{ik} = 0, \quad s_{ik} \defeq \E \abs{x_{ik}}^2\]
for $i =1, \ldots, p$ and $k=1, \ldots, n$. 

\vspace*{0.5cm}
\noindent \textbf{Assumptions}: 
\begin{enumerate}[(A)]
\item The variance matrix $S$ is \emph{flat}, i.e., there is $s_*>0$ such that 
\[ s_{ik} \leq \frac{s_*}{p+n} \]
for all $i=1, \ldots, p$ and $k=1, \ldots, n$.
\item There are $L_1, L_2\in \N$ and $\psi_1, \psi_2 >0$ such that 
\[ [(SS^t)^{L_1}]_{ij} \geq \frac{\psi_1}{p+n}, \quad [(S^tS)^{L_2}]_{kl} \geq \frac{\psi_2}{p+n} \]
for all $i,j=1,\ldots, p$ and $k, l =1, \ldots, n$.
\item All entries of $X$ have bounded moments in the sense that
there are $\mu_m >0$ for $m \in \N$ such that 
\[ \E \abs{x_{ik}}^m \leq \mu_m s_{ik}^{m/2}\]
for all $i=1, \ldots, p$ and $k = 1, \ldots, n$. 
\item The dimensions of $X$ are comparable with each other, i.e., there are constants $r_1, r_2 >0$ such that 
\begin{equation*}
   r_1 \leq \frac p n\leq r_2.
\end{equation*}
\end{enumerate}

In the following, we will assume that $s_*$, $L_1$, $L_2$, $\psi_1$, $\psi_2$, $r_1$, $r_2$ and the sequence $(\mu_m)_m$ are fixed constants 
which we will call, together with some constants introduced later, \emph{model parameters}. The constants in all our estimates will depend on the model parameters 
without further notice. We will use the notation $f \lesssim g$ if there is a constant $c>0$ that depends on the model parameter only such that $f \leq cg$ and their counterparts $f \gtrsim g$ if 
$g \lesssim f$ and $f \sim g$ if $f \lesssim g$ and $f \gtrsim g$. 
The model parameters will be kept fixed whereas the parameters $p$ and $n$ are large numbers which will eventually be sent to infinity. 

We start with a theorem about the deterministic density. 

\newcommand{\dens}{\ensuremath{\pi}}
\newcommand{\poma}{\ensuremath{\pi_*}}

\begin{thm}  \label{thm:XX_star_Stieltjes_transform}
\begin{enumerate}[(i)]
\item If (A) holds true, then there is a unique holomorphic function $m \colon \Hb \to \C^p$ satisfying 
\begin{equation} \label{eq:m_equation}
-\frac{1}{m(\zeta)} = \zeta - S \frac{1}{1 + S^t m (\zeta) } 
\end{equation}
for all $\zeta \in \Hb$ such that $\Im m(\zeta) >0$ for all $\zeta \in \Hb$. Moreover, there is a probability measure $\nu$ on $\R$ whose support is contained in $[0,4s_*]$ such that 
\begin{equation}
\avg{m(\zeta)} = \int_\R \frac{1}{\omega-\zeta} \nu(\dd \omega)
\end{equation}
for all $\zeta \in \Hb$. 
\item Assume (A), (B) and (D). The measure $\nu$ is absolutely continuous wrt. the Lebesgue measure apart from a possible point mass at zero, i.e., there are a number $\poma \in [0,1]$ and a locally 
Hölder-continuous function $\dens\colon (0,\infty) \to [0,\infty)$ such that $\nu(\dd \omega) = \poma \delta_0(\dd \omega) + \dens(\omega) \mathbf 1(\omega >0)  \dd \omega$.
\end{enumerate}
\end{thm}

Part (i) of this theorem has already been proved in \cite{hachem2007} and we will see that it also follows directly from \cite{AjankiQVE,AjankiCPAM}.
We included this part only for completeness.  Part (ii) is a new result. 

For $\zeta \in \C\setminus \R$, we denote the resolvent of $XX^*$ at $\zeta$ by \[R(\zeta) \defeq (XX^*-\zeta)^{-1}\] and its entries by $R_{ij}(\zeta)$ for $i, j =1, \ldots, p$. 

We state our main result, the local law, i.e., optimal estimates on the resolvent $R$, both in entrywise and in averaged form. 
In both cases, we provide different estimates when the real part of the spectral parameter $\zeta$ is in the bulk and when it is away from the spectrum.
As there may be many zero eigenvalues, hence, a point mass at zero in the density $\nu$, our analysis for spectral parameters $\zeta$ in the vicinity of zero 
requires a special treatment.  
We thus first prove the local law under the general assumptions (A) -- (D) for $\zeta$ away from zero. Some additional 
assumptions in the following subsections will allow us to extend our arguments to all $\zeta$. 

All of our results are uniform in the spectral parameter $\zeta$ which is contained in some spectral domain 
\begin{equation} \label{eq:def_spectral_domain}
\mathbb D_{\delta} \defeq \{ \zeta \in \Hb \colon \delta \leq \abs{\zeta} \leq 10s_*\}
\end{equation}
for some $\delta \geq 0$. In the first result, we assume $\delta >0$.
In the next section, under additional assumptions on $S$, we will work on the bigger spectral domain $\mathbb D_0$ that also includes a neighbourhood of zero.

\begin{thm}[Local Law for Gram matrices] \label{thm:local_law_gram}
Let $\delta, \eps_*>0$ and $\gamma \in (0,1)$. 
If $X$ is a random matrix satisfying (A) -- (D) then for every $\eps >0$ and $D>0$ there is a constant $C_{\eps,D} >0$ such that 
\begin{subequations}\label{eq:local_law_XX*}
\begin{align} 
\P \left( \exists \zeta \in \mathbb D_{\delta}, i, j \in \{1, \ldots, p\}: \Im \zeta \geq p^{-1+\gamma}, \:\dens(\Re \zeta) \geq \eps_*, \: 
\left| R_{ij}(\zeta)- m_i (\zeta) \delta_{ij} \right| \geq \frac{p^{\eps}}{\sqrt{p\Im \zeta}} \right) \leq \frac{C_{\eps,D}}{p^D},  \label{eq:local_law_XX_star_bulk}\\
\P \left( \exists \zeta \in \mathbb D_{\delta}, i, j \in \{1, \ldots, p\}: \dist(\zeta,\supp \nu) \geq \eps_*, \: 
\left| R_{ij}(\zeta)- m_i (\zeta) \delta_{ij} \right| \geq \frac{p^{\eps}}{\sqrt{p}} \right) \leq \frac{C_{\eps,D}}{p^D}, \label{eq:local_law_XX_star_away}
\end{align}
\end{subequations}
for all $p \in \N$. 
Furthermore, for any sequences of deterministic vectors $w \in \C^p$ satisfying $\norm{w}_\infty \leq 1$, we have 
\begin{subequations}\label{eq:local_law_XX*_averaged}
\begin{align} 
\P \left( \exists \zeta \in \mathbb D_{\delta} : \Im \zeta \geq p^{-1+\gamma}, \:\dens(\Re \zeta) \geq \eps_*, \: 
\absa{\frac{1}{p} \sum_{i=1}^p w_i \left[ R_{ii}(\zeta)- m_i (\zeta) \right]} \geq \frac{p^\eps}{p\Im \zeta} \right) \leq \frac{C_{\eps,D}}{p^D}, \\
\P \left(\exists \zeta \in \mathbb D_{\delta} : \dist(\zeta,\supp \nu) \geq \eps_*, \:  \absa{\frac{1}{p} \sum_{i=1}^p w_i \left[ R_{ii}(\zeta)- m_i (\zeta) \right]} \geq \frac{p^\eps}{p} \right)\leq \frac{C_{\eps,D}}{p^D},
\end{align}
\end{subequations}
for all $ p \in \N$. In particular, choosing $w_i = 1$ for all $i=1, \ldots, p$ in \eqref{eq:local_law_XX*_averaged} yields that $p^{-1} \tr R(\zeta)$ is close to $\avg{m(\zeta)}$. 

The constant $C_{\eps,D}$ depends, in addition to $\eps$ and $D$, only on the model parameters and on $\gamma$, $\delta$ and $\eps_*$.
\end{thm}

These results are optimal up to the arbitrarily small tolerance exponents $\gamma>0$ and $\eps>0$.  We remark that under stronger (e.g. subexponential) moment conditions in (C), one may replace the
$p^\gamma$ and $p^\eps$ factors with high powers of $\log p$.

Owing to the symmetry of the assumptions (A) -- (D) in $X$ and $X^*$, we can exchange $X$ and $X^*$ in Theorem \ref{thm:local_law_gram} and obtain a statement about $X^*X$ instead of $XX^*$
as well. 

For the results in the up-coming subsections, we need the following notion of a sequence of high probability events. 

\begin{defi}[Overwhelming probability]
Let $N_0\colon (0,\infty) \to \N$ be a function that depends on the model parameters and the tolerance exponent $\gamma$ only. For a sequence $A = (A^{(p)})_p$ of random events, we say that $A$ holds true \textbf{asymptotically 
with overwhelming probability} (a.w.o.p.) if for all $D>0$ 
\[ \P ( A^{(p)} ) \geq 1- p^D \]
for all $p \geq N_0(D)$. 
\end{defi}

We denote the eigenvalues of $XX^*$ by $\lambda_1 \leq \ldots \leq \lambda_{p}$ and define
\begin{equation}
i(\chi) \defeq \left\lceil p\int_{-\infty}^\chi \nu(\dd \omega)\right\rceil, \quad \text{ for } \chi \in \R.
\end{equation}
For a spectral parameter $\chi \in \R$ in the bulk, the nonnegative integer $i(\chi)$ is the index of an eigenvalue expected to be close to $\chi$.

\begin{thm} \label{thm:Bulk_rigidity_general}
Let $\delta, \eps_*>0$ and $X$ be a random matrix satisfying (A) -- (D).
\begin{enumerate}[(i)]
\item (Bulk rigidity away from zero) For every $\eps>0$ and $D>0$, there exists a constant $C_{\eps,D}> 0$ such that 
\begin{equation}  \label{eq:bulk_rigidity_Gram}
\P\left( \exists \;\tau \in (\delta,10s_*]: \dens(\tau) \geq \eps_*, \abs{ \lambda_{i(\tau)} - \tau} \geq  \frac{p^\eps}{p} \right) \leq \frac{C_{\eps,D}}{p^D}
\end{equation}
holds true for all $p \in \N$. 

The constant $C_{\eps,D}$ depends, in addition to $\eps$ and $D$, only on the model parameters as well as on $\delta$ and $\eps_*$.
\item Away from zero, all eigenvalues lie in the vicinity of the support of $\nu$, i.e., a.w.o.p.
\begin{equation} \label{eq:no_eigenvalues_away_from_support}
 \sigma(XX^*) \cap \{ \tau; \abs{\tau} \geq \delta, ~\dist(\tau, \supp \nu) \geq \eps_* \} = \varnothing.
\end{equation}
\end{enumerate}
\end{thm}

In the following two subsections, we distinguish between square Gram matrices, $n=p$,  and properly rectangular Gram matrices, $\abs{p/n -1} \geq d_* >0$, 
in order to extend the local law, Theorem \ref{thm:local_law_gram}, to include zero in the spectral domain $\mathbb D$. Since the density of states behaves differently around zero in these two cases, 
separate statements and proofs are necessary. 

\subsection{Square Gram matrices}

The following concept is well-known in linear algebra. For understanding singularities of the density of states in random matrix theory, it was introduced in \cite{AjankiQVE}. 

\begin{defi}[Fully indecomposable matrix]
A $K \times K$ matrix $T = (t_{ij})_{i,j=1}^K$ with nonegative entries is called \textbf{fully indecomposable} if for any two subsets $I, J \subset \{1,\ldots, K\}$ such that $\# I + \# J \geq K$, 
 the submatrix $(t_{ij})_{i \in I, j \in J}$ contains a nonzero entry. 
\end{defi}

For square Gram matrices, we add the following assumptions.

\begin{enumerate}
\item[\dimhard] The matrix $X$ is square, i.e., $n=p$.
\item[\zerohard] The matrix $S$ is \textbf{block fully indecomposable}, i.e., there are constants $\varphi >0$, $K \in \N$, a fully indecomposable matrix $Z = (z_{ij})_{i,j=1}^K$ with $z_{ij} \in \{0,1\}$ and 
a partition $(I_i)_{i=1}^K$ of $\{1, \ldots, p\}$ such that 
 \[ \# I_i = \frac{p}{K}, \quad s_{xy} \geq \frac{\varphi}{p+n} z_{ij}, \quad  x \in I_i\text{ and }y \in I_j \]
for all $i, j =1, \ldots, K$. 
\end{enumerate}

The constants $\varphi$ and $K$ in \zerohard are considered model parameters as well.

\begin{rem} \label{rem:assumption_square_Gram}
Clearly, \dimhard yields (D) with $r_1=r_2=1$.
Moreover, adapting the proof of Theorem 2.2.1 in \cite{BapatRaghavan}, we see that \zerohard implies (B) with $L_1$, $L_2$, $\psi_1$ and $\psi_2$ explicitly depending on $\varphi$ and $K$.
\end{rem}

\begin{thm}[Local law for square Gram matrices] \label{thm:eigenvalue_hard_edge}
If $X$ satisfies (A), (C), (E1) and (F1), then 
\begin{enumerate}[(i)]
\item The conclusions of Theorem \ref{thm:local_law_gram} are valid with the following modifications: 
 \eqref{eq:local_law_XX_star_away} and \eqref{eq:local_law_XX*_averaged} hold true for $\delta=0$ $($cf. \eqref{eq:def_spectral_domain}$)$ while instead of \eqref{eq:local_law_XX_star_bulk}, we have 
\begin{equation} \label{eq:local_law_XX*_square}
\P \left( \exists \zeta \in \mathbb D_0, \exists i, j : \Im \zeta \geq p^{-1+\gamma}, \:\dens(\Re \zeta) \geq \eps_*, \: 
\left| R_{ij}(\zeta)- m_i (\zeta) \delta_{ij} \right| \geq p^{\eps}\sqrt\frac{\avg{\Im m(\zeta)}}{{p\Im \zeta}} \right) \leq \frac{C_{\eps,D}}{p^D}.
\end{equation}
\item $\poma =0$ and the limit $\lim_{\omega \downarrow 0}\dens(\omega) \sqrt \omega$ exists and lies in $(0,\infty)$. 
\item (Bulk rigidity down to zero) 
For every $\eps_*>0$ and every $\eps>0$ and $D>0$, there exists a constant $C_{\eps,D}>0$ such that 
\begin{equation} \label{eq:bulk_rigidity}
\P\left( \exists \; \tau \in (0,10s_*]: \dens(\tau) \geq \eps_*, \abs{ \lambda_{i(\tau)} - \tau} \geq  \frac{p^\eps}{p}\left(\sqrt{\tau}+\frac{1}{p}\right) \right) \leq \frac{C_{\eps, D}}{p^D}
\end{equation}
for all $p \in \N$. 
The constant $C_{\eps,D}$ depends, in addition to $\eps$ and $D$, only on the model parameters and on $\eps_*$. 
\item There are no eigenvalues away from the support of $\nu$, i.e., \eqref{eq:no_eigenvalues_away_from_support} holds true with $\delta =0$. 
\end{enumerate}
\end{thm}

We remark that the bound of the individual resolvent entries \eqref{eq:local_law_XX*_square} deteriorates as $\zeta$ gets close to zero since $\avg{\Im m(\zeta)} \sim \abs{\zeta}^{-1/2}$ in this regime 
while the averaged version \eqref{eq:local_law_XX*_averaged}, with $\delta=0$, does not show this behaviour. 

\subsection{Properly rectangular Gram matrices}

\begin{enumerate}
\item[\dimrect] The matrix $X$ is properly rectangular, i.e., there is $d_*>0$ such that 
\[ \left|\frac{p}{n} - 1\right| \geq d_*. \] 
\item[\zerorect] The matrix elements of $S$ are bounded from below, i.e., there is a $\varphi >0$ such that 
	\[ s_{ik} \geq \frac{\varphi}{n+p} \]
for all $i=1, \ldots, p$ and $k =1, \ldots, n$.
\end{enumerate}

The constants $d_*$ and $\varphi$ in \dimrect and \zerorect, respectively, are also considered as model parameters. 
Note that \zerorect is a simpler version of \zerohard. For properly rectangular Gram matrices we work under the stronger condition \zerorect 
for simplicity but our analysis could be adjusted to some weaker condition as well.

\begin{rem}
Note that \zerorect immediately implies condition (B) with $L=1$.
\end{rem}

We introduce the lower edge of the absolutely continuous part of the distribution $\nu$ for properly rectangular Gram matrices
\begin{equation} \label{eq:def_delta_dens}
 \delta_\dens \defeq \inf\{ \omega>0\colon \dens(\omega) >0 \}. 
\end{equation}

\begin{thm}[Local law for properly rectangular Gram matrices] \label{thm:eigenvalue_XX_star_X_star_X}
Let $X$ be a random matrix satisfying (A), (C), (D), \dimrect and \zerorect. We have 
\begin{enumerate}[(i)]
\item The gap between zero and the lower edge is macroscopic $\delta_\dens \sim 1$. 
\item (Bulk rigidity down to zero) The estimate \eqref{eq:bulk_rigidity_Gram} holds true with $\delta = 0$. 
\item 
There are no eigenvalues away from the support of $\nu$, i.e., \eqref{eq:no_eigenvalues_away_from_support} holds true with $\delta =0$. 
\item If $p>n$, then $\poma = 1 - n/p$ and $\dim \ker(XX^*) =p-n$ a.w.o.p.
\item If $p<n$, then $\poma = 0$ and $\dim \ker(XX^*) =0$ a.w.o.p.
\item (Local law around zero) For every $\eps_* \in (0, \delta_\dens)$, every $\eps>0$ and $D>0$, there exists a constant $C_{\eps, D}>0$, 
such that  
\begin{equation}
\P\left( \exists \,\zeta \in \Hb, i,j \in \{ 1, \ldots, p \} \colon \abs{\zeta} \leq \delta_\dens-\eps_*, \: \abs{R_{ij}(\zeta) - m_i(\zeta) \delta_{ij}} \geq \frac{p^\eps}{\abs{\zeta}\sqrt{p}}\right) \leq \frac{C_{\eps,D}}{p^D}, 
\label{eq:local_law_around_zero1} 
\end{equation}
for all $p \in \N$ if $p>n$ and 
\begin{equation}
\P\left( \exists \,\zeta \in \Hb, i, j \in \{1, \ldots, p\} \colon \abs{\zeta} \leq \delta_\dens-\eps_*, \: \abs{R_{ij}(\zeta) - m_i(\zeta) \delta_{ij}} \geq \frac{p^\eps}{\sqrt{p}}\right) \leq \frac{C_{\eps,D}}{p^D}, 
\label{eq:local_law_around_zero2} 
\end{equation}
for all $p \in \N$ if $p<n$. Moreover, in both cases
\begin{equation}
\P \left( \exists \,\zeta \in \Hb\colon \abs{\zeta} \leq \delta_\dens-\eps_*, \: \absa{\frac{1}{p} \sum_{i=1}^p [R_{ii}(\zeta) - m_i(\zeta) ]} \geq \frac{p^\eps}{p}\right) \leq \frac{C_{\eps,D}}{p^D},
\label{eq:local_law_around_zero3} 
\end{equation}
for all $p \in \N$.

The constant $C_{\eps,D}$ depends, in addition to $\eps$ and $D$, only on the model parameters and on $\eps_*$. 
\end{enumerate}
\end{thm}

If $p>n$, then the Stieltjes transform of the empirical spectral measure of $XX^*$ has a term proportional to $1/{\zeta}$ due to the macroscopic kernel of $XX^*$. 
This is the origin of the additional factor $1/\abs{\zeta}$ in \eqref{eq:local_law_around_zero1}. 

\begin{rem}
As a consequence of Theorem \ref{thm:eigenvalue_hard_edge} and Theorem \ref{thm:eigenvalue_XX_star_X_star_X} and under the same conditions, the standard methods in \cite{EJP3054} and \cite{Ajankirandommatrix} can be 
used to prove an anisotropic law and delocalization of eigenvectors in the bulk.
\end{rem}

\section{Quadratic vector equation}

For the rest of the paper, without loss of generality we will assume that $s_*=1$ in (A), which can be achieved by a simple rescaling of $X$.
In the whole section, we will assume that the matrix $S$ satisfies (A), (B) and (D) without further notice. 

\subsection{Self-consistent equation for resolvent entries}

We introduce the random matrix $\Xf$ and the deterministic matrix $\Sf$ defined through
\begin{equation} \label{eq:def_H_S}
 \Xf  = \begin{pmatrix} 0 & X \\ X^* & 0 \end{pmatrix}, \qquad \Sf = \begin{pmatrix} 0 & S \\ S^t & 0 \end{pmatrix}.
\end{equation}
Note that both matrices, $\Xf$ and $\Sf$ have dimensions $(p+n)\times (p+n)$. We denote their entries by $\Xf = (h_{xy})_{x,y}$ and $\Sf = (\sf_{xy})_{x,y}$, respectively, where $\sf_{xy} = \E \abs{h_{xy}}^2$ 
with $x,y=1, \ldots, n+p$. 

It is easy to see that condition (B) implies 
\begin{enumerate}
\item[(B')] There are $L \in \N$ and $\psi >0$ such that 
\begin{equation} \label{eq:lower_bound_Sf_L}
 \sum_{k=1}^L (\Sf^k)_{xy} \geq \frac{\psi}{n+p}
\end{equation}
for all $x, y = 1, \ldots, n+p$. 
\end{enumerate}

In the following, a crucial part of the analysis will be devoted to understanding the resolvent of $H$ at $z\in \Hb$, i.e., the matrix 
\begin{equation} \label{eq:def_G}
 G(z) \defeq (H-z)^{-1} 
\end{equation}
whose entries are denoted by $G_{xy}(z)$ for $x, y =1, \ldots, n+p$. 
For $V \subset \{1, \ldots, n+p\}$, we use the notation $G_{xy}^{(V)}$ to denote the entries of the resolvent $G^{(V)}(z) = (H^{(V)}-z)^{-1}$ of the matrix $H^{(V)}_{xy} = h_{xy} \mathbf 1(x \notin V) \mathbf 
1(y \notin V)$ where $x,y =1, \ldots, n+p$.

The Schur complement formula and the resolvent identities applied to $G(z)$ yield the self-consistent equations 
\begin{subequations} \label{eq:self_const_split}
\begin{align}
-\frac{1}{\gone{i}(z)} & = z + \sum_{k=1}^n s_{ik} \gtwo{k}(z) + \done{i}(z),  \label{eq:self_const_1}\\
-\frac{1}{\gtwo{k}(z)} & = z + \sum_{i=1}^p s_{ik} \gone{i}(z) + \dtwo{k}(z), \label{eq:self_const_2} 
\end{align}
\end{subequations}
where $\gone{i}(z) \defeq  G_{ii}(z)$ for $i=1, \ldots, p$ and $\gtwo{k}(z)\defeq G_{k+p,k+p}(z)$ for $k=1, \ldots, n$ with the error terms 
\begin{align*}
\done{r}  & \defeq \sum_{k,l=1, k\neq l}^n x_{rk}G_{kl}^{(r)} \overline x_{rl} + \sum_{k=1}^n \left(\abs{x_{rk}}^2-s_{rk}\right) G_{k+n,k+n}^{(r)} - \sum_{k=1}^n s_{rk} \frac{G_{k+n,r}G_{r,k+n}}{\gone{r}}, \\ 
\dtwo{m}  & \defeq \sum_{i, j =1, i \neq j}^{p} \overline x_{im}G_{ij}^{(m+p)} x_{jm} + \sum_{i=1}^p \left(\abs{x_{im}}^2-s_{im}\right) G_{ii}^{(m+p)} - \sum_{i=1}^p s_{im} \frac{G_{i,m+p}G_{m+p,i}}{\gtwo{m}} 
\end{align*}
for $r=1, \ldots, p$ and $m=1, \ldots, n$. 

We will prove a local law which states that $\gone{i}(z)$ and $\gtwo{k}(z)$ can be approximated by $M_{1,i}(z)$ and $M_{2,k}(z)$, respectively,
where $M_1\colon \Hb \to \C^p$ and $M_2\colon \Hb \to \C^n$ are the unique solution of 
\begin{subequations} \label{eq:QVE_split}
\begin{align}
-\frac{1}{M_1} & = z + S M_2 , \label{eq:QVE_m1}\\
-\frac{1}{M_2} & = z + S^t M_1, \label{eq:QVE_m2}
\end{align}
\end{subequations}
which satisfy $\Im M_1(z)>0$ and $\Im M_2(z)>0$ for all $z \in \Hb$. 

The system of self-consistent equations for $\gon$ and $\gtw$ in \eqref{eq:self_const_split} can be seen as a perturbation of the system \eqref{eq:QVE_split}.
With the help of $\Sf$, equations \eqref{eq:QVE_m1} and \eqref{eq:QVE_m2} can be combined to a vector equation for $\Mf=(M_1, M_2)^t \in \Hb^{p+n}$, i.e., 
\begin{equation}
 -\frac{1}{\Mf}= z + \Sf\Mf. 
\label{eq:combined_QVE}
\end{equation}
Since $\Sf$ is symmetric, has nonnegative entries and fulfills (A) with $s_*=1$, Theorem 2.1 in \cite{AjankiQVE} is applicable to \eqref{eq:combined_QVE}.
Here, we take $a=0$ in Theorem 2.1 of \cite{AjankiQVE}. This theorem implies that \eqref{eq:combined_QVE} has a unique solution $\Mf$ with $\Im \Mf(z) >0$ for any $z \in \Hb$. 
Moreover, by this theorem, $\Mf_x$ is the Stieltjes transform of a symmetric probability measure 
on $\R$ whose support is contained in $[-2,2]$ for all $x=1, \ldots, n+p$ and we have 
\begin{equation}
\label{eq:L2_bound}
\norm{\Mf(z)}_2 \leq \frac{2}{\abs{z}}
\end{equation}
for all $z \in \Hb$. 
The function $\avg{\Mf}$ is the Stieltjes transform of a symmetric probability measure on $\R$ which we denote by $\rho$, i.e., 
\begin{equation} \label{eq:def_rho}
\avg{\Mf(z)} = \int_\R \frac{1}{t-z} \rho(\di t)
\end{equation}
for $z \in \Hb$. Its support is contained in $[-2,2]$.

We combine \eqref{eq:self_const_1} and \eqref{eq:self_const_2} to obtain 
\begin{equation}
 -\frac{1}{\gf} = z + \Sf \gf + \df, 
\label{eq:perturbed_combined_QVE}
\end{equation}
where $\gf = (g_1, g_2)^t$ and $\df =(d_1, d_2)^t$. We think of \eqref{eq:perturbed_combined_QVE} as a perturbation of \eqref{eq:combined_QVE} and 
most of the subsequent subsection is devoted to the study of \eqref{eq:perturbed_combined_QVE} for an arbitrary perturbation $\df$. 

Before we start studying \eqref{eq:combined_QVE} we want to indicate how $m$ and $R$ are related to $\Mf$ and $G$, respectively. 
The Stieltjes transforms as well as the resolvents are essentially related via the same transformation of the spectral parameter.
If $\mathcal G_{11}(z)$ denotes the upper left $p \times p$ block of $G(z)$ then $R(z^2) = (XX^* - z^2)^{-1} = \mathcal G_{11}(z)/z$. 
In the proof of Theorem \ref{thm:XX_star_Stieltjes_transform} in Subsection \ref{subsec:proof_of_existence_and_structure}, we will see that $m$ and $M_1$ are related via $m(\zeta) = M_1(\sqrt \zeta)/\sqrt \zeta$.  
(We always choose the branch of the square root satisfying $\Im \sqrt \zeta >0$ for $\Im \zeta >0$.) 
Assuming this relation and introducing $m_2(\zeta) \defeq M_2(\sqrt \zeta)/\sqrt \zeta$, we obtain 
\begin{align*}
 - \frac{1}{m(\zeta)} & = \zeta (1 + S m_2(\zeta)), \\
 - \frac{1}{m_2(\zeta)} & = \zeta (1 + S^t m(\zeta)) 
\end{align*}
from \eqref{eq:QVE_split}. 
Solving the second equation for $m_2$ and plugging the result into the first one yields \eqref{eq:m_equation} immediately.
In fact, $m_2$ is the analogue of $m$ corresponding to $X^*X$, i.e, the Stieltjes transform of the deterministic measure approximating the eigenvalue density of $X^*X$. 

\subsection{Structure of the solution}

We first notice that the inequality $s_{ik} \leq 1/(n+p)$ implies
\begin{equation} \label{eq:St_norm_two_inf_less_1}
 \norm{S^t w}_\infty = \max_{k=1, \ldots, n} \sum_{i=1}^p s_{ik} \abs{w_i} \leq \max_{k=1, \ldots, n} \left( p \sum_{i=1}^p s_{ik}^2 \right)^{1/2} \left(\frac{1}{p}\sum_{i=1}^p \abs{w_i}^2\right)^{1/2}
 \leq \norm{w}_2  
\end{equation}
for all $w \in \C^p$, i.e., $\normtwoinf{S^t} \leq 1$.
Now, we establish some preliminary estimates on the solution of \eqref{eq:combined_QVE}.

\begin{lem} \label{lem:estimates_m}
Let $z \in \Hb$ and $x\in \{1, \ldots, n+p\}$. We have 
\begin{subequations} \label{eq:mf_upper_bound_supp_rho}
\begin{align}
\abs{\Mf_x(z)} & \leq \frac{1}{\dist(z,\supp \rho)} 
, \label{eq:mf_upper_bound_supp_rho1}\\
\Im \Mf_x(z) & \leq \frac{\Im z}{\dist(z,\supp \rho)^2} 
. \label{eq:mf_upper_bound_supp_rho2}
\end{align}
\end{subequations}
If $z \in \Hb$ and $\abs{z} \leq 10$  then 
\begin{subequations}
\begin{align}
\abs{z} & \lesssim  \abs{\Mf_x(z)} \leq \norm{\Mf(z)}_\infty \lesssim  \frac{\abs{z}^{2-2L}}{\langle \Im \Mf (z) \rangle} \label{eq:first_estimate_m}\\
\abs{z}^{2L}  \langle \Im \Mf (z) \rangle & \lesssim  \Im \Mf_x(z) .  \label{eq:Im_m_i_geq_average_Im_m}
\end{align}
\end{subequations}
In particular, the support of the measures representing $\Mf_x$ is independent of $x$ away from zero.
\end{lem}

The proof essentially follows the same line of arguments as the proof of Lemma 5.4 in \cite{AjankiQVE}.
However, instead of using the lower bound on the entries of $S^L$ as in \cite{AjankiQVE} we have to make use of the lower bound on the entries of $\sum_{k=1}^L \Sf^k$. 

To prove another auxiliary estimate on $\Sf$, we define the vectors $\Sf_x = (\sf_{xy})_{y=1,\ldots, n+p} \in  \R^{n+p}$ for $x=1, \ldots, n+p$. 
Since \eqref{eq:lower_bound_Sf_L} implies 
\[\psi \leq \sum_{k=1}^L \sum_{y=1}^{n+p} (\Sf^k)_{xy} \leq \sum_{k=1}^L \sum_{v=1}^{n+p} \sf_{xv} \max_{t=1, \ldots, n+p} \sum_{y=1}^{n+p} (\Sf^{k-1})_{ty}\leq L \sum_{v=1}^{n+p} \sf_{xv}  \]
for any fixed $x=1, \ldots, n+p$, where we used $\norminf{\Sf^{k-1}} \leq \norminf{\Sf}^{k-1} \leq 1$ by (A), we obtain
\begin{equation}
\inf_{x=1, \ldots, n+p} \norm{\Sf_x}_1 \geq \frac{\psi}{L}. 
\label{eq:estimate_S_i}
\end{equation}
In particular, together with (A), this implies  
\begin{equation}
\label{eq:sum_s_ij_sim_1}
\sum_{j=1}^p s_{jk} \sim 1, \quad \sum_{l=1}^n s_{il} \sim 1, \qquad i =1, \ldots, p,\quad k=1, \ldots, n. 
\end{equation} 
In the study of the stability of \eqref{eq:combined_QVE} when perturbed by a vector $\df$, as in \eqref{eq:perturbed_combined_QVE}, the linear operator 
\begin{equation} \label{eq:definition_Ff}
\Ff(z)v \defeq \abs{\Mf(z)} \Sf( \abs{\Mf(z)}v)
\end{equation}
for $v \in \C^{n+p}$ plays an important role. 
Before we collect some properties of operators of this type in the next lemma, we first recall the definition of the gap of an operator from \cite{AjankiQVE}. 

\begin{defi} \label{def:Gap}
Let $T$ be a compact self-adjoint operator on a Hilbert space. The \emph{spectral gap} $\Gap(T) \geq 0$ is the difference between the two largest eigenvalues of $\abs{T}$ (defined by spectral 
calculus). If the operator norm $\norm{T}$ is a degenerate eigenvalue of $\abs{T}$, then $\Gap(T) =0$. 
\end{defi}

In the next lemma, we study matrices of the form $\widehat \Ff(r)_{xy} \defeq r_x \sf_{xy} r_y $ where $r \in (0,\infty)^{n+p}$ and $x,y = 1, \ldots, n+p$. 
If $\inf_x r_x>0$ then \eqref{eq:lower_bound_Sf_L} implies that all entries of $\sum_{k=1}^L \wh \Ff(r)^k$ are strictly positive. Therefore, by the Perron-Frobenius theorem, the eigenspace corresponding to the 
largest eigenvalue $\wh \lambda(r)$ of $\wh \Ff(r)$ is one-dimensional and spanned by a unique non-negative vector $\widehat \ff = \wh \ff(r)$ such that $\langle \widehat \ff, \widehat \ff \rangle =1 $. 

The block structure of $\Sf$ implies that there is a matrix $\wh F(r) \in \R^{p \times n}$ such that 
\begin{equation}\label{eq:intro_wh_F}
 \wh \Ff(r) = \begin{pmatrix} 0 & \wh F(r) \\ \wh F(r)^t & 0 \end{pmatrix}.
\end{equation}
However, for this kind of operator, we obtain $\sigma\big(\wh \Ff(r)\big) = - \sigma\big(\wh \Ff(r)\big)$, i.e., $\Gap(\wh \Ff(r)) = 0$ by above definition. Therefore, we will 
compute $\Gap(\wh F(r) \wh F(r) ^t)$, instead. 
We will apply these observations for $\Ff(z)$ where the blocks $\wh F(\abs{\Mf(z)})$ will be denoted by $F(z)$. 

\begin{lem} \label{lem:hat_F}
For a vector $r \in (0,\infty)^{n+p}$ which is bounded by constants $r_+ \in (0,\infty)$ and $r_- \in (0,1]$, i.e., 
\[  r_-  \leq r_x \leq r_+  \]
for all $x= 1, \ldots, n+p$, we define the matrix $\widehat \Ff(r)$ with entries $\widehat \Ff(r)_{xy} \defeq r_x \sf_{xy} r_y $ for $x, y = 1, \ldots, n+p$. Then the eigenspace corresponding to 
$\widehat \lambda(r) \defeq \norm{\widehat \Ff(r)}_{2 \to 2}$ is one-dimensional and $\wh \lambda(r)$ satisfies the estimates 
\begin{equation}
r_-^2 \lesssim \widehat \lambda(r) \lesssim r_+^2. 
\label{eq:estimates_hat_lambda}
\end{equation}
There is a unique eigenvector $\wh \ff=\wh \ff(r)$ corresponding to $\wh\lambda(r)$ satisfying $\wh \ff_x \geq 0$ and $\norm{\wh \ff}_2 =1$. Its  components satisfy
\begin{equation}
\label{eq:estimates_wh_f}
 \frac{r_-^{2L}}{r_+^4} \min\left\{ \wh \lambda(r), \wh \lambda(r)^{-L+2}\right\} \lesssim \wh \ff_x \lesssim \frac{r_+^4}{\wh \lambda(r)^2}, \quad \text{ for all }x = 1, \ldots, n+p. 
\end{equation}
Moreover,  $\wh F(r) \wh F(r)^t$ has a spectral gap 
\begin{equation} \label{eq:Gap_FF^t}
\Gap\left(\wh F(r) \wh F(r)^t\right) \gtrsim \frac{r_-^{8L}}{r_+^{16}} \min\left \{ \widehat \lambda(r)^6, \widehat \lambda(r)^{-8L +10}\right\}.
\end{equation}
\end{lem}

The estimates in \eqref{eq:estimates_hat_lambda} and \eqref{eq:estimates_wh_f} can basically be proved following the proof of Lemma 5.6 in \cite{AjankiQVE} where 
$S^L$ is replaced by $\sum_{k=1}^L \Sf^k$ and $(\wh F/\wh \lambda)^L$ by $\sum_{k=1}^L (\wh \Ff/\wh \lambda)^k$. Therefore, we will only show \eqref{eq:Gap_FF^t} assuming the other estimates.

\begin{proof} 
We write $\wh \ff = (\wh f_1, \wh f_2)^t$ for $\wh f_1 \in \C^p$ and $\wh f_2 \in \C^n$ and define a linear operator on $\C^p$ through
\[ T \defeq \sum_{k=1}^L \left(\frac{\wh F \wh F^t}{\wh \lambda^2}\right)^k. \]
Thus, $\normtwo{T} = L$ as $T \wh f_1 = L \wh f_1$. Using (B') we first estimate the entries $t_{ij}$ by 
\[ t_{ij} \geq \sum_{k=1}^L \frac{r_-^{4k}}{\wh \lambda^{2k}} \left((SS^t)^k\right)_{ij} \geq r_-^{4L} \min \left\{ \wh \lambda^{-2}, \wh \lambda^{-2L}\right\} \frac{\psi}{n+p}, \quad\text{ for }i, j = 1, \ldots, p. \]
Estimating $\norm{\wh f_1}_2$ and $\norm{\wh f_1}_\infty$ from \eqref{eq:estimates_wh_f} and applying Lemma 5.6 in \cite{AjankiCPAM} or Lemma 5.7 in \cite{AjankiQVE} yields 
\[ \Gap(T) \geq \frac{\norm{\wh f_1}_2^2}{\norm{\wh f_1}_\infty^2} p \inf_{i,j} t_{ij}  \gtrsim \frac{r_-^{8L}}{r_+^{16}}\min \left\{ \wh \lambda^4, \wh \lambda^{-8L+8}\right\}. \]
Here we used (D) and note that the factor $\inf_{i,j} t_{ij}$ in Lemma 5.6 in \cite{AjankiCPAM} is replaced by $p \inf_{i,j} t_{ij}$ as $t_{ij}$ are considered as the matrix entries of $T$ and not as the 
kernel of an integral operator on $L^2(\{1, \ldots,p \})$ where $\{1, \ldots, p\}$ is equipped with the uniform probability measure. 
As $q(x) \defeq x + x^2 + \ldots + x^L$ is a monotonously increasing, differentiable function on $[0,1]$ and $\sigma(\wh F \wh F^t/\wh \lambda^2) \subset [0,1]$ we obtain 
$\Gap(T) \sim \Gap(\wh F \wh F^t)/\wh \lambda^2$ which concludes the proof.
\end{proof}

\begin{lem} \label{eq:F_without_hat}
The matrix $\Ff(z)$ defined in \eqref{eq:definition_Ff} with entries $\Ff_{xy}(z) = \abs{\Mf_x(z)} \sigma_{xy}\abs{\Mf_y(z)}$ has the norm  
\begin{equation} \label{eq:normtwo_F_explicit}
 \normtwo{\Ff(z)} = 1- \frac{\Im z \langle \ff(z) \abs{\Mf(z)} \rangle}{ \left\langle \ff(z)\Im \Mf(z)\abs{\Mf(z)}^{-1} \right\rangle },
\end{equation}
where $\ff(z)$ is the unique eigenvector of $\Ff(z)$ associated to $\normtwo{\Ff(z)}$. In particular, we obtain 
\begin{equation}
\label{eq:estimate_norm_F}
(1-\normtwo{\Ff(z)})^{-1} \lesssim \frac{1}{\abs{z}} \min\left\{\frac{1}{\Im z}, \frac{1}{\abs{z} \dist(z, \supp \rho)^2}  \right\}
\end{equation}
for $z \in \Hb$ satisfying $\abs{z} \leq 10$. 
\end{lem}

\begin{proof}
The derivation of \eqref{eq:normtwo_F_explicit} follows the same steps as the proof of (4.4) in \cite{AjankiCPAM} (compare Lemma 5.5 in \cite{AjankiQVE} as well). 
We take the imaginary part of \eqref{eq:combined_QVE}, multiply the result by $\abs{\Mf}$ and take the scalar product with $\ff$. Thus, we obtain 
\begin{equation} \label{eq:normtwo_F_aux_1}
 \scalara{\ff}{\frac{\Im \Mf}{\abs{\Mf}}} = \Im z \avg{\ff \abs{\Mf}} + \normtwo{\Ff}\scalara{\ff}{\frac{\Im \Mf}{\abs{\Mf}}},
\end{equation}
where we used the symmetry of $\Ff$ and $\Ff\ff = \normtwo{\Ff}\ff$. Solving \eqref{eq:normtwo_F_aux_1} for $\normtwo{\Ff}$ yields \eqref{eq:normtwo_F_explicit}.

 Now, \eqref{eq:estimate_norm_F} 
is a direct consequence of Lemma \ref{lem:estimates_m} and \eqref{eq:normtwo_F_explicit}. 
\end{proof}

\subsection{Stability away from the edges and continuity}

All estimates of $\Mf - \gf$, when $\Mf$ and $\gf$ satisfy \eqref{eq:combined_QVE} and \eqref{eq:perturbed_combined_QVE}, respectively, are 
based on inverting the linear operator 
\[ \Bf(z) v \defeq \frac{\abs{\Mf(z)}^2}{\Mf(z)^2} v - \Ff(z)v \]
for $v \in \C^{n+p}$. The following lemma bounds $\Bf^{-1}(z)$ in terms of $\langle \Im \Mf (z) \rangle$ if $z$ is away from zero. 
For $\delta >0$, we use the notation $f \lesssim_\delta g$ if and only if there is an $r >0$ which is allowed to depend on model parameters such that $f \lesssim \delta^{-r} g$. 

\begin{lem} \label{lem:bound_B_inverse}
There is a universal constant $\kappa\in \N$ such that for all $\delta>0$ we have 
\begin{align}
\normtwo{\Bf^{-1}(z)}  &\lesssim_\delta 
\min\left\{ \frac{1}{(\Re z)^2 \langle \Im \Mf(z) \rangle^{\kappa}}, \: \frac{1}{\Im z},\:  \frac{1}{\dist(z, \supp \rho)^2}\right\}, 
\label{eq:estimate_norm_B_inverse}\\
\norminf{\Bf^{-1}(z)}  &\lesssim_\delta 
\min\left\{ \frac{1}{(\Re z)^2 \langle \Im \Mf(z) \rangle^{\kappa+2}},\: \frac{1}{(\Im z)^3}, \:  \frac{1}{\dist(z, \supp \rho)^4}\right\}  
\label{eq:estimate_norm_B_inverse_infty} 
\end{align}
for all $z \in \Hb$ satisfying 
$\delta\leq \abs{z} \leq 10$. 
\end{lem}

For the proof of this result, we will need the two following lemmata. 
We recall that by the Perron-Frobenius theorem an irreducible matrix with nonnegative entries has a unique $\ell^2$-normalized eigenvector with positive entries corresponding to its largest eigenvalue. 
By the definition of the spectral gap, Definition \ref{def:Gap}, we observe that if $AA^*$ is irreducible then $\Gap(AA^*)=\normtwo{AA^*}-\max ( \sigma(AA^*)\setminus\{\normtwo{AA^*}\})$.

\begin{lem}[Rotation-Inversion Lemma] \label{lem:bulk_stability}
There exists a positive constant $C$ such that for all $n,p \in \N$, unitary matrices $U_1 \in \C^{p \times p}$, $U_2 \in \C^{n \times n}$ and $A \in \R^{p \times n}$ with nonnegative entries such that $A^*A$ and 
$AA^*$ are irreducible and $\normtwo{A^*A}\in (0,1]$, the following bound holds:
\begin{equation} \label{bound on block inverse}
\normbb{
 \left(
\begin{array}{cc}
U_1	&	A
\\
A^*	&	U_2
\end{array}
\right)^{-1}
}_2
\,\leq\, 
\frac{C}{\Gap(AA^*)\abs{1-\normtwo{A^{*}A}\scalar{v_1}{U_1v_1}\scalar{v_2}{U_2v_2}}}\,,
\end{equation}
where $v_1\in \C^p$ and $v_2 \in \C^n$ are the unique positive, normalized eigenvectors with $AA^*v_1=\normtwo{A^*A} v_1$ and $A^*Av_2=\normtwo{A^*A}v_2$.
The norm on the left hand side of  \eqref{bound on block inverse} is infinite if and only if the right hand side of \eqref{bound on block inverse} is infinite, i.e., in this case the inverse does not exist.
\end{lem}

This lemma is proved in the appendix. 

\begin{lem}
Let $R\colon \C^{n+p} \to \C^{n+p}$ be a linear operator and $D \colon \C^{n+p} \to \C^{n+p}$ a diagonal operator. 
If $R - D$ is invertible and $D_{xx} \neq 0$ for all $x =1, \ldots, n+p$ then
\begin{equation}
\norminf{(R - D )^{-1} } \leq \left(\inf_{x=1}^{n+p} \abs{D_{xx}}\right)^{-1} \left( 1 + \normtwoinf{R}\normtwo{(R - D )^{-1}} \right). 
\label{eq:norm_hat_F_inverse}
\end{equation}
\end{lem}

The proof of \eqref{eq:norm_hat_F_inverse} follows a similar way as the proof of (5.28) in \cite{AjankiQVE}.

\begin{proof}[Proof of Lemma \ref{lem:bound_B_inverse}]
The bound on $\norminf{\Bf^{-1}(z)}$, \eqref{eq:estimate_norm_B_inverse_infty}, follows from \eqref{eq:estimate_norm_B_inverse} by employing \eqref{eq:norm_hat_F_inverse}.
We use \eqref{eq:norm_hat_F_inverse} with $R = \Ff(z)$ and $D = \abs{\Mf(z)}^2/\Mf(z)^2$ and observe that $\normtwoinf{\Ff(z)} \leq \norm{\Mf}_\infty^2 \normtwoinf{\Sf}$. 
Therefore, \eqref{eq:estimate_norm_B_inverse_infty} follows from \eqref{eq:estimate_norm_B_inverse} as $\norm{\Mf}_\infty \lesssim \min\{ \avg{\Im \Mf}^{-1}, (\Im z)^{-1}, \dist(z, \supp \rho)^{-1}\}$ 
by \eqref{eq:first_estimate_m} and \\ $\min\{ \avg{\Im \Mf}^{-1}, (\Im z)^{-1}, \dist(z, \supp \rho)^{-1}\} \gtrsim_\delta 1$ by \eqref{eq:first_estimate_m} and $\delta \leq \abs{z} \leq 10$.

Now we prove \eqref{eq:estimate_norm_B_inverse}. Our first goal is the following estimate
\begin{equation}
\normtwo{\Bf^{-1}(z)} \lesssim_\delta \frac{1}{\Gap(F(z)F(z)^t) (\Re z)^2 \langle \Im \Mf(z) \rangle^{\kappa} }
\label{eq:first_estimate_B_inverse}
\end{equation}
for some universal $\kappa \in \N$ which will be a consequence of Lemma \ref{lem:bulk_stability}. We apply this lemma with 
\[ \begin{pmatrix} 0 & F(z) \\ F(z)^t & 0 \end{pmatrix} = \Ff(z) \defeq \widehat \Ff(\abs{\Mf(z)}), \quad \frak U \defeq \begin{pmatrix} U_1 & 0 \\ 0 & U_2 \end{pmatrix} = \diag\left(\frac{\abs{\Mf(z)}^2}{\Mf(z)^2}\right) \] 
and $v_1 \defeq f_1/\norm{f_1}_2$ and $v_2 \defeq f_2/\norm{f_2}_2$ where $\ff = (f_1, f_2) \in \C^{p+n}$. 
Note that $\lambda(z) \defeq  \wh \lambda(\abs{\Mf(z)}) = \normtwo{\Ff(z)}$ in Lemma \ref{lem:hat_F} and $F(z) = \wh F(\abs{\Mf(z)})$ in the notation of \eqref{eq:intro_wh_F}.
In Lemma \ref{lem:hat_F}, we choose $r_- \defeq \inf_x \abs{\Mf_x(z)}$ and $r_+ \defeq \norm{\Mf(z)}_\infty$ and use the bounds $r_- \gtrsim \abs{z}$ and $r_+ \lesssim \abs{z}^{2-2L}/\langle \Im \Mf(z)\rangle$ 
by \eqref{eq:first_estimate_m}. Moreover, we have 
\begin{equation} \label{eq:estimates_lambda}
\abs{z}^2 \lesssim \normtwo{\Ff(z)} \leq 1 
\end{equation}
by \eqref{eq:first_estimate_m}, \eqref{eq:estimates_hat_lambda} and \eqref{eq:normtwo_F_explicit}.

We write $\frak{U}=\diag(\ee^{-\i2\psi})$, i.e., $\ee^{\i \psi} = \Mf/\abs{\Mf}$, to obtain
\[
\scalar{v_1}{U_1 v_1}=\scalar{v_1}{(\cos \psi_1- \i \sin \psi_1)^2v_1}\,=\, \scalar{v_1}{(1-2(\sin \psi_1)^2-2\i\cos \psi_1\sin\psi_1)v_1}
\]
and a similar relation holds for $\scalar{v_2}{U_2v_2}$. Thus, we compute
\begin{align*}
&\Re \left(1-\normtwo{F(z)^tF(z)}\scalar{v_1}{(1-2(\sin \psi_1)^2-2\i\cos \psi_1\sin\psi_1)v_1}\scalar{v_2}{(1-2(\sin \psi_2)^2-2\i\cos \psi_2\sin\psi_2)v_2}\right) \\
\,=\,&1-\normtwo{F(z)^tF(z)}(1-2\scalar{v_1}{(\sin \psi_1)^2v_1}-2\scalar{v_2}{(\sin \psi_2)^2v_2} +4\scalar{v_1}{(\sin \psi_1)^2v_1}\scalar{v_2}{(\sin \psi_2)^2v_2}) 
\end{align*}
Using $2a+2b-4ab \geq (a+b)(2-a-b)$ for $a, b \in \R$, and estimating the absolute value by the real part yields 
\begin{align}
 \absa{1-\normtwo{F(z)^tF(z)}\scalar{v_1}{U_1 v_1} \scalar{v_2}{U_2v_2}} \hspace*{9cm} \nonumber \\
 \geq 1-\normtwo{F(z)^tF(z)} + \normtwo{F(z)^tF(z)} \left(\langle v_1,\left(\sin \psi_1\right)^2 v_1\rangle+\langle v_2,\left(\sin \psi_2\right)^2v_2\rangle\right) \nonumber \\ 
~~~~ \times \left( \langle v_1, \left(\cos \psi_1\right)^2 v_1 \rangle+\langle v_2,\left(\cos \psi_2\right)^2v_2\rangle\right)\nonumber \\
 \gtrsim \abs{z}^4 \langle \ff, \left(\sin \psi\right)^2 \ff \rangle \langle \ff, \left( \cos \psi\right)^2 \ff \rangle\hspace*{6.7cm}\nonumber \\
 \gtrsim_\delta \left(\inf_{x=1,\ldots, n+p} \ff_x^4\right) \left\langle\left(\frac{\Im \Mf}{\abs{\Mf}}\right)^2\right\rangle\left\langle\left(\frac{\Re \Mf}{\abs{\Mf}}\right)^2\right\rangle,\hspace*{3.9cm}\label{eq:estimate_norm_B_inverse_2}
\end{align}
where we used $1 \geq \normtwo{F(z)^tF(z)}= \normtwo{\Ff}^2 \gtrsim \abs{z}^4$ by \eqref{eq:estimates_lambda} and 
$\langle \ff, \left(\sin \psi\right)^2 \ff \rangle \langle \ff, \left( \cos \psi\right)^2 \ff \rangle \leq 1$ in the second step.
In order to estimate the last expression, we use $r_- \gtrsim \abs{z}$ and $\normtwo{\Ff(z)} \leq 1$ by \eqref{eq:estimates_lambda} as well as \eqref{eq:first_estimate_m}, \eqref{eq:estimates_hat_lambda} and \eqref{eq:estimates_wh_f} 
to get for the first factor 
\begin{equation}
 \inf_{x=1,\ldots, n+p} \ff_x^4 \gtrsim r_-^{8L+8} r_+^{-16} \gtrsim_\delta \langle \Im \Mf \rangle^{16}. 
\label{eq:estimate_norm_B_inverse_4}
\end{equation}
To estimate the last factor in \eqref{eq:estimate_norm_B_inverse_2}, we multiply the real part of \eqref{eq:combined_QVE} with $\abs{\Mf}$ and obtain 
\[ (1 + \Ff) \frac{\Re \Mf}{\abs{\Mf}} = - \tau \abs{\Mf} \]
if $z = \tau + \i \eta$ for $\tau, \eta \in \R$. 
Estimating $\norm{\cdot}_2$ of the last equation yields 
\[ \abs{\tau} \norm{\Mf}_2 \leq 2 \norma{ \frac{\Re \Mf }{\abs \Mf}}_2  \]
by \eqref{eq:estimates_lambda}.
As $\norm{\Mf}_2 \geq \norm{\Im \Mf}_2 \geq \langle \Im \Mf \rangle$ we get 
\begin{equation}
2 \norma{ \frac{\Re \Mf }{\abs \Mf}}_2 \geq \abs{\tau}  \langle \Im \Mf \rangle. 
\label{eq:estimate_two_norm_real_part}
\end{equation}

Finally, we use \eqref{eq:estimate_norm_B_inverse_4} for the first factor in \eqref{eq:estimate_norm_B_inverse_2} and \eqref{eq:estimate_two_norm_real_part} for the last factor and apply the last 
estimate in \eqref{eq:first_estimate_m} 
and Jensen's inequality, $\avg{(\Im \Mf)^2} \geq \avg{\Im \Mf}^2$, to estimate the second factor which yields 
\begin{equation}
 \absa{1-\normtwo{F(z)^tF(z)} \scalar{v_1}{U_1 v_1} \scalar{v_2}{U_2v_2}}  \gtrsim_\delta \abs{\tau}^2 \langle \Im \Mf \rangle^{\kappa}.
\label{eq:estimate_norm_B_inverse_3}
\end{equation}
This completes the proof of \eqref{eq:first_estimate_B_inverse}.

Next, we bound $\Gap(F(z)F(z)^t)$ from below by applying Lemma \ref{lem:hat_F} with $r_- \defeq \inf_x \abs{\Mf_x(z)}$ and $r_+\defeq \norm{\Mf(z)}_\infty$. As $F(z) = \wh F(\abs{\Mf(z)})$ we have  
\[ \Gap(F(z)F(z)^t) \gtrsim_\delta \langle \Im \Mf(z) \rangle ^{16}, \]
where we used the estimates in \eqref{eq:first_estimate_m} and \eqref{eq:estimates_lambda}.
Combining this estimate on $\Gap(F(z)F(z)^t)$ with \eqref{eq:first_estimate_B_inverse} and \eqref{eq:estimate_norm_F} and increasing $\kappa$, we obtain 
\begin{align*}
\normtwo{\Bf^{-1}(z)} & \lesssim_\delta \min\left\{ \frac{1}{(\Re z)^2 \langle \Im \Mf(z)\rangle^{\kappa}} , \frac{1}{\Im z}, \frac{1}{\dist(\Re z, \supp \rho)^2}   \right\} 
\end{align*}
as $\normtwo{\Bf^{-1}(z)} \leq (1-\normtwo{\Ff(z)})^{-1}$ and $\delta \leq \abs{z} \leq 10$.
\end{proof}

\begin{lem}[Continuity of the solution] \label{lem:continuity_mf}
If $\Mf$ is the solution of the QVE \eqref{eq:combined_QVE} then $z \mapsto \langle \Mf(z) \rangle$ can be extended to a locally Hölder-continuous function on $\overline{\Hb} \backslash \{0\}$.
Moreover, for every $\delta>0$ there is a constant $c$ depending on $\delta$ and the model parameters such that 
\begin{equation}
\abs{ \langle \Mf (z_1)\rangle -\langle \Mf(z_2)\rangle} \leq c \abs{z_1-z_2}^{1/(\kappa +1)} 
\end{equation}
for all $z_1, z_2 \in \overline{\Hb} \backslash \{0\}$ such that $\delta \leq \abs{z_1}, \abs{z_2} \leq 10$
where $\kappa$ is the universal constant of Lemma \ref{lem:bound_B_inverse}. 
\end{lem}

\begin{proof}
In a first step, we prove that $z \mapsto \langle \Im \Mf(z)\rangle$ is locally Hölder-continuous. 
Taking the derivative of \eqref{eq:combined_QVE} with respect to $z \in \Hb$ yields 
\[ (1-\Mf^2(z)\Sf) \partial_z \Mf(z) = \Mf(z)^2. \]
By using that $\partial_z \phi = \i 2 \partial_z \Im \phi$ for every analytic function $\phi$ and taking the average, we get 
\[ \i 2 \partial_z \langle\Im \Mf \rangle  = \langle \abs{\Mf}, \Bf^{-1} \abs{\Mf}\rangle. \]
Here, we suppressed the $z$-dependence of $\Bf^{-1}$. We apply Cauchy-Schwarz inequality and use \eqref{eq:L2_bound}, \eqref{eq:estimate_norm_B_inverse} and \eqref{eq:first_estimate_m} to obtain 
\[ \abs{\partial_z \langle \Im \Mf \rangle} \leq \norm{\Mf}_2 \norm{\Bf^{-1}}_{2\to 2} \norm{\Mf}_2 \lesssim_\delta \min\{ (\Re z)^{-2}\langle \Im \Mf \rangle^{-\kappa}, (\Im z)^{-1}\} \lesssim_\delta 
\langle \Im \Mf \rangle^{-\kappa} \]
for all $z \in \Hb$ satisfying $\delta\leq \abs{z} \leq 10$. 
This implies that $z \mapsto \langle \Im \Mf(z) \rangle$ is Hölder-continuous with Hölder-exponent $1/(\kappa +1)$
on $z \in \Hb$ satisfying $\delta \leq \abs{z} \leq 10$. Moreover, it has a unique continuous extension to $I_\delta \defeq \{ \tau \in \R; \delta/3 \leq \abs{\tau} \leq 10\}$.
Multiplying this continuous function on $I_\delta$ by $\pi^{-1}$ yields a Lebesgue-density of the measure $\rho$ (cf. \eqref{eq:def_rho})  
restricted to $I_\delta$. 

We conclude that the Stieltjes transform $\langle \Mf \rangle$ has the same regularity by decomposing $\rho$ into a measure supported around zero and a measure supported away from zero and 
using Lemma A.7 in \cite{AjankiQVE}.
\end{proof}

For estimating the difference between the solution $\Mf$ of the QVE and a solution $\gf$ of the perturbed QVE \eqref{eq:perturbed_combined_QVE}, we introduce the deterministic control parameter
\[ \vartheta(z) \defeq \avg{\Im \Mf(z)} + \dist(z, \supp \rho), \quad z \in \Hb. \]

\begin{lem}[Stability of the QVE] \label{lem:bulk_stability_qve}
Let $\delta\gtrsim 1$. 
Suppose there are some functions $\df \colon \Hb \to \C^{p+n}$ and $\gf\colon \Hb \to (\C\backslash \{0\})^{n+p}$ satisfying \eqref{eq:perturbed_combined_QVE}.
Then there exist universal constants $\kappa_1, \kappa_2 \in \N$ and a function $\lambda_*\colon \Hb \to (0,\infty)$, independent of $n$ and $p$, 
such that $\lambda_*(10 \i) \geq 1/5$, $\lambda_*(z) \gtrsim_{\delta} \vartheta(z)^{\kappa_1}$ and 
\begin{equation} \label{eq:stability_g-m}
\norm{\gf(z) - \Mf(z)}_\infty \mathbf 1\Big(\norm{\gf(z) - \Mf(z)}_\infty \leq \lambda_*(z)\Big) \lesssim_\delta \vartheta(z)^{-\kappa_2} \norm{\df(z)}_\infty
\end{equation}
for all $z \in \Hb$ satisfying $\delta \leq \abs{z} \leq 10$. 
Moreover, there are a universal constant $\kappa_3 \in \N$ and a matrix-valued function $T \colon \Hb \to \C^{(p+n)\times (p+n)}$, depending only on $S$ and 
satisfying $\norm{T(z)}_{\infty\to \infty} \lesssim 1$, such that 
\begin{equation} \label{eq:stability_average_g-m}
\abs{\langle w, \gf(z)-\Mf(z)\rangle} \cdot \mathbf 1\Big(\norm{\gf(z) - \Mf(z)}_\infty \leq \lambda_*(z)\Big) \lesssim_\delta \vartheta(z)^{-\kappa_3} \left(\norm{w}_\infty \norm{\df(z)}_\infty^2 + \abs{\langle T(z)w, 
\df(z)\rangle}\right)
\end{equation}
for all $w \in \C^{p+n}$ and $z \in \Hb$ satisfying $\delta \leq \abs{z} \leq 10$.
\end{lem}

\begin{proof}
We set $\Phi(z) \defeq \max\{1, \norm{\Mf(z)}_\infty\}$, $\Psi(z) \defeq \max\{1, \norminf{\Bf^{-1}(z)}\}$ and $\lambda_*(z) \defeq (2 \Phi\Psi)^{-1}$. 
As $\Phi(z) \leq \max\{1, (\Im z)^{-1}\}$ and $\norminf{\Bf^{-1}(z)} \leq (1-\norminf{\Ff(z)})^{-1} \leq (1-(\Im z)^{-2})^{-1}$ due to $\norm{\Mf(z)}_\infty \leq (\Im z)^{-1}$ we obtain $\lambda_*(10 \i) \geq 1/5$. 
Since $\delta \leq \abs{z}$ we obtain $\langle \Im \Mf(z) \rangle^{-1} \gtrsim_\delta 1$ by \eqref{eq:first_estimate_m}. Thus, for $z \in \Hb$ satisfying $\delta \leq \abs{z} \leq 10$ 
the first estimate in \eqref{eq:mf_upper_bound_supp_rho1}, the last estimate in \eqref{eq:first_estimate_m} and \eqref{eq:estimate_norm_B_inverse_infty} yield 
\[ \Phi \lesssim_\delta \vartheta^{-1}, \quad \Psi \lesssim_\delta \vartheta^{-\kappa-2}, \]
where $\kappa$ is the universal constant from Lemma \ref{lem:bound_B_inverse}. Therefore, $\lambda_*(z) \gtrsim_\delta \vartheta(z)^{\kappa+3}$  
and Lemma 5.11 in \cite{AjankiQVE} yield the assertion as $\norm{w}_1 = (p+n)^{-1} \sum_{i} \abs{w_i} \leq \norm{w}_\infty$. 
\end{proof}

\subsection{Proof of Theorem \ref{thm:XX_star_Stieltjes_transform} }
\label{subsec:proof_of_existence_and_structure}

\begin{proof}[Proof of Theorem \ref{thm:XX_star_Stieltjes_transform}]
We start by proving the existence of the solution $m$ of \eqref{eq:m_equation}. Let $\Mf=(M_1, M_2)^t$ be the solution of \eqref{eq:combined_QVE} satisfying $\Im \Mf(z) >0$ for $z \in \Hb$.
For $\zeta \in \Hb$, we set $m(\zeta) \defeq M_1(\sqrt{\zeta})/\sqrt \zeta$. Then it is straightforward to check that $m$ satisfies \eqref{eq:m_equation} by solving \eqref{eq:QVE_m2} for $M_2$ 
and plugging the result into \eqref{eq:QVE_m1}.
Note that $\Im m(\zeta)>0$ for all $\zeta \in \Hb$ since 
 $M_{1,i}$ is the Stieltjes transform of a symmetric measure on $\R$ (cf. the explanation before \eqref{eq:L2_bound} for the symmetry of this measure). 

Next, we show the uniqueness of the solution $m$ of \eqref{eq:m_equation} with $\Im m(\zeta) >0$ for $\zeta \in \Hb$ which is a consequence of the uniqueness of the solution of \eqref{eq:combined_QVE}. 
Therefore, we set $m_1(\zeta) \defeq m(\zeta)$, $m_2(\zeta) \defeq -1/(\zeta(1+S^t m_1(\zeta)))$ and $\mf(\zeta) \defeq (m_1(\zeta), m_2(\zeta))^t$ for $\zeta \in \Hb$. 
From \eqref{eq:m_equation}, we see that 
\begin{equation} \label{eq:estimate_m1_aux}
\abs{m_1} = \frac{1}{\absa{\zeta - S \frac{1}{1+S^tm_1}}} \leq \frac{1}{ \Im \zeta + S \frac{1}{\abs{1+S^tm_1}} S^t \Im m_1}  \leq \frac{1}{\Im \zeta}
\end{equation}
for all $\zeta \in \Hb$. Since $m_2$ satisfies 
\begin{equation}\label{eq:m_2_equation}
 -\frac{1}{m_2(\zeta)} = \zeta + S^t \frac{1}{1 + Sm_2}(\zeta)
\end{equation}
for $\zeta \in \Hb$, a similar argument yields $\abs{m_2} \leq (\Im \zeta)^{-1}$. Combining these two estimates, we obtain $\abs{\mf (\zeta)} \leq (\Im \zeta)^{-1}$ for all $\zeta \in \Hb$. 
Therefore, multiplying \eqref{eq:m_equation} and \eqref{eq:m_2_equation} with $m_1$ and $m_2$, respectively, yields 
\begin{equation*}
\abs{1 + \i \xi \mf_x(\i \xi)} \leq \norm{\mf(\i \xi)}_\infty \frac{1}{1-\norm{\mf(\i \xi)}_\infty} \leq \frac{1}{\xi - 1} \to 0
\end{equation*}
for $\xi \to \infty$ and $x =1, \ldots, n+p$ where we used $\abs{\mf (\zeta)} \leq (\Im \zeta)^{-1}$ in the last but one step. 
Thus, $\mf_x$ is the Stieltjes transform of a probability measure $\nu_x$ on $\R$ for all $x=1, \ldots, n+p$. 
Multiplying \eqref{eq:m_equation} by $m_1$, taking the imaginary part and averaging at $\zeta = \chi + \i \xi$, for $\chi \in \R$ and $\xi >0$, yields
\begin{align}
\chi \langle \Im m_1 \rangle  + \xi \avg{\Re m_1} & = - \scalara{\Re m_1}{S \frac{1}{\abs{1+S^t m_1}^2} S^t \Im m_1} + \scalara{\Im m_1}{S \frac{1}{\abs{1+S^tm_1}^2}(1 + S^t \Re m_1 )} \nonumber \\
 & = \scalara{\Im m_1}{S \frac{1}{\abs{1+S^tm_1}^2}}  \geq 0, \label{eq:support_in_nonnegative_reals} 
\end{align}
where we used the definition of the transposed matrix and the symmetry of the scalar product in the last step.
On the other hand, we have 
\[ \chi \avg{ \Im m_1 } + \xi \avg{\Re m_1} = \int_\R \frac{\xi t}{(t-\chi)^2 + \xi^2} \nu(\dd t). \]
Assuming that there is a $\chi <0$ such that $\chi \in \supp \nu$ we obtain that $\chi \avg{\Im m_1} + \xi\avg{\Re m_1} <0$ for $\xi\downarrow 0$ which contradicts \eqref{eq:support_in_nonnegative_reals}.
Therefore $\supp \nu_x \subset [0,\infty)$ for $x =1, \ldots, p$. 

Together with a similar argument for $m_2$, we get that $\supp \nu_x \subset [0, \infty)$ for all $x = 1, \ldots, n+p$. 
In particular, we can assume that $\mf$ is defined on $\C \setminus [0,\infty)$. 
We set $M_1 (z) \defeq z m_1(z^2)$, $M_2(z) \defeq z m_2(z^2)$ and $\Mf(z) \defeq (M_1(z), M_2(z))^t$ for all $z \in \Hb$. Hence, we get
\[ \Im \Mf_x(\tau + \i \eta) = \eta \int_{[0,\infty)} \frac{t + \tau^2 + \eta^2}{(t-\tau^2 + \eta^2)^2 + 4 \eta^2 \tau^2 } \nu_x(\di t) \]
as $\supp \nu_x \subset [0, \infty)$. This implies $\Im \Mf(z) >0$ for $z \in \Hb$ and thus the uniqueness of solutions of \eqref{eq:combined_QVE} with positive imaginary part implies the 
uniqueness of $m_1$. 

Finally, we verify the claim about the structure of the probability measure representing $\avg{m}$.
By Lemma \ref{lem:continuity_mf} and the statements following \eqref{eq:combined_QVE}, 
$\langle M_1 \rangle$ is the Stieltjes transform of $\poma \delta_0 + \rho_1(\omega) \dd \omega$ for some $\poma \in [0,1]$ and some symmetric Hölder-continuous 
function $\rho_1\colon \R\setminus \{0\} \to [0,\infty)$ whose support is contained in $[-2,2]$.
Therefore, $m$ is the Stieltjes transform of $\nu(\dd \omega) \defeq \poma \delta_0(\dd \omega) + \dens(\omega) \mathbf 1(\omega >0) \dd \omega$ where $\dens(\omega) = \omega^{-1/2} \rho_1(\omega^{1/2})$ for $\omega >0$. 
Thus, the support of $\nu$ is contained in $[0,4]$.
\end{proof}

\subsection{Square Gram matrices}

In this subsection, we study the stability of \eqref{eq:combined_QVE} for $n=p$. 
Here, we assume (A), \dimhard and \zerohard. These assumptions are strictly stronger than (A), (B) and (D) (cf. Remark \ref{rem:assumption_square_Gram}). 

For the following arguments, it is important that $\Mf$ is purely imaginary for $\Re z = 0$ as $\Mf( - \bar z ) = - \overline{\Mf(z)}$ for all $z \in \Hb$.
If we set
\begin{equation} \label{eq:Mf_purely_imaginary}
\vf(z) = \Im \Mf(z) 
\end{equation} 
for $z \in \Hb$, then $\vf$ fulfills 
\begin{equation} \label{eq:v_equation}
\frac{1}{\vf(\ii \eta)} = \eta + \Sf \vf(\ii \eta)
\end{equation}
for all $\eta \in (0,\infty)$ due to \eqref{eq:combined_QVE}. The study of this equation will imply the stability of the QVE at $z =0$. 
The following proposition is the main result of this subsection. 

\begin{pro} \label{pro:hard_edge_stability}
Let $n=p$, i.e., \dimhard holds true, and $S$ satisfies (A) as well as \zerohard.
\begin{enumerate}[(i)]
\item There exists a $\wh \delta \sim 1$ such that $\abs{\Mf(z)} \sim 1$ uniformly for all $z \in \Hb$ satisfying $\abs{z} \leq 10$ and $\Re z \in [-\wh \delta, \wh \delta]$. 
Moreover, $\avg{\Im \Mf(z)} \gtrsim 1$ for all $z \in \Hb$ satisfying $\abs{z} \leq 10$ and $\Re z \in [-\wh \delta, \wh \delta]$ and there is a $\vf(0) = (v_1(0), v_2(0))^t \in \R^p \oplus \R^p$ such that $\vf(0) \sim 1$ 
and  
\begin{equation*}
\ii \vf(0) = \lim_{\eta \downarrow 0} \Mf(\ii \eta).
\end{equation*}
\item (Stability of the QVE at $z=0$) 
Suppose that some functions $\df = (d_1, d_2)^t \colon \Hb \to \C^{p+p}$ and $\gf=(g_1, g_2)^t \colon \Hb \to (\C\backslash \{0\})^{p+p}$ satisfy \eqref{eq:perturbed_combined_QVE} and 
\begin{equation} \label{eq:assumption_stability_hard_edge_easy}
\avg{g_1(z)} = \avg{g_2(z)}
\end{equation}
for all $z \in \Hb$. There are numbers $\lambda_*, \wh \delta \gtrsim 1$, depending only on $S$, such that 
\begin{equation} \label{eq:stability_g-m_hard_edge}
\norm{\gf(z) - \Mf(z)}_\infty \mathbf 1\Big(\norm{\gf(z) - \Mf(z)}_\infty \leq \lambda_*\Big) \lesssim \norm{\df(z)}_\infty
\end{equation}
for all $z \in \Hb$ satisfying $\abs{z} \leq 10$ and $\Re z \in [-\wh \delta, \wh \delta]$. 
Moreover, there is a matrix-valued function $T \colon \Hb \to \C^{2p\times 2p}$, depending only on $S$ and satisfying $\norminf{T(z)} \lesssim 1$, such that 
\begin{equation} \label{eq:stability_average_g-m_hard_edge}
\abs{\langle w, \gf(z)-\Mf(z)\rangle} \cdot \mathbf 1\Big(\norm{\gf(z) - \Mf(z)}_\infty \leq \lambda_*\Big) \lesssim \norm{w}_\infty \norm{\df(z)}_\infty^2 + \abs{\langle T(z)w, \df(z)\rangle}
\end{equation}
for all $w \in \C^{2p}$ and $z \in \Hb$ satisfying $\abs{z} \leq 10$ and $\Re z \in [-\wh \delta, \wh \delta]$.
\end{enumerate}
\end{pro}

The remainder of this subsection will be devoted to the proof of this proposition. 
Therefore, we will always assume that (A), \dimhard and \zerohard are satisfied. 

\begin{lem} \label{lem:dad_problem}
The function $\vf\colon \ii (0,\infty) \to \R^{2p}$ defined in \eqref{eq:Mf_purely_imaginary} satisfies 
\begin{equation} \label{eq:estimate_v_im_Mf}
1 \lesssim \inf_{\eta \in (0,10]} \vf(\ii \eta) \leq \sup_{\eta >0 } \norm{\vf(\ii \eta)}_\infty \lesssim 1.
\end{equation}
If we write $\vf = (v_1, v_2)^t$ for $v_1, v_2 \colon \ii (0,\infty) \to \R^p$, then 
\begin{equation} \label{eq:avg_v_1_equals_avg_v_2}
\avg{v_1(\ii \eta)} = \avg{v_2(\ii \eta)}
\end{equation}
for all $\eta \in (0,\infty)$.
\end{lem}

The estimate in \eqref{eq:estimate_v_im_Mf}, with some minor modifications which we will explain next, is shown as in the proof of (6.30) of \cite{AjankiQVE}.

\begin{proof}
From \eqref{eq:v_equation} and the definition of $\Sf$, we obtain
$ \eta \avg{v_1} - \eta \avg{v_2} = \scalar{v_1}{Sv_2} - \scalar{v_2}{S^t v_1} = 0 $ for all $\eta \in (0,\infty)$ which proves \eqref{eq:avg_v_1_equals_avg_v_2}.
Differing from \cite{AjankiQVE}, the discrete functional $\wt J$ is defined as follows: 
 \begin{equation} \label{eq:def_wt_J}
\wt J(u ) = \frac{\varphi}{2K} \sum_{i,j = 1}^{2K} u(i) \mathcal Z_{ij} u(j) - \sum_{i=1}^{2K} \log u(i) 
\end{equation}
for $u \in (0,\infty)^{2K}$ (we used the notation $u(i)$ to denote the $i$-th entry of $u$) where $\mathcal Z$ is the $2K \times 2K$ matrix with entries in $\{0,1\}$ defined by 
\begin{equation}
\mathcal Z = \begin{pmatrix} 0 & Z \\ Z^t & 0 \end{pmatrix}.
\end{equation}
Decomposing $u = (u_1, u_2)^t$ for $u_1, u_2 \in (0,\infty)^K$ and writing $u_1(i)=u(i)$ and $u_2(j)=u(K+j)$ for their entries we obtain 
\begin{equation}
\wt J(u) = \frac{\varphi}{K} \sum_{i,j=1}^K u_1(i) Z_{ij} u_2(j) -\sum_{i=1}^K (\log u_1(i) +\log u_2(i)). 
\end{equation}

\begin{lem} \label{lem:max_bound_discrete_functional}
If $\Psi < \infty$ is a constant such that $u= (u_1, u_2)^t \in (0,\infty)^K \times (0,\infty)^K$ satisfies 
\[ \wt J (u ) \leq \Psi,\]
where $\wt J$ is defined in \eqref{eq:def_wt_J}, and $\avg{u_1} = \avg{u_2}$, then there is a constant $\Phi <\infty$ depending only on $(\Psi, \varphi, K)$ such that 
\[ \max_{k=1}^{2K} u(k)  \leq \Phi.\] 
\end{lem}

\begin{proof}
We define $\wt Z_{ij} \defeq Z_{i\sigma(j)}$ where $\sigma$ is a permutation of $\{1, \ldots, K\}$ such that $\wt Z_{ii} =1$ for all $i =1, \ldots, K$ 
where we use the FID property of $Z$. Moreover, we set $M_{ij} \defeq u_1(i) \wt Z_{ij} u_2(\sigma(j))$ and follow the proof of Lemma 6.10 in \cite{AjankiQVE} 
to obtain 
\[ u_1(i) u_2(\sigma(j)) \lesssim (M^{K-1})_{ij} \lesssim 1 \]
for all $i, j =1, \ldots, K$. Averaging over $i$ and $j$ yields 
\[ \avg{u_1}^2 = \avg{u_2}^2 \lesssim 1 \]
where we used $\avg{u_1} = \avg{u_2}$. This concludes the proof of Lemma \ref{lem:max_bound_discrete_functional}.
\end{proof}

Recalling the function $\vf$ in Lemma \ref{lem:dad_problem}, we set $u = (\avg{\vf}_1, \ldots, \avg{\vf}_{2K})$ with $\avg{\vf}_i = Kp^{-1} \sum_{x \in I_i} \vf_x$, where $I_i \defeq p + I_{i-K}$ for $i \geq K+1$.
Then we have $\avg{u_1} = \avg{u_2}$ by \eqref{eq:avg_v_1_equals_avg_v_2} and since $I_1, \ldots, I_{2K}$ is an equally sized partition of $\{1, \ldots, 2p\}$.
Therefore, the assumptions of Lemma \ref{lem:max_bound_discrete_functional} are met which implies \eqref{eq:estimate_v_im_Mf} of Lemma \ref{lem:dad_problem} as in \cite{AjankiQVE}.
\end{proof}

\newcommand{\oneminusone}{\ensuremath{\mathfrak{e}_-}}

We recall from Lemma \ref{eq:F_without_hat} that $\ff=(f_1, f_2)$ is the unique nonnegative, normalized eigenvector of $\Ff$ corresponding to the eigenvalue $\normtwo{\Ff}$. Moreover, we define
$\ff_-\defeq (f_1, -f_2)$ 
which clearly satisfies
\begin{equation} \label{eq:eigenvector_rel_ff_minus}
\Ff \ff_- = -\normtwo{\Ff} \ff_-.
\end{equation} 
Since the spectrum of $\Ff$ is symmetric, $\spec(\Ff) = - \spec(\Ff)$ with multiplicities, and $\normtwo{\Ff}$ is a simple eigenvalue of $\Ff$, the same is true for the eigenvalue $-\normtwo{\Ff}$ of $\Ff$ and $\ff_-$ spans its 
associated eigenspace.
We introduce 
\begin{equation} \label{eq:def_oneminusone}
\oneminusone \defeq \begin{pmatrix} 1 \\ -1 \end{pmatrix} \in \C^{p}\oplus \C^p. 
\end{equation}

\begin{lem} \label{lem:controling_derivative_Mf_on_imaginary_axis}
For $\eta \in (0,\infty)$, the derivative of $\Mf$ satisfies
\begin{equation} \label{eq:derivative_Mf_at_ii_eta}
 \Mf' (\ii \eta) = \frac{\di}{\di z} \Mf(\ii \eta) = - \vf(\ii \eta) (1+ \Ff(\ii \eta))^{-1} \vf(\ii \eta).
\end{equation}
Moreover, $\abs{\Mf'(\ii \eta)} \lesssim 1$ uniformly for $\eta \in (0,10]$.
\end{lem}

\begin{proof}
In the whole proof, the quantities $\vf$, $\ff$, $\ff_-$ and $\Ff$ are evaluated at $z=\ii\eta$ for $\eta>0$. Therefore, we will mostly suppress the $z$-dependence of all quantities.  
Differentiating \eqref{eq:combined_QVE} with respect to $z$ and using \eqref{eq:Mf_purely_imaginary} yields 
\[ -(1 + \Ff) \frac{\Mf'}{\vf}  = \vf. \]
As $ \normtwo{\Ff}<1$ by \eqref{eq:normtwo_F_explicit}, the matrix $(1+\Ff)$ is invertible which yields \eqref{eq:derivative_Mf_at_ii_eta} for all $\eta \in (0,\infty)$.

In order to prove $\abs{\Mf'(\ii \eta)} \lesssim 1$ uniformly for $\eta \in (0,\infty)$, we first prove that 
\begin{equation} \label{eq:scalar_v_f_minus}
 \abs{\avg{\ff_-(\ii \eta)\vf(\ii \eta)}} \leq O(\eta).
\end{equation}
We define the auxiliary operator 
$A \defeq \normtwo{\Ff} + \Ff = 1 + \Ff - \eta \frac{\avg{\ff \vf}}{\avg{\ff}}$ where we used \eqref{eq:normtwo_F_explicit} and \eqref{eq:Mf_purely_imaginary}. 
Note that 
\begin{equation} \label{eq:properties_A}
 A\ff_- = 0, \quad A\oneminusone = \oneminusone + \Ff \oneminusone -  \eta \frac{\avg{\ff v}}{\avg{\ff}} 
\oneminusone = O(\eta), 
\end{equation}
where we used $\Ff \oneminusone = - \oneminusone + \eta \begin{pmatrix} v_1 \\ - v_2 \end{pmatrix}$ which follows from \eqref{eq:combined_QVE} and the definition of $\Ff$. 

Defining $Qu \defeq u - \avg{\ff_- u} \ff_-$ for $u \in \C^{2p}$ and decomposing
\[ \oneminusone = \avga{\ff_- \oneminusone} \ff_- + Q\oneminusone\]
yield $AQ \oneminusone = O(\eta)$ because of \eqref{eq:properties_A}. 
As $\abs{\Mf(\ii\eta)} \sim 1$ by \eqref{eq:estimate_v_im_Mf} for $\eta \in (0,10]$ the bound \eqref{eq:Gap_FF^t} in Lemma \ref{lem:hat_F} implies that 
 there is an $\eps \sim 1$ such that for all $\eta \in(0,10]$ we have
\begin{equation} \label{eq:spec_Ff}
\spec(\Ff)\subset \{ -\normtwo{\Ff}\} \cup \left[ -\normtwo{\Ff} + \eps, \normtwo{\Ff} - \eps\right] \cup  \{ \normtwo{\Ff} \}. 
\end{equation}
Since $-\normtwo{\Ff}$ is a simple eigenvalue of $\Ff$ and \eqref{eq:eigenvector_rel_ff_minus} the symmetric matrix $A=\normtwo{\Ff} + \Ff$ 
is invertible on $\ff_-^\perp$ and $\normtwoa{\big(A|_{\ff_-^\perp}\big)^{-1}}= \eps^{-1} \sim 1$. 
As $\ff_- \perp Q\oneminusone$ we conclude $Q\oneminusone = O(\eta)$ and hence
\begin{equation}\label{eq:1-minus_avg_f}
 (1-\avg{\ff})(1+ \avg{\ff})=1 - \avg{\ff}^2 = 1- \avga{\ff_- \oneminusone}^2 = \norma{Q\oneminusone}_{2}^2 = O(\eta^2). 
\end{equation}
Thus, using \eqref{eq:avg_v_1_equals_avg_v_2} and \eqref{eq:1-minus_avg_f}, this implies 
\[\abs{\avg{\ff_-(\ii \eta)\vf(\ii \eta)}} = \absa{\avga{\vf \oneminusone} +  \avga{\vf \left[ \ff_- - \oneminusone\right] } }
\lesssim \norma{\ff_- - \oneminusone}_{2} = \sqrt{2(1- \avg{\ff})} = O(\eta), \]
which concludes the proof of \eqref{eq:scalar_v_f_minus}.

In \eqref{eq:derivative_Mf_at_ii_eta}, we decompose $\vf = \avg{\ff_- \vf} \ff_- + Q\vf$ and, 
using $\Ff\ff_-=-\normtwo{\Ff}\ff_-$ and \eqref{eq:normtwo_F_explicit}, 
we obtain 
\[ \Mf' = - \vf \frac{\avg{\ff_- \vf} }{\eta} \frac{\avg{\ff}}{\avg{\ff \vf}} \ff_- - \vf (1 + \Ff)^{-1} Q\vf. \]
Using \eqref{eq:spec_Ff}, we see that $\normtwo{(1+\Ff)^{-1}Q\vf} \sim 1$ uniformly for $\eta \in (0,10]$.
Together with $\avg{\ff_-(\ii \eta)\vf(\ii \eta)}= O(\eta)$ by \eqref{eq:scalar_v_f_minus}, this yields $\abs{\Mf'(\ii \eta)} \lesssim 1$ 
uniformly for $\eta\in(0,10]$.
\end{proof}

The previous lemma, \eqref{eq:v_equation} and Lemma \ref{lem:dad_problem} imply that $\vf(0) \defeq \lim_{\eta \downarrow 0} \vf(\ii \eta)$ exists and satisfies 
\begin{equation} \label{eq:properties_v_0}
\vf(0) \sim 1,\quad 1 = \vf(0) \Sf \vf(0) = \Ff(0)1, \quad \avg{v_1(0)} = \avg{v_2(0)},
\end{equation}
where $\vf(0) = (v_1(0), v_2(0))^t$.

In the next lemma, we establish an expansion of $\Mf(z)$ on the upper half-plane around $z = 0$. The proof of this result and later the stability estimates on $\gf - \Mf$ will be a consequence of the equation
\begin{equation} \label{eq:stability_relation}
 \Bf u = \ee^{-\ii \psi} u \Ff u + \ee^{-\ii \psi} \gf \df  
\end{equation}
where $u = (\gf - \Mf)/\abs{\Mf}$ and $\ee^{\ii \psi} = \Mf/\abs{\Mf}$. This quadratic equation in $u$ was derived in Lemma 5.8 in \cite{AjankiQVE}.

\begin{lem} \label{lem:expansion_Mf_around_zero}
For $z \in \Hb$, we have 
\begin{subequations}
\begin{align}
\Mf ( z ) & = \ii \vf(0) - z \vf(0) (1+ \Ff(0))^{-1} \vf(0)  + O(\abs{z}^2), \label{eq:expansion_Mf} \\
 \frac{\Mf(z)}{\abs{\Mf(z)}} & = \ii - (\Re z) (1 + \Ff(0))^{-1} \vf (0)  + O(\abs{z}^2).  \label{eq:expansion_phase_Mf}
\end{align}
\end{subequations}
In particular, there is a $\wh \delta \sim 1$ such that $\abs{\Mf(z)} \sim 1$ uniformly for $z \in \Hb$ satisfying $\Re z \in [-\wh \delta, \wh \delta]$ and $\abs{z} \leq 10$. 
Moreover, 
\begin{equation} \label{eq:estimate_eigenvector_F}
 \norm{\ff(z) -1}_{\infty} = O(\abs{z}), \quad \norma{\ff_-(z) -\oneminusone}_{\infty} = O(\abs{z}).
\end{equation}
\end{lem}

\begin{proof}
In order to prove \eqref{eq:expansion_Mf}, we consider \eqref{eq:combined_QVE} at $z$ as a perturbation of \eqref{eq:combined_QVE} at $z=0$ perturbed by $\df =z$ in the notation of \eqref{eq:perturbed_combined_QVE}. 
The solution of the unperturbed equation is $\Mf = \ii \vf(0)$. 
Following the notation of \eqref{eq:perturbed_combined_QVE}, we find that \eqref{eq:stability_relation} holds with $\gf = \Mf(z)$ and $u(z) = (\Mf(z) - \ii \vf(0))/\vf(0)$. 
We write $u(z) = \theta(z) \oneminusone  + w(z)$ with $ w \perp \oneminusone$. (We will suppress the
$z$-dependence in our notation.) Plugging this into \eqref{eq:stability_relation} and projecting onto $\oneminusone$ yields 
\begin{equation} \label{eq:theta_relation_with_w}
 \theta \avg{\vf(0)} = - \avga{\oneminusone \vf(0) w},
\end{equation}
where we used that $\Ff(0) 1 = 1$, i.e., $\avg{\Ff(0) w } = \avg{w}$, $\avg{\oneminusone w \Ff(0) w } = 0$ and $\avg{v_1(0)} = \avg{v_2(0)}$. Thus, we have $\theta = O(\norm{w}_{\infty})$ because of \eqref{eq:properties_v_0},
so that we conclude $ -(1+ \Ff(0)) w = z \vf(0) + O(\norm{w}_{\infty}^2 + \abs{z}\norm{w}_{\infty})$.
As $w$, $(1+\Ff(0))w$ and $\vf(0)$ are orthogonal to \oneminusone, the error term is also orthogonal to it which implies 
\begin{equation}\label{eq:relation_for_w}
w = - z (1 + \Ff(0))^{-1} \vf(0) + O(\abs{z}^2)
\end{equation}
using that $(1+\Ff(0))^{-1}$ is bounded on $\oneminusone^\perp$.

Observing that $\avg{M_1(z)} = \avg{M_2(z)}$ for $z \in \Hb$ by \eqref{eq:combined_QVE} and differentiating this relation yields $\avga{\Mf'(\ii \eta)  \oneminusone } = 0$ for all $\eta \in (0, \infty)$. 
Hence, \begin{equation} \label{eq:derivative_Mf_average}
\avg{\oneminusone \vf(0) (1+\Ff(0))^{-1} \vf(0)} = - \lim_{\eta \downarrow 0} \avga{\oneminusone \Mf'(\ii \eta)} = 0 
\end{equation}
by Lemma \ref{lem:controling_derivative_Mf_on_imaginary_axis}.

Plugging \eqref{eq:relation_for_w} into \eqref{eq:theta_relation_with_w}, we obtain 
\[ \theta \avg{\vf(0)} = \avg{\oneminusone \vf(0) (1+\Ff(0))^{-1} \vf(0)} + O(\abs{z}^2) = O(\abs{z}^2), \]
where we used \eqref{eq:derivative_Mf_average}. 
Hence, $\Mf(z) = \vf(0)(u + \ii \vf(0))$ concludes the proof of \eqref{eq:expansion_Mf} which immediately implies \eqref{eq:expansion_phase_Mf}. 

Using the expansion of $\Mf$ in \eqref{eq:expansion_Mf} in a similar argument as in the proof of $\normtwo{\ff_-(\ii \eta) - \oneminusone} = O(\eta)$ in Lemma \ref{lem:controling_derivative_Mf_on_imaginary_axis} yields  
\[ \normtwo{\ff(z) -1} = \normtwo{\ff_-(z) - \oneminusone} = O(\abs{z}). \]
Similarly, using \eqref{eq:norm_hat_F_inverse}, we obtain \eqref{eq:estimate_eigenvector_F}. 
\end{proof}

By a standard argument from perturbation theory and possibly reducing $\wh \delta \sim 1$, we can assume that $\Bf(z)$ has a unique eigenvalue $\beta(z)$ of smallest modulus for $z \in \Hb$ 
satisfying $\abs{\Re z} \leq \wh \delta$ and $\abs{z} \leq 10$ such that 
$\abs{\beta'} - \abs{\beta} \gtrsim 1$ for $\beta' \in \sigma(\Bf(z))$ and $\beta' \neq \beta$. This follows from $\abs{\Mf} \sim 1$ and thus $\Gap(F(z) F(z)^t) \gtrsim 1$ by Lemma \ref{lem:hat_F}.
For $z \in \Hb$ satisfying $\abs{\Re z} \leq \wh \delta$ and $\abs{z} \leq 10$, we therefore find a unique (unnormalized) vector $b(z) \in \C^{2p}$ such that $\Bf(z) b(z) = \beta(z) b(z)$ and $\scalar{\ff_-}{b(z)} = 1$.

We introduce the spectral projection $P$ onto the spectral subspace of $\Bf(z)$ in $(\C^{2p}, \norm{\cdot}_\infty)$ associated to $\beta(z)$  which fulfills the relation 
\[ P = \frac{\scalar{\bar b}{\cdot}}{\avg{b^2}} b. \]
Note that $P$ is not an orthogonal projection in general. 
Let $Q \defeq 1 - P$ denote the complementary projection onto the spectral subspace of $\Bf(z)$ not containing $\beta(z)$
(this $Q$ is different from the one in the proof of Lemma \ref{lem:controling_derivative_Mf_on_imaginary_axis}). 
Since $\Bf(z) = -1 - \Ff(z) + O(\abs{z})$ we obtain 
\begin{equation} \label{eq:b_close_to_oneminusone}
 \norma{b(z) - \oneminusone}_{\infty} = \big\lVert \overline{b(z)} - \oneminusone\big\rVert_{\infty} = O(\abs{z})
\end{equation}
for $z \in \Hb$ satisfying $\abs{\Re z} \leq \wh \delta$ and $\abs{z} \leq 10$.

\begin{lem} \label{lem:final_estimate_norm_B_inverse}
By possibly reducing $\wh \delta$ from Lemma \ref{lem:expansion_Mf_around_zero}, but still $\wh \delta \gtrsim 1$, we have 
\begin{equation} \label{eq:B_inverse_norm_hard_edge}
 \norminf{\Bf^{-1}(z)} \lesssim \frac{1}{\abs z}, \quad \norminf{\Bf^{-1}(z)Q} + \norminf{(\Bf^{-1}(z)Q)^*} \lesssim 1
\end{equation}
for $z \in \Hb$ satisfying $\abs{\Re z} \leq \wh \delta$ and $\abs{z} \leq 10$.
\end{lem}

\begin{proof}
Due to $\abs{\Mf(z)} \sim 1$ and using \eqref{eq:norm_hat_F_inverse} with $R= \Ff(z)$ and $D = \abs{\Mf(z)}^2/\Mf(z)^2$, it is enough to prove the estimates in \eqref{eq:B_inverse_norm_hard_edge} with $\norminf{\cdot}$ 
replaced by $\normtwo{\cdot}$. 
We first remark that $\abs{\Mf(z)} \sim 1$ and arguing similarly as in the proof of Lemma \ref{eq:F_without_hat} imply 
$\normtwo{\Bf^{-1}(z)} \lesssim (\Im z)^{-1}$. 

Now we prove $\normtwo{\Bf^{-1}(z)} \lesssim (\Re z)^{-1}$. We apply Lemma \ref{lem:bulk_stability} and recall $U_1 = \abs{M_1}^2/M_1^2$ and $U_2 = \abs{M_2}^2/M_2^2$ to get
\begin{equation} \label{eq:computation_imaginary_part_estimate_B_inverse_norm}
 \Im\left( 1-\normtwo{F(z)^tF(z)}\scalara{\frac{f_1}{\norm{f_1}_2}}{U_1\frac{f_1}{\norm{f_1}_2}}\scalara{\frac{f_2}{\norm{f_2}_2}}{U_2\frac{f_2}{\norm{f_2}_2}}\right) =  \frac{\normtwo{F(z)^tF(z)}}{\norm{f_1}_2 \norm{f_2}_2} 
\avg{\vf(0)} \Re z + O(\abs{z}^2),  
\end{equation}
where we used \eqref{eq:expansion_phase_Mf}, \eqref{eq:estimate_eigenvector_F} and $\norm{f_1}_2, \norm{f_2}_2, \normtwo{F(z)^tF(z)} \sim 1$. 
Since $\vf(0) \sim 1$ and $\Gap(F(z)F(z)^t) \gtrsim 1$ by Lemma \ref{lem:hat_F} and $\abs{\Mf(z)} \sim 1$, \eqref{eq:computation_imaginary_part_estimate_B_inverse_norm} and 
Lemma \ref{lem:bulk_stability} yield $\normtwo{\Bf^{-1}(z)} \lesssim (\Re z)^{-1}$ and hence $\normtwo{\Bf^{-1}(z)} \lesssim \min\{ (\Im z)^{-1}, (\Re z)^{-1} \} \lesssim \abs{z}^{-1}$.

The estimate $\norminf{\Bf^{-1}(z)Q} \lesssim 1$ in \eqref{eq:B_inverse_norm_hard_edge} follows from $\Gap(F(z)F(z)^t) \gtrsim 1$ by Lemma \ref{lem:hat_F}, $\abs{\Mf(z)} \sim 1$ and a standard argument from perturbation 
theory as presented in Lemma 8.1 of  \cite{AjankiQVE}. Here, it might be necessary to reduce $\wh \delta$. 
We remark that $\Bf^* = \abs{\Mf}^2/\overline{\Mf}^2 - \Ff$ and similarly $P^* = \scalar{b}{\cdot}/\avg{\overline{b}^2}\overline{b}$, i.e., $\Bf^*$ and $P^*$ emerge by the same constructions where $\Mf$ is replaced by $\overline{\Mf}$.
Therefore, we obtain $\norminf{(\Bf^{-1}(z)Q)^*} \lesssim 1$.
\end{proof}

\begin{proof}[Proof of Proposition \ref{pro:hard_edge_stability}] 
The part (i) follows from the previous lemmata. 

The part (ii) has already been proved for $\abs{z} \geq \delta$ in Lemma \ref{lem:bulk_stability_qve} and for any $\delta \gtrsim 1$. 
Therefore, we restrict ourselves to $\abs{z} \leq \delta$ for a sufficiently small $\delta\gtrsim1$. 
We recall $\ee^{\ii \psi} = \Mf/\abs{\Mf}$. 

 Owing to  Lemma \ref{lem:expansion_Mf_around_zero} and \eqref{eq:b_close_to_oneminusone}, there are positive 
constants $\delta, \Phi, \wh \Phi \sim 1$ which only depend on the model parameters such that 
\begin{equation} 
 \norminf{\Mf(z)} \leq \Phi, \qquad \norm{b(z) - \oneminusone}_2 \norminf{b} + \norminfa{\ee^{-\ii \psi} + \ii }\norminf{b}^2 \leq \wh \Phi\abs{\avg{b^2}}\abs{z}  
\label{eq:auxiliary_definitions_of_suprema} 
\end{equation}
for all $z \in \Hb$ satisfying $\abs{z} \leq \delta$. 
Here, we used $\normtwo{w} \leq \norminf{w}$ for all $w \in \C^{2p}$. 
Note that we employed \eqref{eq:b_close_to_oneminusone} for estimating $\normtwo{b-\oneminusone}$ as well as to obtain $\norminf{b}\sim 1$ and $\abs{\avg{b^2}} \sim 1$ for all $z \in \Hb$ satisfying 
$\abs{z}\leq \delta$ if $\delta \gtrsim 1$ is small enough. 

Lemma \ref{lem:final_estimate_norm_B_inverse} implies the existence of $\Psi, \wh \Psi \sim 1$ such that  
\begin{equation}
\norminf{\Bf^{-1}(z)} \leq \Psi \abs{z}^{-1},  \qquad \norminf{\Bf^{-1}(z)Q} \leq \wh \Psi  \label{eq:norm_inverse_Bf}
\end{equation}
for all $z \in \Hb$ satisfying $\abs{z}\leq \delta$ if $1 \lesssim \delta\leq \wh\delta$ is sufficiently small. 
  With these definitions, we set 
\begin{equation}\label{eq:def_lambda_star}
\lambda_* \defeq \frac{1}{2\Phi(\Psi \wh \Phi + \wh \Psi)}.
\end{equation}
The estimate on $h \defeq \gf(z) - \Mf(z)  = u \abs{\Mf}$ will be obtained from inverting $\Bf$ in \eqref{eq:stability_relation}. In order to control the right-hand side of \eqref{eq:stability_relation}, we decompose 
it, according to $1 = P + Q$, as 
\[ \ee^{-\ii \psi} u \Ff u = \frac{\avga{b \ee^{-\ii \psi} u \Ff u }}{\avg{b^2}} b + Q \ee^{-\ii \psi} u \Ff u, 
\quad \ee^{-\ii \psi} \gf \df = \frac{\avga{\ee^{-\ii \psi}\gf \df b}}{\avg{b^2}} b + Q \ee^{-\ii \psi} \gf \df. \]
Clearly, as $\norminf{\Sf}\leq 1$ we have $\norminf{(\Bf^{-1} Q)(\ee^{-\ii \psi} u \Ff u)} \leq \wh \Psi \norminf{h}^2$ and $\norminf{(\Bf^{-1} Q)(\ee^{-\ii \psi} \gf \df)} \leq
 \wh \Psi \norminf{\gf}\norminf{\df}$ due to \eqref{eq:norm_inverse_Bf}. 
Using $\avg{\oneminusone h\Sf h} = 0$ and \eqref{eq:auxiliary_definitions_of_suprema}, we obtain
\begin{equation*}
\norminfa{\avga{b \ee^{-\ii \psi} u \Ff u}\frac{b}{\avg{b^2}}} \leq \left(\abs{-\ii \avg{ h \Sf h \oneminusone}} + \abs{-\ii\avg{(b - \oneminusone) h \Sf h}} + \absa{ \avga{\left(\ee^{-\ii \psi} + \ii\right) b h \Sf h }} 
\right) \frac{\norminf{b}}{\abs{\avg{b^2}}} 
\leq \wh \Phi \abs{z}\norminf{h}^2. 
\end{equation*}
Similarly, due to \eqref{eq:auxiliary_definitions_of_suprema} and $\avg{\gf \df \oneminusone} = \avg{g_1(z) d_1(z)} - \avg{g_2(z) d_2(z)}= 0$ by the perturbed QVE \eqref{eq:perturbed_combined_QVE}, we get
\begin{equation*}
\norminfa{\avga{\ee^{-\ii \psi}\gf \df b}\frac{b}{\avg{b^2}}} \leq \left(\abs{\avg{\gf \df \oneminusone}} + \abs{\avg{(b - \oneminusone) \gf \df}} + \absa{ \avga{\left(\ee^{-\ii \psi} + \ii\right) b \gf \df}} 
\right) \frac{\norminf{b}}{\abs{\avg{b^2}}} 
\leq \wh \Phi \abs{z}\norminf{\gf}\norminf{\df}. 
\end{equation*}
Thus, inverting $\Bf$ in \eqref{eq:stability_relation}, multiplying the result with $\abs{\Mf}$, taking its norm and using \eqref{eq:norm_inverse_Bf} yield
\[ \norminf{h} \leq \Phi(\Psi \wh \Phi + \wh \Psi) \norminf{h}^2 + \Phi( \Psi \wh \Phi + \wh \Psi) \norminf{\gf} \norminf{\df}, \]
which implies 
\[ \norminf{h} \mathbf 1\Big(\norminf{h} \leq \lambda_*\Big) \leq \Phi (1 + 2 \Phi( \Psi \wh \Phi + \wh \Psi))\norminf{\df} \]
 by the definition of $\lambda_*$ in \eqref{eq:def_lambda_star}. This concludes the proof of \eqref{eq:stability_g-m_hard_edge}.

For the proof of \eqref{eq:stability_average_g-m_hard_edge}, inverting $\Bf$ in \eqref{eq:stability_relation} and taking the scalar product with $w$ yield 
\begin{align} 
 \scalar{w}{h} & = \scalar{w}{\Bf^{-1}(\ee^{-\ii \psi} h \Sf h) }+ \frac{\scalar{w}{\abs{\Mf}\Bf^{-1}b}}{\avg{b^2}}\avga{h\df \left[ ( \ee^{-\ii\psi} + \ii) b - \ii (b- \oneminusone) \right]}\nonumber \\
& + \scalar{(\Bf^{-1} Q)^*(\abs{\Mf} w)}{\ee^{-\ii \psi} h \df} + \scalar{Tw}{\df},\label{eq:expansion_averaged_form}
\end{align}
where we used $\avg{\oneminusone\gf \df} =0$ and set $Tw \defeq \avg{b^2}^{-1} \scalar{\abs{\Mf} \Bf^{-1}b}{w} \overline{\Mf} \left[ ( \ee^{\ii\psi} - \ii) \overline{b} 
+ \ii (\overline{b}- \oneminusone) \right] + \ee^{\ii \psi} \overline{\Mf} (\Bf^{-1}Q)^*(\abs{\Mf} w)$.

Using \eqref{eq:auxiliary_definitions_of_suprema} and \eqref{eq:norm_inverse_Bf} as well as a similar argument as in the proof of \eqref{eq:stability_g-m_hard_edge} for the first term in the definition of $T$ 
and $\norminf{(\Bf^{-1} Q)^*} \lesssim 1$ by \eqref{eq:B_inverse_norm_hard_edge} for the second term, we obtain $\norminf{T} \lesssim 1$. 
Moreover, as above we see that the first term on the right-hand side of \eqref{eq:expansion_averaged_form} is $\lesssim\norminf{w}\norminf{h}^2$. 
The estimates \eqref{eq:auxiliary_definitions_of_suprema} and \eqref{eq:norm_inverse_Bf} imply that the second term on the right-hand side of \eqref{eq:expansion_averaged_form} is $\lesssim \norminf{w}\norminf{h}\norminf{\df}$. 
Applying \eqref{eq:stability_g-m_hard_edge} to these bounds
yields \eqref{eq:stability_average_g-m_hard_edge}.
\end{proof}

\subsection{Properly rectangular Gram matrices}

In this subsection, we study the behaviour of $M_1$ and $M_2$ for $z$ close to zero for $p/n$ different from one. 
We establish that the density of the limiting distribution is zero around zero -- a well-known feature of the Marchenko-Pastur distribution for $p/n$ different from one.

We suppose that the assumptions (A), (C) and (D) are fulfilled and we will study the case $p>n$. More precisely, we assume that
\begin{equation}
\label{eq:p_bigger_than_n}
\frac{p}{n} \geq 1 +d_*
\end{equation}
for some $d_*>0$ which will imply that each component of $M_1$ diverges at $z=0$ whereas each component of $M_2$ stays bounded at $z=0$. 
Later, we will see that these properties carry over to $m_1$ and $m_2$.
We use the notation $D_\delta(w) \defeq \{ z \in \C\colon \abs{z-w} < \delta \}$ for $\delta>0$ and $w \in \C$. 

\begin{pro}[Solution of the QVE close to zero] \label{pro:QVE_close_to_zero}
If \zerorect and \eqref{eq:p_bigger_than_n} are satisfied then there exist a vector $u \in \C^p$, a constant $\delta_* \gtrsim 1$ and analytic functions 
$a\colon D_{\delta_*}(0) \to \C^p$, $b\colon D_{\delta_*}(0) \to \C^n$ such that the unique solution $\Mf=(\mon, \mtw)^t$ of 
\eqref{eq:combined_QVE} with $\Im \Mf >0$ fulfills 
\begin{equation}
M_1(z) = z a(z) - \frac{u}{z}, \qquad M_2 (z) = z b(z)
\label{eq:ansatz}
\end{equation}
for all $z \in D_{\delta_*}(0)\cap \Hb$. Moreover, we have 
\begin{enumerate}[(i)]
\item $\sum_{i=1}^p u_i = p-n$ and $ 1\lesssim u_i \leq 1$ for all $i=1, \ldots, p$, 
\item $b(0)=1/S^t u \sim 1$,
\item $\norm{a(z)}_\infty + \norm{b(z)}_\infty \lesssim 1$ uniformly for all $z \in D_{\delta_*}(0)$, 
\item $\lim_{\eta\downarrow 0} \Im M_1(\tau + \i \eta) = 0$ and $\lim_{\eta\downarrow 0} \Im M_2(\tau + \i \eta)=0$ locally uniformly for all $\tau \in (-\delta_*,\delta_*)\backslash\{0\}$.  
\end{enumerate}
\end{pro}

The ansatz \eqref{eq:ansatz} is motivated by the following heuristics. Considering $H$ as an operator $\C^p \oplus \C^n \to \C^n \oplus \C^p$, we expect that the first component 
described by $X^*\colon \C^p \to \C^n$ has a nontrivial kernel for dimensional reasons whereas the second component has a trivial kernel. Since the nonzero eigenvalues of $H^2$ correspond to the 
nonzero eigenvalues of $XX^*$ and $X^*X$, the Marchenko-Pastur distribution indicates that there is a constant $\delta_* \gtrsim 1$ such that $H$ has no nonzero eigenvalue in $(-\delta_*, \delta_*)$. 
As the first component $M_1$ of $\Mf$ corresponds to the ``first component'' of $H$, the term $-u/z$ in \eqref{eq:ansatz} implements the expected kernel. For dimensional reasons, 
the kernel should be $p-n$ dimensional which agrees with part (i) of Proposition \ref{pro:QVE_close_to_zero}. 
The factor $z$ in the terms $z a(z)$ and $z b(z)$ in \eqref{eq:ansatz} realizes the expected gap in the eigenvalue distribution around zero. 

\begin{proof}[Proof of Proposition \ref{pro:QVE_close_to_zero}]
We start with the defining equations for $u$ and $b$. We assume that $u \in (0,1]^p$ fulfills 
\begin{equation}
\frac{1}{u} = 1 + S \frac{1}{S^t u}
\label{eq:def_u}
\end{equation}
and $b\colon D_{\delta_*}(0) \to \C^p$ fulfills 
\begin{equation}
-\frac{1}{b(z)}  = z^2 - S^t \frac{1}{1+ S b(z)}  \label{eq:b_equation}
\end{equation}
for some $\delta_*>0$.
We then define $a \colon D_{\delta_*}(0) \to \C^p$ through 
\begin{equation}
z^2 a(z) = u - \frac{1}{1+ S b(z)}
\label{eq:z2u_v_Sw}
\end{equation}
and set $\wh M_1(z) \defeq z a(z) - u/z$ and $\wh M_2(z) \defeq zb(z)$ for $z\in D_{\delta_*}(0)$. 
Thus, for $z \in D_{\delta_*}(0)$, we obtain
\[ z + S^t \wh M_1(z) = z - S^t \frac{1}{1+Sb(z)} = - \frac{1}{zb(z)} = - \frac{1}{\wh M_2(z)}, \]
where we used \eqref{eq:z2u_v_Sw} in the first step and \eqref{eq:b_equation} in the second step. 
Similarly, solving \eqref{eq:z2u_v_Sw} for $Sb(z)$ yields
\begin{equation} \label{eq:wh_M_1_M_2}
 z + S \wh M_2(z) = z + z \left( \frac{1}{ u - z^2 a (z)} - 1\right) = - \frac{1}{\wh M_1(z)}, \qquad z \in D_{\delta_*}(0).
\end{equation}
Thus, $(\wh M_1, \wh M_2)$ satisfy \eqref{eq:combined_QVE}, the defining equation for $\Mf = (M_1, M_2)$ and we will be able to conclude that $\wh M_1 = M_1$ and $\wh M_2 = M_2$.

For the rigorous argument, we first establish the existence and uniqueness of $u$ and $b$ that follow from the next two lemmata whose proofs are given later. 

\begin{lem}  \label{lem:solution_u}
If \zerorect and \eqref{eq:p_bigger_than_n} are satisfied then there is a unique solution of \eqref{eq:def_u} in the set $u \in (0,1]^p$. Moreover, 
\begin{equation}
1 > u_i \gtrsim 1, \qquad (S^tu)_k \gtrsim 1
\label{eq:bound_u_St_u}
\end{equation} 
for all $i =1, \ldots, p$ and $k =1, \ldots, n$ and $\sum_{i=1}^p u_i = p-n$.
\end{lem}

\begin{lem} \label{lem:solution_b}
If \zerorect and \eqref{eq:p_bigger_than_n} are satisfied, then there are a $\delta_*\sim 1$ and a unique holomorphic function $b \colon D_{\delta_*}(0) \to \C^n$ satisfying \eqref{eq:b_equation}
with $b(0) = 1/(S^t u)$, where $u$ is the solution of \eqref{eq:def_u}. Moreover, we have $\norm{b(z)}_\infty \lesssim 1$ and $\norm{(1+Sb(z))^{-1}}_\infty\leq 1/2$ for all $z \in D_{\delta_*}(0)$, $b(0) \sim 1$, $b'(0)=0$,  
$\Im (zb(z)) >0$ for all $z \in D_{\delta_*}(0)$ with $\Im z>0$ and $\Im(z b(z)) =0$ for $z \in (- \delta_*,  \delta_*)$.
\end{lem}

Given $u$ and $b(z)$, the formula \eqref{eq:z2u_v_Sw} defines $a(z)$ for $z \neq 0$. To extend its definition to $z = 0$, 
we observe that the right-hand side of \eqref{eq:z2u_v_Sw} is a holomorphic function for all $z \in D_{\delta_*}(0)$ by Lemma \ref{lem:solution_b}. Since $b(0)=1/(S^tu)$ and the derivative of the 
right-hand side of \eqref{eq:z2u_v_Sw} vanishes as $b'(0)=0$, the first two coefficients of the Taylor series of the right-hand side on $D_{\delta_*}(0)$ are zero by \eqref{eq:def_u}. 
Thus, \eqref{eq:z2u_v_Sw} defines a holomorphic function $a\colon D_{\delta_*}(0) \to \C^p$. 

Furthermore, $\Im \wh M_2(z) >0$ for $\Im z >0$ by Lemma \ref{lem:solution_b}. Taking the imaginary part of \eqref{eq:wh_M_1_M_2} yields 
\begin{equation} \label{eq:QVE_m1_imaginary_part}
 \frac{\Im \wh M_1(z)}{\abs{\wh M_1(z)}^2} = \Im z + S \Im \wh M_2(z) ,
\end{equation}
which implies $\Im \wh M_1(z)>0$ for $\Im z >0$ as $\Im \wh M_2(z) >0$ for $z \in \Hb\cap D_{\delta_*}(0)$. 
Since the solution $\Mf$ of \eqref{eq:combined_QVE} 
with $\Im \Mf(z) >0$ for $\Im z>0$ is unique by Theorem 2.1 in \cite{AjankiQVE}, we have $\Mf(z) = \wh \Mf(z) \defeq (\wh M_1(z), \wh M_2(z))^t$ for all $z \in \Hb\cap D_{\delta_*}(0)$. 
The statements in (i), (ii) and (iii) follow from Lemma \ref{lem:solution_u}, Lemma \ref{lem:solution_b} and \eqref{eq:z2u_v_Sw}.

For the proof of (iv), we note that $\lim_{\eta \downarrow 0} \Im M_2(\tau + \i \eta) = 0$ for all $\tau \in (-\delta_*, \delta_*)$ locally uniformly by Lemma \ref{lem:solution_b}.
Because of \eqref{eq:QVE_m1_imaginary_part} and the locally uniform convergence of $M_1(\tau + \i \eta)$ to $\tau a(\tau)-u/\tau$ for $\eta \downarrow 0$ and $\tau \in (-\delta_*,\delta_*)\backslash \{ 0\}$, 
we have $\lim_{\eta \downarrow 0} \Im M_1(\tau + \i \eta)=0$ locally uniformly for all $\tau \in (-\delta_*,\delta_*) \backslash \{ 0\}$ as well, which concludes the proof of (iv). 
\end{proof}

We conclude this subsection with the proofs of Lemma \ref{lem:solution_u} and Lemma \ref{lem:solution_b}.

\begin{proof}[Proof of Lemma \ref{lem:solution_u}]
We will show that the functional 
\[ J \colon (0,1]^p \to \R, \quad u \mapsto \frac{1}{p} \sum_{j=1}^n\log\bigg(\sum_{i=1}^p s_{ij} u_i \bigg) + \frac{1}{p}\sum_{i=1}^p \left( u_i - \log u_i \right) \]
has a unique minimizer $u$ with $u_i >0$ for all $i = 1, \ldots, p$ which solves \eqref{eq:def_u}. 
Note that 
\begin{equation}
\label{eq:J_1_ldots_1}
J(1, \ldots, 1) = \frac{1}{p}\sum_{j=1}^n \log\bigg(\sum_{i=1}^p s_{ij}\bigg) + \frac{p}{p} \leq 1. 
\end{equation}

We start with an auxiliary bound on the components of $u$. Using \zerorect and Jensen's inequality, we get 
\begin{align}
J(u) & \geq \frac{1}{p} \sum_{k=1}^n \log\left( \sum_{i=1}^p \frac{\varphi}{n+p} u_i \right) + \frac{1}{p} \sum_{i=1}^p \left( u_i - \log u_i \right) \nonumber \\
 & \geq \frac{1}{p} \left( \sum_{i=1}^p \frac{n}{p}\log\left( \frac{\varphi}{2}u_i \right) -\sum_{i=1}^p \log u_i   \right) \nonumber \\
& \geq - \frac{1}{p} \frac{d_*}{1+d_*} \sum_{i=1}^p \log u_i + \frac{n}{p} \log \left( \frac{\varphi}{2} \right),
\end{align}
where we used \eqref{eq:p_bigger_than_n} in the last step.
For any $u \in (0,1]^p$ with $J(u) \leq J(1, \ldots, 1)$, using \eqref{eq:J_1_ldots_1}, we obtain 
\[ 1 \geq J(1, \ldots, 1) \geq J(u) \geq - \frac{d_*}{p (1 +d_*)} \sum_{i=1}^p \log u_i + \frac{n}{p} \log \left( \frac{\varphi}{2}\right) \geq -\frac{d_*}{p(1+d_*)} \log u_i +\frac{1}{r_1} \log\left(\frac\varphi 2\right), \] 
for any $i=1, \ldots, p$, i.e., $u_i \geq \exp(-p(1+d_*)(1-r_1^{-1}\log( \varphi/2))/d_*)>0$. 

Therefore, taking a minimizing sequence, using a compactness argument and the continuity of $J$, 
we obtain the existence of $u^\star \in (0,1]^p$ such that $J(u^\star) = \inf_{u \in (0,1]^p} J(u)$ and 
\begin{equation}
\label{eq:aux_lower_bound_minimizer}
u^\star_i \geq \exp\left(-p\frac{1+d_*}{d_*}\left(1-\frac{1}{r_1}\log\left( \frac{\varphi}{2} \right)\right)\right), \qquad i =1, \ldots, p. 
\end{equation} 
Next, we show that $u^\star_i<1$ for all $i=1, \ldots, p$. Assume that $u^\star_i=1$ for some $i \in \{1, \ldots, p\}$. Consider a vector $\wh u$ that agrees with $u^\star$ except that $u^\star_i$ is replaced 
by $\lambda \in (0,1)$. An elementary calculation then shows that $J(\wh u) \geq J(u^\star)$ implies $s_{ik} = 0$ for all $k=1, \ldots, n$ which contradicts \eqref{eq:sum_s_ij_sim_1}.  

Therefore, evaluating the derivative $J(u^\star + \tau h)$ for $h \in \R^p$ at $\tau = 0$, which vanishes since $u^\star\in (0,1)^p$ is a minimizer, we see that $u^\star$ satisfies \eqref{eq:def_u}. 

To see the uniqueness of the solution of \eqref{eq:def_u}, we suppose that $u^\star, v^\star \in (0,1]^p$ satisfy \eqref{eq:def_u}, i.e., $u^\star = f(u^\star)$ and $v^\star = f(v^\star)$ 
where $f\colon (0,1]^p \to (0,1]^p, \; f(u)=(1+ S((S^t u)^{-1}))^{-1}$. 
On $(0,1]^p$ we define the distance function 
\begin{equation}
D (u,v ) \defeq \sup_{i=1, \ldots, p} d(u_i, v_i) 
\end{equation}
where $d(a,b) = (a-b)^2/(ab)$ for $a, b >0$. 
This function $d$ defined on $(0,\infty)^2$ is the analogue of $D$ defined in (A.6) of \cite{AjankiCPAM} on $\Hb^2$. Therefore, we can apply Lemma A.2 in \cite{AjankiCPAM} with the natural 
substitutions which yields 
\[ 
D(u^\star, v^\star)=D(f(u^\star), f(v^\star)) = \left(1 + \frac{1}{S(S^t u^\star)^{-1}}\right)^{-1}\left(1 + \frac{1}{S(S^t v^\star)^{-1}}\right)^{-1} D(u^\star, v^\star) \leq c D(u^\star, v^\star).
\]
for some number $c$.
Here we used 1. and 2. of Lemma A.2 in \cite{AjankiCPAM} in the second step and 3. of Lemma A.2 in \cite{AjankiCPAM} in the last step. Since we can choose $c<1$ by \eqref{eq:aux_lower_bound_minimizer},
 we conclude $u^\star=v^\star$. This argument applies particularly to minimizers of $J$ on $(0,1]^p$. 

In the following, we will denote the unique minimizer of $J$ by $u$.
To compute the sum of the components of $u$ we multiply \eqref{eq:def_u} by $u$ and sum over $i=1, \ldots, p$ and obtain 
\[ p = \sum_{i=1}^p u_i + \sum_{i=1}^p u_i \left(S \frac{1}{S^tu}\right)_i = \sum_{i=1}^p u_i + \sum_{j=1}^n (S^t u)_j \frac{1}{(S^t u)_j}  = \sum_{i=1}^p u_i +n, \]
i.e., $\sum_{i=1}^p u_i = p-n$.

Finally, we show that the components of the minimizer $u$ are bounded from below by a positive constant which only depends on the model parameters. 
For $k \in \{1, \ldots, n\}$, we obtain  
\begin{equation}
 (S^t u)_k \geq \frac{\varphi}{n+p} \sum_{i=1}^{p} u_i \geq \frac{\varphi}{2} \langle u \rangle = \frac{\varphi}{2} \left(1- \frac{n}{p}\right) \geq \frac{\varphi d_*}{2(1+d_*)},
\label{eq:estimate_Stu_from_below}
\end{equation}
where we used \zerorect in the first step, $n \leq p$ in the second step, $\sum_{i=1}^p u_i = p-n$ in the third step and \eqref{eq:p_bigger_than_n} in the last step. 
This implies the third bound in \eqref{eq:bound_u_St_u}. 

Therefore, we obtain for all $i=1, \ldots, p$ from \eqref{eq:def_u} 
\[ \frac{1}{u_i} = 1 + \sum_{k=1}^n s_{ik} \frac{1}{(S^tu)_k} \leq 1 + \frac{2(1+d_*)}{\varphi d_*}, \] 
where we used (A) with $s_*=1$ in the last step. This shows that $u_i$ is bounded from below by a positive constant which only depends on the model parameters, i.e., the second bound in \eqref{eq:bound_u_St_u}.
\end{proof}

\begin{proof}[Proof of Lemma \ref{lem:solution_b}]
Instead of solving \eqref{eq:b_equation} directly, we solve a differential equation with the correctly chosen initial condition in order to obtain $b$. 
Note that $b_0 \defeq 1/(S^tu)$ fulfills \eqref{eq:b_equation} for $z=0$ and $b_0 \sim 1$ by \eqref{eq:bound_u_St_u} and \eqref{eq:sum_s_ij_sim_1}. 

For any $b \in \C^n$ satisfying $(Sb)_i \neq -1$ for $i=1, \ldots, p$,  we define the linear operator 
\[L(b) \colon \C^n \to \C^n, \quad v \mapsto L(b)v \defeq b S^t \frac{1}{(1+Sb)^2} S(bv),\] 
where $bv$ is understood as componentwise multiplication. 
Using the definition of $L(b)$, $b_0=1/(S^tu)$ and \eqref{eq:def_u}, we get 
\begin{equation} \label{eq:L_b_0_1}
L(b_0)1 = \frac{1}{S^tu}S^t u^2 S \frac{1}{S^tu} = \frac{1}{S^tu}\left(S^tu - S^tu^2\right) = 1- \frac{S^tu ^2}{S^tu} \leq 1 - \kappa 
\end{equation}
for some $\kappa \sim 1$. Here we used \eqref{eq:sum_s_ij_sim_1}, $u^2 \gtrsim 1$ and \eqref{eq:bound_u_St_u} in the last step. 
As \[L(b_0)=\frac{1}{S^tu} S^t u^2 S\left(\frac{1}{S^tu} \:\cdot\right)\] is symmetric and positivity-preserving, 
Lemma 4.6 in \cite{AjankiQVE} implies $\norm{L(b_0)}_{2\to 2} \leq 1-\kappa$ because of \eqref{eq:L_b_0_1}. 
Therefore, $(1-L(b_0))$ is invertible and $\norm{(1-L(b_0))^{-1}}_{2\to 2} \leq \kappa^{-1}$. 
Moreover, $\norm{(1-L(b_0))^{-1}}_{\infty \to \infty} \leq 1 + \norm{L(b_0)}_{2\to \infty} \kappa^{-1}$ by \eqref{eq:norm_hat_F_inverse} with $R= L(b_0)$ and $D =1$. 
The estimate \eqref{eq:St_norm_two_inf_less_1} and the submultiplicativity of the operator norm $\normtwo{\cdot}$ yield $\normtwoinf{L(b_0)} \lesssim 1$. Thus, we obtain \[ \norminf{(1-L(b_0))^{-1}} \lesssim 1.\]

We introduce the notation $U_{\delta'} \defeq \{b \in \C^n; \norm{b-b_0}_\infty < \delta' \}$. If we choose $\delta' \leq (2 \norm{S}_{\infty \to \infty})^{-1}$ then 
\[ \abs{(1+Sb)_i} = \abs{u_i^{-1} + (S(b-b_0))_i} \geq \abs{u_i^{-1}} - \norm{S}_{\infty \to \infty} \norm{b-b_0}_\infty \geq 1/2 \]
for all $i=1,\ldots, p$, where we used the definition of $b_0$, \eqref{eq:def_u} and $u_i \leq 1$. Therefore, $\norm{(1+Sb)^{-1}}_\infty \leq 1/2$ for all $b \in U_{\delta'}$, i.e., 
$U_{\delta'} \to \C^{n \times n}, \; b \mapsto L(b)$ will be a holomorphic map. In particular, 
\begin{equation} \label{eq:L_lipschitz}
\norminf{L(b)-L(b_0)} \lesssim \norm{b-b_0}_\infty. 
\end{equation}

If $D \defeq L(b)-L(b_0)$ and $\norm{(1-L(b_0))^{-1}D}_{\infty\to\infty} \leq 1/2$ then $(1-L(b))$ will be invertible and
\begin{equation*}
(1-L(b))^{-1} = \left(1-(1-L(b_0))^{-1}D\right)^{-1}(1-L(b_0))^{-1},
\end{equation*}
as well as $\norm{(1-L(b))^{-1}}_{\infty\to\infty} \leq 2 \norm{(1-L(b_0))^{-1}}_{\infty\to\infty}$. 
Therefore, \eqref{eq:L_lipschitz} implies the existence of $\delta'\sim 1$ such that $(1-L(b))$ is invertible and $\norminf{(1-L(b))^{-1}} \lesssim 1$ for all $b \in U_{\delta'}$.

Hence, the right-hand side of the differential equation 
\begin{equation}
b' \defeq \frac{\partial}{\partial z} b = 2z b (1-L(b))^{-1} b  \defqe f(z, b) 
\label{eq:final_ode}
\end{equation}
is holomorphic on $D_{\delta'}(0) \times U_{\delta'}$. 
As $\delta' \sim 1$ and $\sup\{ \norm{f(z,w)}_\infty; z \in D_{\delta'}(0), b \in U_{\delta'} \} \lesssim 1$,
the standard theory of holomorphic differential equations yields the existence of $\delta_* \gtrsim 1$ and a holomorphic function $b\colon D_{\delta_*}(0) \to \C^n$ which is the unique solution of \eqref{eq:final_ode} 
on $D_{\delta_*}(0)$ satisfying $b(0)=b_0$. 

The solution of the differential equation \eqref{eq:final_ode} is a solution of \eqref{eq:b_equation} since 
dividing by $b$, multiplying by $(1-L(b))$ and dividing by $b$ in \eqref{eq:final_ode} yields 
\begin{equation*}
 \frac{b'}{b^2} = 2 z  + \frac{1}{b}L(b) \frac{b'}{b}. 
\label{eq:ode_b}
\end{equation*}
This is the derivative of \eqref{eq:b_equation}. Since $b(0)=b_0$ fulfils \eqref{eq:b_equation} for $z=0$ the unique solution of \eqref{eq:final_ode} with this initial condition 
is a solution of \eqref{eq:b_equation} for $z \in D_{\delta_*}(0)$. There is only one holomorphic solution of \eqref{eq:b_equation} due to the uniqueness of the solution of \eqref{eq:final_ode}. 
This proves the existence and uniqueness of $b(z)$ in Lemma \ref{lem:solution_b}.

Since $b$ is a holomorphic function on $D_{\delta_*}(0)$ such that $\abs{b(z)} \lesssim 1$ on $D_{\delta_*}(0)$ and $\delta_* \sim 1$ there is a holomorphic function $b_1 \colon D_{\delta_*}(0) \to \C^n$ such that  
\[ b(z) = b_0 + b_1(z)z \]
and $\abs{b_1(z)} \lesssim 1$. 
Thus, we can assume that $\delta_*\gtrsim 1$ is small enough such that $\Im zb(z) \geq (b_0 - \abs{z}\abs{b_1(z)})\Im z >0$ for all $z \in D_{\delta_*}(0)\cap \Hb$.

Taking the imaginary part of \eqref{eq:b_equation} for $\tau \in \R$, we get
\begin{equation*}
\frac{\Im b(\tau)}{\abs{b(\tau)}^2} = S^t \frac{1}{\abs{1+Sb(\tau)}^2} S \Im b(\tau) 
\end{equation*}
or equivalently, introducing \[\wt L(z)\colon \C^n \to \C^n, \quad v\mapsto \wt L(z)v \defeq \abs{b(z)} S^t \abs{1+Sb(z)}^{-2} S(\abs{b(z)}v)\] for $z \in D_{\delta_*}(0)$, we have
\begin{equation}
\left(1-\wt L(\tau)\right) \frac{\Im b(\tau)}{\abs{b(\tau)}} = 0.
\label{eq:im_b_zero}
\end{equation}
As $\norm{(1+Sb(z))^{-1}}_\infty \leq 1/2$ for all $z \in D_{\delta_*}(0)$, the linear operator $\wt L(z)$ is well-defined for all $z \in D_{\delta_*}(0)$. 
Because $\wt L(0)=L(b_0)$ and $\norminf{\wt L(b)-\wt L(b_0)}\lesssim \norm{b-b_0}_\infty$ we can assume that $\delta_*\gtrsim 1$ is small enough such that 
$(1-\wt L(z))$ is invertible for all $z \in D_{\delta_*}(0)$. 
Thus, \eqref{eq:im_b_zero} implies that $\Im b(\tau) = 0$ for all $\tau \in (-\delta_*, \delta_*)$ and consequently, $\Im \tau b(\tau) = 0$ for all $\tau \in (-\delta_*, \delta_*)$.
\end{proof}

\section{Local laws}
\subsection{Local law for $\Xf$}
\label{subsec:local_law_H}

In this section, we will follow the approach used in \cite{Ajankirandommatrix} to prove a local law for the Wigner-type matrix $\Xf$. 
We will not give all details but refer the reader to \cite{Ajankirandommatrix}.
Therefore, we consider \eqref{eq:self_const_split} as a perturbed QVE of the form \eqref{eq:perturbed_combined_QVE} with $\gf\defeq (g_1, g_2)^t \colon \Hb \to \C^{p+n}$ and $\df \defeq (d_1, d_2)^t \colon \Hb \to \C^{p+n}$,
in particular $\gf(z) =(G_{xx}(z))_{x=1, \ldots, n+p}$ where $G_{xx}$ are the diagonal entries of the resolvent of $H$ defined in \eqref{eq:def_G}. 
We recall that $\rho$ is the probability measure on $\R$ whose Stieltjes transform is $\avg{\Mf}$, cf. \eqref{eq:def_rho}, where $\Mf$ is the solution of \eqref{eq:combined_QVE} satisfying $\Im \Mf(z) >0$ for $z \in \Hb$.

\begin{defi}[Stochastic domination]
Let $P_0 \colon (0,\infty)^2 \to \N$ be a given function which depends only on the model parameters and the tolerance exponent $\gamma$. If $\varphi = (\varphi^{(p)})_{p}$ and $\psi = (\psi^{(p)})_{p}$ 
are two sequences of nonnegative random variables then we will say that $\varphi$ is \textbf{stochastically dominated} by $\psi$, $\varphi \prec \psi$, if for all $\eps >0$ and $D>0$ we have 
\[ \P \left( \varphi^{(p)} \geq p^\eps \psi^{(p)} \right) \leq p^{-D} \]
for all $ p \geq P_0(\eps,D)$.
\end{defi}

In the following, we will use the convention that $\tau \defeq \Re z$ and $\eta \defeq \Im z$ for $ z \in \C$.

\begin{thm}[Local law for $\Xf$ away from the edges] \label{thm:local_law_H}
Fix any $\delta, \eps_*>0$ and $\gamma \in (0,1)$ independent of $p$. 
If the random matrix $X$ satisfies (A) -- (D) then the resolvent entries $G_{xy}(z)$ of $\Xf$ defined in \eqref{eq:def_G} and \eqref{eq:def_H_S}, respectively, 
fulfill 
\begin{subequations} \label{eq:local_law_H}
\begin{align}
\max_{x,y=1, \ldots, n+p} \abs{G_{xy}(z) - \Mf_x(z) \delta_{xy}}  & \prec \frac{1}{\sqrt{p\eta}}, & & \text{ if } \Im z \geq p^{-1 + \gamma} \text{ and } \langle \Im \Mf(z) \rangle \geq \eps_*, \label{eq:local_law_H_bulk} \\
\max_{x,y=1, \ldots, n+p} \abs{G_{xy}(z) - \Mf_x(z) \delta_{xy}}  & \prec \frac{1}{\sqrt{p}}, & &  \text{ if } \dist(z, \supp \rho) \geq \eps_*, \label{eq:local_law_H_away_support}
\end{align}
\end{subequations}
uniformly for $z \in \Hb$ satisfying $\delta \leq \abs{z} \leq 10$. 
For any sequence of deterministic vectors $w \in \C^{n+p}$ satisfying $\norm{w}_\infty \leq 1$, we have 
\begin{subequations} \label{eq:averaged_local_law_H}
\begin{align}
\abs{\scalar{w}{\gf(z)-\Mf(z)}} & \prec \frac{1}{p\eta}, & & \text{ if } \Im z \geq p^{-1 + \gamma} \text{ and } \langle \Im \Mf(z) \rangle \geq \eps_*, \label{eq:averaged_local_law_H_bulk} \\
\abs{\scalar{w}{\gf(z)-\Mf(z)}} & \prec \frac{1}{p}, & &  \text{ if } \dist(z, \supp \rho) \geq \eps_*, \label{eq:averaged_local_law_H_away_support}
\end{align}
\end{subequations}
uniformly for $z \in \Hb$ satisfying $\delta \leq \abs{z} \leq 10$. Here, the threshold function $P_0$ in the definition of the relation $\prec$ depends on the model parameters as well as $\delta$, $\eps_*$ and $\gamma$. 
\end{thm}

\begin{rem} 
The proof of Theorem \ref{thm:local_law_H} actually shows an explicit dependence of the estimates \eqref{eq:local_law_H} and \eqref{eq:averaged_local_law_H} on $\eps_*$. More precisely, 
if the right-hand sides of \eqref{eq:local_law_H} and \eqref{eq:averaged_local_law_H} are multiplied by a universal inverse power of $\eps_*$ and the right-hand side of the 
condition $\Im z \geq p^{-1+\gamma}$ is multiplied by the same inverse power of $\eps_*$ then Theorem \ref{thm:local_law_H} holds true where the 
relation $\prec$ does not depend on $\eps_*$ any more. 
\end{rem}

Let $\mu_1 \leq \ldots \leq \mu_{n+p}$ be the eigenvalues of $H$. We define
\begin{equation}
I (\tau) \defeq \left\lceil (n+p)\int_{-\infty}^{\tau} \rho(\di \omega) \right\rceil, \quad \tau \in \R.
\end{equation}
Thus, $I(\tau)$ denotes the index of an eigenvalue expected to be close to the spectral parameter $\tau \in \R$.

\begin{coro}[Bulk rigidity, Absence of eigenvalues outside of $\supp \rho$] \label{cor:rigidity_H}
Let $\delta, \eps_*>0$.
\begin{enumerate}[(i)]
\item Uniformly for all $\tau \in [-10, -\delta] \cup [\delta,10]$ satisfying $\rho(\tau) \geq \eps_*$ or $\dist(\tau, \supp \rho) \geq \eps_*$, we have
\begin{equation}
\absa{\# \{ j; \mu_j \leq \tau \} - (n+p) \int_{-\infty}^\tau \rho(\di \omega)} \prec 1.
\end{equation}
\item Uniformly for all $\tau \in [-10, -\delta] \cup [\delta,10]$ satisfying $\rho(\tau) \geq \eps_*$, we have
\begin{equation}
\abs{\mu_{I(\tau)} - \tau} \prec \frac{1}{n+p}.
\end{equation}
\item Asymptotically with overwhelming probability, we have 
\begin{equation} \label{eq:no_eigenvalues_away_from_supp_rho}
 \#\Big( \sigma(H) \cap \{ \tau \in [-10,-\delta] \cup [\delta, 10]; \dist(\tau, \supp \rho ) \geq \eps_*\} \Big)  = 0.
\end{equation}
\end{enumerate}
\end{coro}

The estimates \eqref{eq:averaged_local_law_H_bulk} and \eqref{eq:averaged_local_law_H_away_support} in Theorem \ref{thm:local_law_H} imply Corollary \ref{cor:rigidity_H} in the same 
way as the corresponding results, Corollary 1.10 and Corollary 1.11, in \cite{Ajankirandommatrix} were proved.
In fact, inspecting the proofs in \cite{Ajankirandommatrix}, rigidity at a particular point $\tau_0$ in the bulk requires only (i) the local law, \eqref{eq:averaged_local_law_H_bulk}, around $\tau_0=\Re z$,
(ii) the local law somewhere outside of the support of $\rho$, \eqref{eq:averaged_local_law_H_away_support}, and (iii) a
uniform global law with optimal convergence rate, \eqref{eq:averaged_local_law_H_away_support}, for any $z$ away from $\supp \rho$.

\begin{proof}[Proof of Theorem \ref{thm:local_law_H}]
In the proof, we will use the following shorter notation. We introduce the spectral domain
\renewcommand{\DH}{\ensuremath{\mathbb D_{\Xf}}}
\begin{equation*}
\DH \defeq \big\{ z \in \Hb \colon \delta \leq \abs{z} \leq 10,~~ \Im z \geq p^{-1+\gamma}, ~~ \langle \Im \Mf(z) \rangle \geq \eps_* \text{ or } \dist(z, \supp \rho) \geq \eps_* \big\}  
\end{equation*}
for the parameters $\gamma >0, \eps_* >0  $ and $\delta>0$. Moreover, we define the random control parameters 
\begin{equation*}
\Lambda_d(z) \defeq \norm{\gf(z) -\Mf(z)}_\infty, \quad \Lambda_o(z) \defeq \max_{\substack{x,y=1, \ldots, n+p\\x \neq y}} \abs{G_{xy}(z)}, \quad \Lambda(z) \defeq \max\{\Lambda_d(z), \Lambda_o(z)\}.
\end{equation*}

Before proving \eqref{eq:local_law_H} and \eqref{eq:averaged_local_law_H}, we establish the auxiliary estimates: Uniformly for all $z \in \DH$, we have
\begin{subequations}
\begin{align}
\Lambda_d(z) + \norm{\df(z)}_\infty & \prec  \sqrt\frac{\langle \Im \Mf(z) \rangle}{(n+p)\eta} + \frac{1}{(n+p)\eta} + \frac{1}{\sqrt{n+p}} , \label{eq:local_law_H1}\\
\Lambda_o(z) & \prec  \sqrt\frac{\langle \Im \Mf(z) \rangle}{(n+p)\eta} + \frac{1}{(n+p)\eta} + \frac{1}{\sqrt{n+p}} \label{eq:local_law_H2}. 
\end{align}
\end{subequations}
Moreover, for every sequence of vectors $w \in \C^{p+n}$ satisfying $\norm{w}_\infty \leq 1$, 
\begin{equation}
\abs{\scalar{w}{\gf(z)-\Mf(z)}} \prec \frac{\langle \Im \Mf(z) \rangle}{(n+p)\eta} + \frac{1}{(n+p)^2 \eta^2} + \frac{1}{n+p}
\label{eq:local_law_H_averaged}
\end{equation}
uniformly for $z \in \DH$. 

Now, we show that \eqref{eq:local_law_H_averaged} follows from \eqref{eq:local_law_H1} and \eqref{eq:local_law_H2}. To that end, we use the following lemma which is proved as Theorem 3.5 in \cite{Ajankirandommatrix}.
\begin{lem}[Fluctuation Averaging] \label{lem:fluct_averaging}
For any $z \in \DH$ and any sequence of deterministic vectors $w \in \C^{n+p}$ with the uniform bound, $\norm{w}_\infty \leq 1$ the following holds true: If $\Lambda_o(z) \prec \Phi$
for some deterministic ($n$ and $p$-dependent) control parameter $\Phi$ with $\Phi \leq (n+p)^{-\gamma/3}$ and $\Lambda(z) \prec (n+p)^{-\gamma/3}$ a.w.o.p., then 
\begin{equation}
\abs{\langle w, \df(z)\rangle} \prec \Phi^2 + \frac{1}{n+p}.
\end{equation}
\end{lem}

By \eqref{eq:local_law_H1}, the indicator function in \eqref{eq:stability_average_g-m} is nonzero a.w.o.p. Moreover, \eqref{eq:local_law_H2} ensures the applicability of the fluctuation averaging, Lemma 
\ref{lem:fluct_averaging}, which implies that the last term in \eqref{eq:stability_average_g-m} is stochastically dominated by the right-hand side in \eqref{eq:local_law_H_averaged}. 
Using \eqref{eq:local_law_H1} again, we conclude that the first term of the right-hand side of \eqref{eq:stability_average_g-m} is dominated by the right-hand side of \eqref{eq:local_law_H_averaged}.

In order to show \eqref{eq:local_law_H1} and \eqref{eq:local_law_H2} we use the following lemma whose proof we omit, since it follows exactly the same steps as the proof of Lemma 2.1 in \cite{Ajankirandommatrix}.
\begin{lem}
Let $\lambda_* \colon \Hb \to (0,\infty)$ be the function from Lemma \ref{lem:bulk_stability_qve}. 
We have 
\begin{subequations}
\begin{align}
\norm{\df(z)}_\infty \mathbf 1(\Lambda(z) \leq \lambda_*(z)) & \prec \sqrt\frac{\Im \langle \gf(z) \rangle}{(n+p)\eta} + \frac{1}{\sqrt{n+p}}, \label{eq:df_aux_estimate}\\
\Lambda_o(z) \mathbf 1(\Lambda(z) \leq \lambda_*(z)) & \prec \sqrt\frac{\Im \langle \gf(z) \rangle}{(n+p)\eta} + \frac{1}{\sqrt{n+p}} \label{eq:lambda_o_aux_estimate}
\end{align}
\end{subequations}
uniformly for all $z \in \DH$. 
\end{lem}

By \eqref{eq:stability_g-m} and \eqref{eq:df_aux_estimate}, we obtain 
\[ (\Lambda_d(z) + \norm{\df(z)}_\infty) \mathbf 1(\Lambda_d(z) \leq \lambda_*(z)) \prec \sqrt{\frac{\langle \Im \Mf\rangle}{(n+p)\eta}} + (n+p)^{-\eps} \Lambda_d + \frac{(n+p)^{\eps}}{(n+p)\eta} + \frac{1}{\sqrt{n+p}}\]
for any $\eps \in (0,\gamma)$. Here we used $\Im \gf = \Im \Mf + \mathcal O(\Lambda_d)$. 
We absorbe $(n+p)^{-\eps} \Lambda_d$ into the left-hand side and get 
\begin{equation} \label{eq:lambda_d_df_aux_estimate1}
 (\Lambda_d(z) + \norm{\df(z)}_\infty) \mathbf 1(\Lambda_d(z) \leq \lambda_*(z)) \prec \sqrt{\frac{\langle \Im \Mf\rangle}{(n+p)\eta}} + \frac{1}{(n+p)\eta} + \frac{1}{\sqrt{n+p}}
\end{equation}
as $\eps \in (0,\gamma)$ is arbitrary. 
From \eqref{eq:lambda_o_aux_estimate}, we conclude 
\begin{equation} \label{eq:lambda_o_aux_estimate1}
 \Lambda_o(z) \mathbf 1(\Lambda(z) \leq \lambda_*(z)) \prec \sqrt{\frac{\langle \Im \Mf\rangle}{(n+p)\eta}} + \frac{1}{(n+p)\eta} + \frac{1}{\sqrt{n+p}},
\end{equation}
where we used $\Im \gf = \Im \Mf  + \mathcal O(\Lambda_d)$ and \eqref{eq:lambda_d_df_aux_estimate1} and the fact that $\Lambda_d \leq \Lambda$. 

We will conclude the proof by establishing that $\mathbf 1(\Lambda(z) \leq \lambda_*(z)) =1$ a.w.o.p. due to an application of Lemma A.1 in \cite{Ajankirandommatrix}.
Combining \eqref{eq:lambda_d_df_aux_estimate1} and \eqref{eq:lambda_o_aux_estimate1} and using $\langle \Im \Mf(z) \rangle \lesssim (\Im z)^{-1}$, we obtain 
\begin{equation} \label{eq:lambda_aux_estimate}
\Lambda(z) \mathbf 1(\Lambda(z) \leq \lambda_*(z)) \prec (n+p)^{-\gamma/2} 
\end{equation}
for $z \in \DH$ by the definition of $\DH$. 
We define the function $\Phi(z) \defeq (n+p)^{-\gamma/3}$ and note that
$\Lambda(z) = \norm{\gf(z) - \Mf(z)}_\infty$ is Hölder-continuous since $\gf$ and $\Mf$ are Hölder-continuous by
\begin{equation} \label{eq:G_xy_hoelder_continuity}
\max_{x,y=1,\ldots, n+p} \abs{G_{xy}(z_1) - G_{xy}(z_2)} \leq \frac{\abs{z_1-z_2}}{(\Im z_1)(\Im z_2)} \leq (n+p)^2 \abs{z_1-z_2} 
\end{equation}
for $z_1, z_2 \in \DH$ and Lemma \ref{lem:continuity_mf}, respectively. We choose $z_0 \defeq 10 \i$. Since $\abs{G_{xy}(z)}, \abs{\Mf_x(z)} \leq (\Im z)^{-1}$ we get $\Lambda(10\i) \leq 1$ and hence
 $\mathbf 1(\Lambda(10\i) \leq \lambda_*(10\i)) =1$ 
by Lemma \ref{lem:bulk_stability_qve}. Therefore, we conclude $\Lambda(z_0) \leq (n+p)^{-\gamma/2} \leq \Phi(z_0)$ from \eqref{eq:lambda_aux_estimate}. 
Moreover, \eqref{eq:lambda_aux_estimate} implies $\Lambda \cdot \mathbf 1(\Lambda \in [\Phi-(n+p)^{-1}, \Phi]) < \Phi -(n+p)^{-1}$ a.w.o.p. uniformly on $\DH$. 
Thus, we get $\Lambda(z) \leq (n+p)^{-\gamma/3}$ a.w.o.p. for all $z \in \DH$ by applying Lemma A.1 in \cite{Ajankirandommatrix} to $\Lambda$ and $\Phi$ on the connected domain $\DH$, 
i.e., $\mathbf 1(\Lambda(z) \leq \lambda_*(z))=1$ a.w.o.p. 
Therefore, \eqref{eq:lambda_d_df_aux_estimate1} and \eqref{eq:lambda_o_aux_estimate1} yield \eqref{eq:local_law_H1} and \eqref{eq:local_law_H2}, respectively.
As remarked above this also implies \eqref{eq:local_law_H_bulk}.

For the proof of \eqref{eq:local_law_H_away_support} and \eqref{eq:averaged_local_law_H_away_support}, we first notice that 
\[ G_{xx} (z) = \sum_{a=1}^{n+p} \frac{\abs{u_a(x)}^2}{\mu_a - z} \]
for all $x = 1, \ldots, n+p$, where $u_a(x)$ denotes the $x$-component of a $\norm{\cdot}_2$ normalized eigenvector $u_a$ corresponding to the eigenvalue $\mu_a$ of $H$. Therefore, we conclude 
\[ \Im G_{xx}(z) = \eta \sum_{a=1}^{n+p} \frac{\abs{u_a(x)}^2}{(\mu_a-\tau)^2 + \eta^2} \prec \eta \sum_{a=1}^{n+p} \mathbf 1(A_a) \frac{\abs{u_a(x)}^2}{(\mu_a-\tau)^2 + \eta^2}
\prec \eta \]
for all $z \in \Hb$ satisfying $\delta \leq \abs{z} \leq 10$ and $\dist(z, \supp \rho) \geq \eps_*$. Here we used that $A_a \defeq \{ \dist(\mu_a, \supp \rho) \leq \eps_*/2 \}$ occurs a.w.o.p by 
\eqref{eq:no_eigenvalues_away_from_supp_rho} 
and thus $1 - \mathbf 1(A_a) \prec 0$. In particular, we have $\avg{\Im \gf} \prec \eta$. 
Now, \eqref{eq:df_aux_estimate} and \eqref{eq:lambda_o_aux_estimate} yield 
\begin{subequations}
\begin{align}
\norm{\df(z)}_\infty \mathbf 1(\Lambda(z) \leq \lambda_*(z)) & \prec \frac{1}{\sqrt{n+p}}, \label{eq:df_aux_estimate2}\\
\Lambda_o(z) \mathbf 1(\Lambda(z) \leq \lambda_*(z)) & \prec \frac{1}{\sqrt{n+p}}. \label{eq:lambda_o_aux_estimate2}
\end{align}
\end{subequations}
Following the previous argument but
using \eqref{eq:df_aux_estimate2} and \eqref{eq:lambda_o_aux_estimate2} instead of \eqref{eq:df_aux_estimate} and \eqref{eq:lambda_o_aux_estimate}, we obtain \eqref{eq:local_law_H_away_support} and 
\eqref{eq:averaged_local_law_H_away_support} and this completes the proof of Theorem \ref{thm:local_law_H}.
\end{proof}

\subsection{Local law for Gram matrices}

\begin{proof}[Proofs of Theorem \ref{thm:local_law_gram} and Theorem \ref{thm:Bulk_rigidity_general}]
Splitting the resolvent of $\Xf$ at $z \in \C \setminus \R$ into blocks
\[ G(z) = \begin{pmatrix} \mathcal G_{11}(z) & \mathcal G_{12}(z) \\ \mathcal G_{21}(z) & \mathcal G_{22}(z) \end{pmatrix} \]
and computing the product $G(z)(H-z)$ blockwise, we obtain that $(XX^*-z^2)^{-1} = \mathcal G_{11}(z) / z$ and $(X^*X -z^2)^{-1} = \mathcal G_{22}(z)/z$ for $z \in \C\setminus \R$. 
Therefore, \eqref{eq:local_law_XX*} follows from \eqref{eq:local_law_H} as well as $\abs{z} \geq \delta$ and
$m(\zeta) = M_1(\sqrt \zeta)/\sqrt \zeta$ for $\zeta \in \Hb$.

As $p \sim n$ we obtain 
\begin{equation*}
\abs{\scalar{w}{\mathrm{diag}(XX^* - \zeta )^{-1} - m(\zeta)}} \lesssim \absa{\scalara{(w,0)^t}{\frac{1}{\sqrt \zeta}\left(\gf(\sqrt \zeta) - \Mf(\sqrt \zeta)\right)}} 
\end{equation*}
for $w \in \C^p$. Using $p \sim n$, this implies \eqref{eq:local_law_XX*_averaged} by \eqref{eq:averaged_local_law_H}. This concludes the proof of Theorem~\ref{thm:local_law_gram}. 

Theorem \ref{thm:Bulk_rigidity_general} is a consequence of the corresponding result for $\Xf$, namely Corollary~\ref{cor:rigidity_H}. 
\end{proof}

\begin{proof}[Proof of Theorem \ref{thm:eigenvalue_hard_edge}]
As $m(\zeta) = M_1(\sqrt \zeta)/\sqrt \zeta$ for $\zeta \in \Hb$, Proposition \ref{pro:hard_edge_stability} implies $\abs{m(\zeta)} \lesssim \abs{\zeta}^{-1/2}$. Thus, $\poma =0$. 
Recalling $\dens(\omega) = \omega^{-1/2} \rho_1(\omega^{1/2}) \mathbf 1 (\omega >0)$, where $\rho_1$ is the bounded density representing $\avg{M_1}$, yields
\[ \lim_{\omega \downarrow 0} \dens(\omega)\sqrt \omega = \frac{1}{\pi}\avg{v_1(0)} \in (0,\infty) \]
by \eqref{eq:expansion_Mf} which proves part (ii) of Theorem \ref{thm:eigenvalue_hard_edge}.

Since $n=p$, in this case we have $\sigma(XX^*) = \sigma(X^*X)$. Thus, $\avg{g_1} = \avg{g_2}$, i.e., \eqref{eq:assumption_stability_hard_edge_easy} is fulfilled and Proposition \ref{pro:hard_edge_stability} is applicable. 

Using Proposition \ref{pro:hard_edge_stability} instead of Lemma \ref{lem:bulk_stability_qve} and following the argument in Subsection \ref{subsec:local_law_H}, we obtain the same result as Theorem \ref{thm:local_law_H} 
without the restriction $\abs{z} \geq \delta$. As in the proof of Theorem \ref{thm:local_law_gram}, we obtain 
\[ \abs{R_{ij}(\zeta) - \delta_{ij} m_i(\zeta)} \prec \frac{\sqrt{\Re \sqrt{\zeta}}}{\abs{\sqrt{\zeta}} \sqrt{p \Im \zeta}} \lesssim \sqrt{\frac{\avg{\Im m(\zeta)}}{p \Im \zeta}}. \]
Here, we deviated from the proof of Theorem \ref{thm:local_law_gram} since $\abs{z}$ can be arbitrarily small for $z \in \mathbb D_0$ and used part (ii) of Theorem \ref{thm:eigenvalue_hard_edge} in the last step.
This concludes the proof of part (i) of Theorem \ref{thm:eigenvalue_hard_edge}. 

Consequently, a version of Corollary \ref{cor:rigidity_H} for $\delta = 0$ holds true. Then, part (iii) and (iv) of the theorem follow immediately. 
\end{proof}

\subsection{Proof of Theorem \ref{thm:eigenvalue_XX_star_X_star_X}}

In this subsection, we will assume that (A), (C), (D) and \zerorect as well as
\begin{equation} \label{eq:p_bigger_than_n2}
\frac{p}{n} \geq 1 + d_*
\end{equation}
for some $d_*>0$ hold true. 

\begin{thm}[Local law for $\Xf$ around $z=0$] \label{thm:kernel_XX_star}
If (A), (C), (D), \zerorect and \eqref{eq:p_bigger_than_n2} hold true, then 
\begin{enumerate}[(i)]
\item The kernel of $\Xf$ and the kernel of $\Xf^2$ have dimension $p-n$ a.w.o.p. 
\item 
There is a $\gamma_* \gtrsim 1$ such that 
\begin{equation} \label{eq:nonzero_eigenvalues_of_Xf_are_away_from_zero}
 \abs{\mu} \geq \gamma_*
\end{equation}
a.w.o.p. for all $\mu \in \spec(\Xf)$ such that $\mu \neq 0$. 
\item For every $\eps_*>0$, we have 
\begin{subequations}
\begin{align}
\max_{x,y=1, \ldots, n+p} \abs{G_{xy}(z) - \Mf_x(z) \delta_{xy}} & \prec \frac{1}{\abs{z}\sqrt{n+p}}, \label{estimate_entrywise_close_to_zero}\\
\abs{\avg{\gf} -\avg{\Mf}} & \prec \frac{\abs{z}}{n+p}.\label{estimate_average_close_to_zero}
\end{align}
\end{subequations}
uniformly for $z \in \Hb$ satisfying $\abs{z} \leq \sqrt{\delta_\dens} -\eps_*$.
\end{enumerate}
\end{thm}

We will prove that the kernel of $\Xf^2$ has dimension $p-n$ by using a result about the smallest nonzero eigenvalue of $XX^*$ from \cite{Feldheim2010}.
Since this result requires the entries of $X$ to have the same variance and a symmetric distribution, in order to cover the general case, we employ 
a continuity argument which replaces $x_{ik}$, for definiteness, by centered Gaussians with variance $(n+p)^{-1}$. 
This will immediately imply Theorem \ref{thm:kernel_XX_star} and consequently Theorem \ref{thm:eigenvalue_XX_star_X_star_X}.

We recall the definition of $\delta_\dens$ from \eqref{eq:def_delta_dens} and choose $\delta_*$ as in Proposition \ref{pro:QVE_close_to_zero} for the whole section. 
Note that $\delta_*^2 \leq \delta_\dens$.
 
\begin{lem} \label{lem:gap_spectrum_H2}
If \eqref{eq:p_bigger_than_n2} holds true then for all $\delta_1, \delta_2 >0$ such that $\delta_1 < \delta_2 < \delta_*^2/2$, the matrix 
$\Xf^2$ has no eigenvalues in $[\delta_1, \delta_2]$ a.w.o.p. 
\end{lem}

\begin{proof}
Part (iii) of Corollary \ref{cor:rigidity_H} with $\delta = \delta_1$ and $\eps_* = \min\{\delta_1,\delta_\pi-\delta_2\}$ implies 
\[ \# \left(\sigma({\Xf}) \cap [\sqrt \delta_1, \sqrt \delta_2] \right)= 0 \]
a.w.o.p. because there is a gap in the support of $\rho$ by part (iii) of Proposition \ref{pro:QVE_close_to_zero}. Since $\sigma(H^2) = \sigma(H)^2$ this concludes the proof. 
\end{proof} 

For the remainder of the section, let $\wh X=(\wh x_{ik})_{i=1,\ldots,p}^{k=1, \ldots, n}$ consist of independent centered Gaussians with $\E \abs{\wh x_{ik}}^2 = (n+p)^{-1}$. 
We set 
\[ \wh {\Xf} \defeq \begin{pmatrix} 0 & \wh X \\ \wh X^* & 0 \end{pmatrix}. \]

\begin{lem} \label{lem:kernel_XX_star_Gauss}
If \eqref{eq:p_bigger_than_n2} holds true then the kernel of $\wh X \wh X^*$ has dimension $p-n$ a.w.o.p., $\ker( \wh X^* \wh X) = \{0\}$ a.w.o.p.
and there is a $\wh\gamma \sim 1$ such that 
\begin{equation} \label{eq:estimate_wh_lambda}
\wh \lambda \geq \wh \gamma
\end{equation}
 for all $\wh\lambda  \in \sigma( \wh X^* \wh X)$.
\end{lem}

\begin{proof}
Let $\wh \lambda_1 \leq \ldots \leq \wh \lambda_p$ be the eigenvalues of $\wh X \wh X^*$. The assertion will follow once we have established that 
$\wh \lambda_{p-n+1} \gtrsim 1$ a.w.o.p. since $\wh X \wh X^*$ and $\wh X^* \wh X$ have the same nonzero eigenvalues and $\dim \ker \wh X \wh X^* \geq p-n$ for dimensional reasons. 
Corollary V.2.1 in \cite{Feldheim2010} implies that $\wh \lambda_{p-n+1} \geq \gamma_- - p^{-2/3+\eps}$ a.w.o.p. for each $\eps >0$  
where $\gamma_- \defeq 1- 2\sqrt{pn}/(n+p) \gtrsim 1$, thus $\wh \lambda_{p-n+1} \gtrsim 1$ a.w.o.p. 
 In fact, our proof only requires 
that  $\wh \lambda_{p-n+1} \geq \gamma_- -\eps$ 
for any $\eps>0$ a.w.o.p, which already follows from the argument in  \cite{silverstein1985}.  
\end{proof}

\begin{proof}[Proof of Theorem \ref{thm:kernel_XX_star}]
We define ${\Xf}_t \defeq \sqrt{1-t} {\Xf} + \sqrt t \wh {\Xf}$ for $t \in [0,1]$ and set $\gamma_* \defeq \min\{\delta_*/2, \sqrt{\wh\gamma}\}$, where $\wh \gamma$ is chosen as in 
\eqref{eq:estimate_wh_lambda}.  
By Lemma \ref{lem:gap_spectrum_H2} with $\delta_2 \defeq \gamma_*^2$ and $\delta_1 \defeq \gamma_*^2/2$, ${\Xf}_t^2$ has no eigenvalues in $[\delta_1, \delta_2]$ a.w.o.p. for every $t \in [0,1]$. 
Clearly, the eigenvalues of ${\Xf}_t^2$ depend continuously on $t$. 
Therefore, $\# (\sigma({\Xf}^2) \cap [0,\delta_1))  = \# (\sigma(\wh {\Xf}^2) \cap [0,\delta_1))$. Thus, we get the chain of inequalities 
\[ p-n \leq \dim \ker {\Xf} = \dim \ker {\Xf}^2 \leq \# \left(\sigma({\Xf}^2) \cap [0,\delta_1)\right) = \#\left( \sigma(\wh {\Xf}^2) \cap [0,\delta_1) \right)= \dim \ker \wh {\Xf}^2 = p-n. \]
Here we used Lemma \ref{lem:kernel_XX_star_Gauss} in the last step. As the left and the right-hand-side are equal all of the inequalities are equalities which concludes the proof of part (i)
and part (ii). 

We will omit the proof of part (iii) of Theorem \ref{thm:kernel_XX_star} as it is very similar to the proof of part (vi) of Theorem \ref{thm:eigenvalue_XX_star_X_star_X}  below 
which will be independent of part (iii) of Theorem \ref{thm:kernel_XX_star}.
\end{proof}

\begin{proof}[Proof of Theorem \ref{thm:eigenvalue_XX_star_X_star_X}]
Since $\delta_*$ is chosen as in Proposition \ref{pro:QVE_close_to_zero} we conclude $\delta_\dens \geq \delta_*^2 \gtrsim 1$ from part (iv) of this proposition. Part (ii) and (iii)  
of the theorem follow immediately from \eqref{eq:nonzero_eigenvalues_of_Xf_are_away_from_zero} in Theorem \ref{thm:kernel_XX_star}. 

If $p>n$, then $\dim \ker XX^*=p-n$ a.w.o.p. as $p-n \leq \dim \ker XX^* \leq \dim \ker H^2 = p-n$ a.w.o.p by part (i) of Theorem \ref{thm:kernel_XX_star}.
By Proposition \ref{pro:QVE_close_to_zero}, we obtain $\poma = \avg{u} = 1- n/p$, where $u$ is defined as in this proposition. This proves part (iv). 
If $p<n$, then part (v) follows from interchanging the roles of $X$ and $X^*$ and following the same steps as in the proof of part (iv). 

For the proof of part (vi), we first assume $p>n$. By Proposition \ref{pro:QVE_close_to_zero} we can uniquely extend $\zeta m(\zeta) = \sqrt \zeta M_1(\sqrt \zeta)$ 
to a holomorphic function on $D_{\delta_*^2}(0)$. We fix $\gamma_*$ as in \eqref{eq:nonzero_eigenvalues_of_Xf_are_away_from_zero}. 
On the event $\{ \lambda_i \geq \gamma_*^2 \text{ for all }i=p-n+1, \ldots, p \}$, which holds true a.w.o.p. by \eqref{eq:nonzero_eigenvalues_of_Xf_are_away_from_zero}, 
the function $\zeta R(\zeta)$ can be uniquely extended to a holomorphic function on $D_{\gamma_*^2}(0)$. We set $\delta \defeq \min\{ \gamma_*^2/2, \delta_*^2\}$ 
and assume without loss of generality that $\delta \leq \delta_\dens - \eps_*$. For $\zeta \in \Hb$ satisfying 
$\delta \leq \abs{\zeta} \leq \delta_\dens - \eps_*$, \eqref{eq:local_law_around_zero1} is immediate from \eqref{eq:local_law_XX_star_away}. We apply \eqref{eq:local_law_XX_star_away} to obtain 
$\max_{i,j} \abs{R_{ij}(\zeta) - m_i(\zeta) \delta_{ij}} \prec 1/p$ for $\zeta \in \Hb$ satisfying $\abs{\zeta} = \delta$. By the symmetry of $R(\zeta)$ and $m(\zeta)$ this estimate holds true for all $\zeta \in \C$ satisfying
$\abs{\zeta}= \delta$. Thus, the maximum principle implies that $\max_{i,j} \abs{\zeta R_{ij}(\zeta) - \zeta m_i (\zeta)\delta_{ij}} \prec 1/p$ which proves \eqref{eq:local_law_around_zero1} since 
$\{\lambda_i \geq 2\delta \text{ for all }i=p-n+1, \ldots, p \}$ which holds true a.w.o.p. by $2\delta \leq \gamma_*^2$ and \eqref{eq:nonzero_eigenvalues_of_Xf_are_away_from_zero}.   
If $p<n$ then $XX^*$ does not have a kernel a.w.o.p. by  (v).  
Therefore, a similar argument yields \eqref{eq:local_law_around_zero2}.

For the proof of \eqref{eq:local_law_around_zero3}, we observe that $\dim \ker (XX^*) = p \poma$ a.w.o.p. in both cases by (iv) and (v). Thus, 
\[ \frac{1}{p} \sum_{i=1}^p [R_{ii}(\zeta) - m_i(\zeta)] = \frac{1}{p} \left( \sum_{j\colon \lambda_j \geq \gamma_*^2} \frac{1}{\lambda_j - \zeta} - \sum_{i=1}^p a_i(\sqrt \zeta) \right) \]
a.w.o.p. for $\zeta \in D_\delta(0)$, $\delta$ chosen as above, 
by \eqref{eq:nonzero_eigenvalues_of_Xf_are_away_from_zero}, where $a$ is the holomorphic function on $D_{\delta_*}(0)$ defined in Proposition \ref{pro:QVE_close_to_zero}. 
The right-hand side of the previous equation can therefore be uniquely extended to 
a holomorphic function on $D_{\delta_*}(0)$. As before, the estimate \eqref{eq:local_law_XX_star_away} can be extended to $\zeta \in \Hb$ with $\abs{\zeta} \leq \delta$ by the maximum principle.
\end{proof}

The local law for $\zeta$ around zero needed an extra argument, Theorem \ref{thm:eigenvalue_XX_star_X_star_X}, due to the possible singularity at $\zeta =0$. We note that this separate treatment 
is necessary even if $p<n$, in which case $XX^*$ does not have a kernel and $R(\zeta)$ is regular at $\zeta =0$, since we study $XX^*$ and $X^*X$ simultaneously. Our main stability results are formulated and proved in 
terms of $H$, as defined in \eqref{eq:def_H_S}. Therefore, these results are not sensitive to whether $p$ or $n$ is bigger which means whether $XX^*$ has a kernel or $X^*X$. 

\appendix
\section{Appendix: Proof of the Rotation-Inversion lemma}

In this appendix, we prove the Rotation-Inversion lemma, Lemma \ref{lem:bulk_stability}.

\begin{proof}[Proof of Lemma \ref{lem:bulk_stability}]
In this proof, we will write $\norm{A}$ to denote $\normtwo{A}$.  Moreover, we introduce a few short hand notations,
\[
\cal{U}\,:=\, 
\left(
\begin{array}{cc}
U_1	&	0
\\
0	&	U_2
\end{array}
\right)\,,
\qquad
\cal{A}\,:=\, 
\left(
\begin{array}{cc}
0	&	A
\\
A^*	&	0
\end{array}
\right)\,,
\qquad
a_{\pm}\,:=\, \frac{1}{\sqrt{2}}
\left(
\begin{array}{c}
v_1
\\
\pm v_2 	
\end{array}
\right)\,,\qquad
\rho\,:=\, \norm{A^*A}^{1/2}.
\]
In particular,  we have $Av_2=\rho \2\ee^{\ii \psi}v_1$ and $A^*v_1=\rho\2\ee^{-\ii \psi} v_2$ for some $\psi \in \R$. By redefining $v_1$ to be $\ee^{\ii \psi}v_1$ we may assume that $\psi=0$ and get $\cal{A} a_{\pm}= \pm \rho\2a_{\pm}$ as well. 

Let us check that indeed $\cal{U}+\cal{A}$ is not invertible if the right hand side of \eqref{bound on block inverse} is infinite, i.e., if 
\[
\norm{A^{*}A}\scalar{v_1}{U_1v_1}\scalar{v_2}{U_2v_2}\,=\, 1\,.
\]
In this case we find $\norm{A^{*}A}=1$, $\scalar{v_1}{U_1v_1}=\ee^{\ii \1\varphi}$ and $\scalar{v_2}{U_2v_2}=\ee^{-\ii \1\varphi}$ for some $\varphi \in \R$. Thus, $v_1$ and $v_2$ are eigenvectors of $U_1$ and $U_2$, respectively. Therefore, both $\cal{U}$ and $\cal{A}$ leave the $2$-dimensional subspace spanned by $(v_1,0)$ and $(0,v_2)$ invariant and in this basis the restriction of $\cal{U}+\cal{A}$ is represented by the $2\times 2$-matrix
\[
\left(
\begin{array}{cc}
\ee^{\ii\1\varphi}	&	1
\\
1	&	\ee^{-\ii\1\varphi}
\end{array}
\right),
\]
which is not invertible. 

We will now show that in every other case $\cal{U}+\cal{A}$ is invertible and its inverse satisfies \eqref{bound on block inverse}.
To this end we will derive a lower bound on $\norm{(\cal{U}+\cal{A})w}$ for an arbitrary normalized vector $w \in \C^{n+p}$. Any such vector admits a decomposition,
\[
w\,=\, \alpha_+\1a_+ + \alpha_- \1a_- + \beta \1 b\,,
\]
where $\alpha_\pm \in \C$, $\beta \geq 0$ and $b$ is a normalized vector in the orthogonal complement of the $2$-dimensional space spanned by $a_+$ and $a_-$. The normalization of $w$ implies
\bels{w normalization}{
|\alpha_+|^2+|\alpha_-|^2+\beta^2\,=\, 1\,.
}
The case $\beta=1$ is trivial because the spectral gap of $A^*A$ implies a spectral gap of $\cal{A}$ in the sense that
\bels{spectral gap of cal A}{
\spec(\cal{A}/\rho)\,\subseteq\,\{-1\}\cup \big[-1 +\rho^{-2}\1\rr{Gap}(AA^*),1 -\rho^{-2}\1\rr{Gap}(AA^*)\2\big]\cup\{1\}\,.
}
Thus, we will from now on assume $\beta<1$. 

We will use the notations $\cal{P}_\parallel$ and $\cal{P}_\perp$ for the orthogonal projection onto the $2$-dimensional subspace spanned by $ a_\pm$ and its orthogonal complement, respectively. We also introduce
\bels{definition of lambda and kappa}{
\lambda\,:=\, \frac{1}{2}\frac{ \abs{\alpha_++\alpha_-}^2}{\abs{\alpha_+}^2+\abs{\alpha_-}^2} \in [0,1]\,,
\qquad 
\kappa\,:=\, 
(\abs{\alpha_+}^2+\abs{\alpha_-}^2)^{-1/2}\norm{\cal{P}_\parallel(1+\cal{U}^*\cal{A})(\alpha_+ a_+ + \alpha_- a_-)}\,.
}
With this notation we will now prove
\bels{Lower bound on norm of cal U+ cal A w}{
\norm{(\cal{U}+\cal{A})w}\,\geq\, c_1\2\rr{Gap}(AA^*)\2\kappa\,,
}
for some positive numerical constant $c_1$. The analysis is split into the following regimes:
\begin{description}
\item [Regime 1:] $  \kappa^{1/2} < 10\beta$,
\item [Regime 2:] $\kappa^{1/2}\geq 10\beta  $ and $\lambda < 1/10$,
\item  [Regime 3:] $\kappa^{1/2}\geq 10\beta $ and $\lambda > 9/10$,
\item [Regime 4:] $\kappa^{1/2}\geq 10\beta $ and $1/10\leq \lambda \leq 9/10$ and $\abs{\scalar{v_1}{U_1v_1}}^2+\abs{\scalar{v_2}{U_2v_2}}^2\leq 2-\kappa/2$,
\item [Regime 5:] $\kappa^{1/2}\geq 10\beta $ and $1/10\leq \lambda \leq 9/10$ and $\abs{\scalar{v_1}{U_1v_1}}^2+\abs{\scalar{v_2}{U_2v_2}}^2> 2-\kappa/2$.
\end{description}
These regimes can be chosen more carefully in order to optimize the constant $c_1$ in \eqref{Lower bound on norm of cal U+ cal A w}, but we will not do that here.  

\medskip
\noindent{\textbf{Regime 1: }} In this regime we make use of the spectral gap of $A^*A$ by simply using the triangle inequality,
\[
\norm{(\cal{U}+\cal{A})w}\,\geq\, \norm{w}- \norm{\cal{A}w}\,=\, 1- \sqrt{\1\rho^2\1\abs{\alpha_+}^2 +\1\rho^2\1\abs{\alpha_-}^2+\beta^2\norm{\cal{A}b}^2}.
\]
We use the inequality $1-\sqrt{1-\tau}\geq \tau/2$ for $\tau \in [0,1]$ as well as the normalization \eqref{w normalization} and find
\bes{
2\1\norm{(\cal{U}+\cal{A})w}\,\geq\, 1-\1\rho^2\1 +\1\rho^2\1\beta^2-\beta^2\norm{\cal{A}b}^2\,\geq\, \rho\2\beta^2(\rho-\norm{\cal{A}b})
\,\geq\, \beta^2\rr{Gap}(AA^*)\,.
}
The last inequality follows from \eqref{spectral gap of cal A} and because $b$ is orthogonal to $a_\pm$. Since $\beta^2\geq \kappa/100 $, we conclude that in the first regime \eqref{Lower bound on norm of cal U+ cal A w} is satisfied.

\medskip
\noindent{\textbf{Regime 2: }} In this regime we project on the second component of $(\cal{U}+\cal{A})w$. 
\bes{
\sqrt{2}\norm{(\cal{U}+\cal{A})w}
\,&\geq\,
\norm{(\alpha_+-\alpha_-) U_2 v_2+\sqrt{2} \1\beta\1 U_2 b_2-(\alpha_++\alpha_-)A^* v_1-\sqrt{2} \1\beta \1A^* b_1 }
\\
\,&\geq\,
\abs{\alpha_+-\alpha_-}\norm{ U_2 v_2}-\sqrt{2}\1\beta \norm{U_2 b_2}-\rho\2\abs{\alpha_++\alpha_-}\norm{v_2}-\sqrt{2}\1\beta\norm{A^*b_1}
\\
\,&\geq\,
\sqrt{2}\1\sqrt{\abs{\alpha_+}^2+\abs{\alpha_-}^2}(\sqrt{1-\lambda}-\sqrt{\lambda})- 2\sqrt{2}\1\beta\,.
}
Here we used the notation $b=(b_1,b_2)$ for the components of $b$. The last inequality holds by the normalization of $v_2$ and $b$, by $\rho \leq 1$ and by the definition of $\lambda$ from \eqref{definition of lambda and kappa}, which also implies
\[
\abs{\alpha_+-\alpha_-}^2\,=\, 2(1-\lambda)(\abs{\alpha_+}^2+\abs{\alpha_-}^2)\,.
\]
Since $\lambda< 1/4$  in this regime and $\kappa\leq 2$ by the definition of $\kappa$ in \eqref{definition of lambda and kappa}  we find  $\beta \leq \kappa^{1/2}/10 \leq 1/5$ and infer
\bes{
\norm{(\cal{U}+\cal{A})w}
\,&\geq\,\sqrt{1-\beta^2}(\sqrt{1-\lambda}-\sqrt{\lambda})- 2\1\beta\,\geq\, 1/10\,\geq\, \kappa/20\,.
}

\medskip
\noindent{\textbf{Regime 3: }}  By the symmetry in $a_\pm$ and $\alpha_\pm$ and therefore in $\lambda$ and $1-\lambda$ this regime is treated in the same way as Regime 2 by estimating the norm of the first component of $(\cal{U}+\cal{A})w$ from below.

\medskip
\noindent{\textbf{Regime 4: }} Here we project onto the orthogonal complement of the subspace spanned by $a_+$ and $a_-$,
\bels{start regime 4}{
\norm{(\cal{U}+\cal{A})w}
\,&\geq\, 
\norm{\cal{P}_\perp(\cal{U}+\cal{A})w}
\,\geq\,  
\norm{\cal{P}_\perp\1\cal{U}(\alpha_+ \1a_+ + \alpha_- \1a_-)}-\beta \1\norm{\cal{P}_\perp(\cal{U}+\cal{A})b}
\,.
}
We compute the first term in this last expression more explicitly,
\bels{to minimize}{
\quad\msp{5}\norm{\cal{P}_\perp\1\cal{U}(\alpha_+ \1a_+ + \alpha_- \1a_-)}^2
\,&=\, 
\norm{\alpha_+ \1a_+ + \alpha_- \1a_-}^2-\norm{\cal{P}_\parallel\2\cal{U}(\alpha_+ \1a_+ + \alpha_- \1a_-)}^2
\\
\,&=\,
\abs{\alpha_+}^2+\abs{\alpha_-}^2- \frac{1}{2}\1\abs{\alpha_++\alpha_-}^2\abs{\scalar{v_1}{U_1 v_1}}^2 - \frac{1}{2}\1\abs{\alpha_+-\alpha_-}^2\abs{\scalar{v_2}{U_2 v_2}}^2 
\\
\,&=\,
(1-\beta^2)
\pb{
1- \lambda\abs{\scalar{v_1}{U_1 v_1}}^2 -(1-\lambda)\abs{\scalar{v_2}{U_2 v_2}}^2 
}\,.
}
For the second equality we used that
\[
\norm{\cal{P}_\parallel u}^2 \,=\, \abs{\scalar{v_1}{u_1}}^2+ \abs{\scalar{v_2}{u_2}}^2, \qquad u\2=\2(u_1,u_2) \in \C^{p+n}.
\]
With the choice of variables 
\[
\xi\,:=\, \abs{\scalar{v_1}{U_1 v_1}}^2\,,\qquad \eta \,:=\, \abs{\scalar{v_2}{U_2 v_2}}^2 \,,
\]
we are minimizing the last line in \eqref{to minimize} under the restrictions that are satisfied in this regime, 
\bes{
\min\{1-\lambda\1\xi-(1-\lambda)\1\eta:\;\xi,\eta \in [0,1]\,,\; 2\1\xi+2\2\eta \leq 4-\kappa\}
\,\geq\,\frac{1}{2}\2\kappa\1\min\{1-\lambda,\lambda\}
\,.
}
We use the resulting estimate in \eqref{start regime 4} and 
in this way we arrive at
\bes{
\norm{(\cal{U}+\cal{A})w}
\,&\geq\,\frac{1}{\sqrt{2}}\2 \kappa^{1/2}\sqrt{1-\beta^2}\2\min\{1-\lambda,\lambda\}^{1/2}-2\1\beta
\,\geq\, \frac{\kappa^{1/2}}{100}\,\geq\, \frac{\kappa}{200}
\,.
}
In the second to last inequality we used $\beta\leq 1/5$ which was already established in the consideration of Regime 2 and in the last inequality we used $\kappa\leq 2$. 

\medskip
\noindent{\textbf{Regime 5: }} 
In this regime we project onto the span of $a_+$ and $a_-$,
\bels{start regime 5}{
\norm{(\cal{U}+\cal{A})w}
\,&=\, 
\norm{(1+\cal{U}^*\cal{A})w}
\\
\,&\geq\, \norm{\cal{P}_\parallel(1+\cal{U}^*\cal{A})(\alpha_+ \1a_+ + \alpha_- \1a_-)}
-\beta\1 \norm{\cal{P}_\parallel(1+\cal{U}^*\cal{A})b}
\\
\,&=\, \sqrt{\abs{\alpha_+}^2+\abs{\alpha_-}^2}\,\kappa
-\beta\2\norm{\cal{P}_\parallel\2\cal{U}^*\cal{A}b}\,.
}
The second term in the last line is estimated by using
\bes{
\norm{\cal{P}_\parallel\2\cal{U}^*\cal{A}b}^2\,&\leq\, \norm{\cal{A}b}\sup_{h\parallel a_\pm}\sup_{u\perp a_\pm}\abs{\scalar{h}{\cal{U}^*u}}^2\,,
}
where the suprema are taken over normalized vectors $h$ and $u$ in the $2$-dimensional subspace spanned by $a_\pm$ and its orthogonal complement, respectively. First we perform the supremum over $h$ and get
\bels{regime 5 bound}{
\norm{\cal{P}_\parallel\2\cal{U}^*\cal{A}b}^2\,&\leq\,\sup_{u\perp a_\pm}\pb{\abs{\scalar{v_1}{U_1^*u_1}}^2+\abs{\scalar{v_2}{U_2^*u_2}}^2}
\,\leq\, \sup_{u_1 \perp v_1}\abs{\scalar{v_1}{U_1^*u_1}}^2 + \sup_{u_2 \perp v_2}\abs{\scalar{v_2}{U_2^*u_2}}^2
,
}
where the vectors $u_1\in \C^p$ and $u_2\in \C^n$ are normalized. 
Computing
\[
\sup_{u_1 \perp v_1}\abs{\scalar{v_1}{U_1^*u_1}}^2\,=\, 1- \abs{\scalar{v_1}{U_1v_1}}^2\,,\qquad 
\sup_{u_2 \perp v_2}\abs{\scalar{v_2}{U_2^*u_2}}^2\,=\, 1- \abs{\scalar{v_2}{U_2v_2}}^2\,,
\]
we get
\[
\norm{\cal{P}_\parallel\2\cal{U}^*\cal{A}b}^2\,\leq\,2-\abs{\scalar{v_1}{U_1v_1}}^2- \abs{\scalar{v_2}{U_2v_2}}^2\,\leq\,  \kappa/2\,,
\]
where we used the choice of Regime 5 in the last step.
Plugging this bound into \eqref{start regime 5} and using $\beta \leq \kappa^{1/2}/10$ as well as $\beta \leq 1/5$ yields
\[
\norm{(\cal{U}+\cal{A})w}
\,\geq\, 
\sqrt{1-\beta^2}\,\kappa
-\beta\2\kappa^{1/2}\,\geq\, \kappa/2\,.
\]

\medskip
This finishes the proof of \eqref{Lower bound on norm of cal U+ cal A w}. In order to show \eqref{bound on block inverse}, and thus the lemma, we notice that
\[
\kappa\,\geq\, \inf_{u \parallel a_\pm}\norm{\cal{P}_\parallel(1+\cal{U}^*\cal{A})u} \,,
\]
where the infimum is taken over normalized vectors $u$ in the span of $a_+$ and $a_-$. Thus, it suffices to estimate the norm of the inverse of $\cal{P}_\parallel(1+\cal{U}^*\cal{A})\cal{P}_\parallel$, restricted to the $2$-dimensional subspace with orthonormal basis $(v_1,0)$ and $(0,v_2)$. In this basis this linear operator takes the form of the simple $2\times2$-matrix,
\[
\left(
\begin{array}{cc}
1	&	\rho\1\scalar{v_1}{U_1v_1}
\\
\rho\1\scalar{v_2}{U_2v_2}	&	1
\end{array}
\right).
\]
Its inverse is bounded by the right hand side of \eqref{bound on block inverse}, up to the factor $\Gap(AA^*)$ that we encountered in \eqref{Lower bound on norm of cal U+ cal A w}, and the lemma is proven.
\end{proof}

\bibliographystyle{amsplain}
\bibliography{literature}

\providecommand{\bysame}{\leavevmode\hbox to3em{\hrulefill}\thinspace}
\providecommand{\MR}{\relax\ifhmode\unskip\space\fi MR }
\providecommand{\MRhref}[2]{%
  \href{http://www.ams.org/mathscinet-getitem?mr=#1}{#2}
}
\providecommand{\href}[2]{#2}
\begin{thebibliography}{10}

\bibitem{ECP3121}
O.~Ajanki, L.~Erd\H{o}s, and T.~Kr{\"u}ger, \emph{Local semicircle law with
  imprimitive variance matrix}, Elect. Comm. Probab. \textbf{19} (2014), no.
  33, 1--9.

\bibitem{AjankiQVE}
\bysame, \emph{Quadratic vector equations on complex upper half-plane},
  arXiv:1506.05095v4, 2015.

\bibitem{AjankiCPAM}
\bysame, \emph{Singularities of solutions to quadratic vector equations on the
  complex upper half-plane}, Comm. Pure Appl. Math. (2016),
  doi:10.1002/cpa.21639 (Online).

\bibitem{Ajankirandommatrix}
\bysame, \emph{Universality for general {W}igner-type matrices}, Prob. Theor.
  Rel. Fields (2016), doi:10.1007/s00440-016-0740-2 (Online).

\bibitem{bao2012}
Z.~Bao, G.~Pan, and W.~Zhou, \emph{Tracy-{W}idom law for the extreme
  eigenvalues of sample correlation matrices}, Elect. J. Probab. \textbf{17}
  (2012), 32 pp.

\bibitem{BapatRaghavan}
R.~B. Bapat and T.~E.~S. Raghavan, \emph{Nonnegative matrices and
  applications}, Encyclopedia of Mathematics and its Applications, Cambridge
  University Press, 1997.

\bibitem{GramMatrix}
F.~Benaych-Georges and R.~Couillet, \emph{Spectral analysis of the {G}ram
  matrix of mixture models}, ESAIM: PS \textbf{20} (2016), 217--237.

\bibitem{Berezin1973}
F.~A. Berezin, \emph{{Some remarks on Wigner distribution}}, Theoret. Math.
  Phys. \textbf{3} (1973), no.~17, 1163--1175.

\bibitem{EJP3054}
A.~Bloemendal, L.~Erd\H{o}s, A.~Knowles, H.-T. Yau, and J.~Yin, \emph{Isotropic
  local laws for sample covariance and generalized {W}igner matrices}, Elect.
  J. Probab. \textbf{19} (2014), no. 33, 1--53.

\bibitem{Bourgarde2014}
P.~Bourgade, H.-T. Yau, and J.~Yin, \emph{Local circular law for random
  matrices}, Prob. Theor. Rel. Fields \textbf{159} (2014), no.~3-4, 545--595.

\bibitem{DOZIER2007678}
R.~Brent~Dozier and J.~W. Silverstein, \emph{On the empirical distribution of
  eigenvalues of large dimensional information-plus-noise-type matrices}, Jour.
  Mult. Anal. \textbf{98} (2007), no.~4, 678 -- 694.

\bibitem{Cacciapuoti2013}
C.~Cacciapuoti, A.~Maltsev, and B.~Schlein, \emph{Local {M}archenko-{P}astur
  law at the hard edge of sample covariance matrices}, J. Math. Phys.
  \textbf{54} (2013), no.~4.

\bibitem{wirelesscommunication}
R.~Couillet and M.~Debbah, \emph{Random matrix methods for wireless
  communications}, Cambridge University Press, 2011.

\bibitem{couillet2014}
R.~Couillet and W.~Hachem, \emph{Analysis of the limiting spectral measure of
  large random matrices of the separable covariance type}, Random Matrices:
  Theory and Applications \textbf{03} (2014), no.~04, 1450016.

\bibitem{ErdosSchleinYau2011}
L.~Erd\H{o}s, B.~Schlein, and H.-T. Yau, \emph{Universality of random matrices
  and local relaxation flow}, Invent. math. \textbf{185} (2011), no.~1,
  75--119.

\bibitem{erdoes_relaxation_flow_2012}
L.~Erd\H{o}s, B.~Schlein, H.-T. Yau, and J.~Yin, \emph{The local relaxation
  flow approach to universality of the local statistics for random matrices},
  Ann. Inst. H. Poincaré Probab. Statist. \textbf{48} (2012), no.~1, 1--46.

\bibitem{ErdoesYau2012}
L.~Erd\H{o}s and H.-T. Yau, \emph{Universality of local spectral statistics of
  random matrices}, Bull. Amer. Math. Soc. \textbf{49} (2012), no.~3, 377--414.

\bibitem{EYYbulk}
L.~Erd{\H o}s, H.-T. Yau, and J.~Yin, \emph{{Bulk universality for generalized
  Wigner matrices}}, Prob. Theor. Rel. Fields \textbf{154} (2011), no.~1-2,
  341--407.

\bibitem{EYYBern}
\bysame, \emph{{Universality for generalized Wigner matrices with Bernoulli
  distribution}}, J. Comb. \textbf{2} (2011), no.~1, 15--82.

\bibitem{EYYrig}
\bysame, \emph{{Rigidity of eigenvalues of generalized Wigner matrices}}, Adv.
  Math. \textbf{229} (2012), no.~3, 1435--1515.

\bibitem{Feldheim2010}
O.~N. Feldheim and S.~Sodin, \emph{A universality result for the smallest
  eigenvalues of certain sample covariance matrices}, Geom. Func. Anal.
  \textbf{20} (2010), no.~1, 88--123.

\bibitem{BLTJ:BLTJ2015}
G.~J. Foschini, \emph{Layered space-time architecture for wireless
  communication in a fading environment when using multi-element antennas},
  Bell Labs Tech. J. \textbf{1} (1996), no.~2, 41--59.

\bibitem{Foschini1998}
G.J. Foschini and M.J. Gans, \emph{On limits of wireless communications in a
  fading environment when using multiple antennas}, Wireless Per. Commun.
  \textbf{6} (1998), no.~3, 311--335.

\bibitem{Girko2001}
V.~L. Girko, \emph{{Theory of stochastic canonical equations. Vol. I}},
  Mathematics and its Applications, vol. 535, Kluwer Academic Publishers,
  Dordrecht, 2001.

\bibitem{HachemCLT}
W.~Hachem, M.~Kharouf, J.~Najim, and J.~W. Silverstein, \emph{A {CLT} for
  information-theoretic statistics of non-centered {G}ram random matrices},
  Random Matrices: Theory Appl. \textbf{01} (2012), no.~02, 1150010.

\bibitem{Hachem2008IEEE}
W.~Hachem, O.~Khorunzhiy, P.~Loubaton, J.~Najim, and L.~Pastur, \emph{A new
  approach for mutual information analysis of large dimensional multi-antenna
  channels}, IEEE Trans. Inf. Theory \textbf{54} (2008), no.~9, 3987--4004.

\bibitem{Hachem2006649}
W.~Hachem, P.~Loubaton, and J.~Najim, \emph{The empirical distribution of the
  eigenvalues of a {G}ram matrix with a given variance profile}, Ann. Inst. H.
  Poincaré Probab. Statist. \textbf{42} (2006), no.~6, 649 -- 670.

\bibitem{hachem2007}
\bysame, \emph{Deterministic equivalents for certain functionals of large
  random matrices}, Ann. Appl. Probab. \textbf{17} (2007), no.~3, 875--930.

\bibitem{hachem2008}
\bysame, \emph{A {CLT} for information-theoretic statistics of {G}ram random
  matrices with a given variance profile}, Ann. Appl. Probab. \textbf{18}
  (2008), no.~6, 2071--2130.

\bibitem{KhorunzhyPastur1994}
A.~M. Khorunzhy and L.~A. Pastur, \emph{{On the eigenvalue distribution of the
  deformed Wigner ensemble of random matrices}}, Spectral operator theory and
  related topics, Adv. Soviet Math., 19, Amer. Math. Soc., Providence, RI,
  1994, pp.~97--127.

\bibitem{Knowles2014}
A.~Knowles and J.~Yin, \emph{Anisotropic local laws for random matrices}, Prob.
  Theor. Rel. Fields (2016), doi:10.1007/s00440-016-0730-4 (Online).

\bibitem{MarcenkoPastur1967}
V.~A. Marchenko and L.~A. Pastur, \emph{Distribution of eigenvalues for some
  sets of random matrices}, Mat. Sbornik \textbf{1} (1967), no.~4, 457--483.

\bibitem{mehta2004random}
M.~L. Mehta, \emph{Random matrices}, Pure and Applied Mathematics, Elsevier
  Science, 2004.

\bibitem{pillai2014}
N.~S. Pillai and J.~Yin, \emph{Universality of covariance matrices}, Ann. Appl.
  Probab. \textbf{24} (2014), no.~3, 935--1001.

\bibitem{silverstein1985}
J.~W. Silverstein, \emph{The smallest eigenvalue of a large dimensional
  {W}ishart matrix}, Ann. Probab. \textbf{13} (1985), no.~4, 1364--1368.

\bibitem{SILVERSTEIN1995175}
J.~W. Silverstein and Z.~D. Bai, \emph{On the empirical distribution of
  eigenvalues of a class of large dimensional random matrices}, Jour. Mult.
  Anal. \textbf{54} (1995), no.~2, 175 -- 192.

\bibitem{TaoVu2011_Acta}
T.~Tao and V.~Vu, \emph{Random matrices: Universality of local eigenvalue
  statistics}, Acta Math. \textbf{206} (2011), no.~1, 127--204.

\bibitem{tao2012}
\bysame, \emph{Random covariance matrices: Universality of local statistics of
  eigenvalues}, Ann. Probab. \textbf{40} (2012), no.~3, 1285--1315.

\bibitem{ETT:ETT4460100604}
E.~Telatar, \emph{Capacity of multi-antenna {G}aussian channels}, Eur. Trans.
  Telecomm. \textbf{10} (1999), no.~6, 585--595.

\bibitem{TulinoVerdu}
A.~M. Tulino and S.~Verdú, \emph{Random matrix theory and wireless
  communications}, Found. Trends Commun. Inf. Theory \textbf{1} (2004), no.~1,
  1--182.

\bibitem{Wegner1979}
F.~J. Wegner, \emph{{Disordered system with $ n $ orbitals per site: $ n
  =\infty $ limit}}, Physical Review B \textbf{19} (1979).

\bibitem{Wigner1955}
E.~P. Wigner, \emph{Characteristic vectors of bordered matrices with infinite
  dimensions}, Ann. of Math. \textbf{62} (1955), no.~3, 548--564.

\bibitem{WISHART1928}
J.~Wishart, \emph{The generalised product moment distribution in samples from a
  normal multivariate population}, Biometrika \textbf{20A} (1928), no.~1-2,
  32--52.

\end{thebibliography}

\end{document}